\documentclass[11pt]{amsproc}
\pdfoutput=1
\usepackage{amsmath,amssymb,amsthm,eucal, eufrak,amscd,tikz,mathtools}
\usepackage{xcolor}
\usepackage{enumitem}
\usepackage[shortalphabetic]{amsrefs}
\usepackage{tikz-cd}
\usepackage{graphicx}
\definecolor{allrefcolors}{rgb}{0,0.2,0.5}
\usepackage[linktocpage=true,colorlinks=true,allcolors=allrefcolors,bookmarksopen,bookmarksdepth=3]{hyperref}
\usepackage[margin=1.25in]{geometry}

\newcommand{\ord}{\operatorname{ord}}

\newcommand{\mC}{\mathcal{C}}

\newcommand{\degr}{\operatorname{deg}}

\newcommand{\Gammam}{\Gamma_{\operatorname{mero}}}

\newcommand{\Hla}{H^{\lambda}}
\newcommand{\hla}{h^{\lambda}}
\newcommand{\Hom}{\operatorname{Hom}}
\newcommand{\D}{\mathbf{D}}
\newcommand{\AL}{\mathcal{A}_\Lambda}
\newcommand{\PSS}{\operatorname{PSS}}
\newcommand{\PSSlog}{\operatorname{PSS}_{log}}

\newcommand{\E}{\mathbf{E}}

\newcommand{\Etop}{E_{\operatorname{top}}}
\newcommand{\Egeo}{E_{\operatorname{geo}}}

\newcommand{\Spec}{\operatorname{Spec}}
\def\hlm{h^{m}}
\def\Hlm{H^{m}}

\def\rlm{R^{m}}

\newcommand{\HH}{\operatorname{HH}}
\newcommand{\hatH}{\hat{H}_2(M,\mathbb{Z})}

\newcommand{\FwCF}{F_wCF_{\Lambda}^*(X \subset M; H^{\lambda})}

\newcommand{\hatX}{\widehat{\Sigma}_{\epsilon_{1}}}

\newcommand{\hatLio}{\widehat{X}_{\epsilon_{1}}}

\newcommand{\EG}{\mathbb{E}(\underline{\Gamma})}
\newcommand{\GammaPSS}{\underline{\Gamma}_{PSS}}
\newcommand{\Gammas}{\Gamma^{\operatorname{stab}}}
\newcommand{\oEG}{\vec{\mathbb{E}}(\underline{\Gamma})}
\newcommand{\VG}{\mathbb{V}(\underline{\Gamma})}
\newcommand{\EGPSS}{\mathbb{E}(\underline{\Gamma}_{PSS})}
\newcommand{\oEGPSS}{\vec{\mathbb{E}}(\underline{\Gamma}_{PSS})}
\newcommand{\VGPSS}{\mathbb{V}(\underline{\Gamma}_{PSS})}

\newcommand{\RK}{\mathfrak{R}_{K_{\operatorname{inv}}}}

\newcommand{\hD}{h_{D^{c_{0}}_{\v}}}

\def\ra{\rightarrow}

\def\Z{{\mathbb Z}}
\newcommand{\T}{\mathbb{T}}
\newcommand{\bD}{\mathbb{D}}
\newcommand{\ham}{\operatorname{ham}} 
\def\R{{\mathbb R}}

\def\K{{\mathbf{k}}}

\def\A{\mathcal{A}}

\newcommand{\Perf}{\operatorname{Perf}}
\newcommand{\cO}{\mathcal{O}}

\def\mc#1{\mathcal{#1}}

\def\z2{\Z / 2\Z}

\def\z{\mc{Z}}

\def\v{\mathbf{v}}

\def\A{\mc{A}}

\newcommand{\WXp}{\Perf(\mathcal{W}(X))}

\newtheorem{lem}{Lemma}[section]
\newtheorem{prop}[lem]{Proposition}
\newtheorem{thm}[lem]{Theorem}
\newtheorem{cor}[lem]{Corollary}
\newtheorem{conj}[lem]{Conjecture}
\newtheorem{defn}[lem]{Definition}

\newtheorem{ques}[lem]{Question}
\newtheorem{rem}[lem]{Remark}

\theoremstyle{remark}

\newtheorem{example}{Example}[section]
\numberwithin{equation}{section}
\begin{document}
\begin{abstract}

 The main result of the present paper concerns finiteness properties of Floer theoretic invariants on affine log Calabi-Yau varieties $X$.  Namely, we show that: \begin{enumerate} \item  the degree zero symplectic cohomology $SH^0(X)$ is finitely generated and is a filtered deformation of a certain algebra defined combinatorially in terms of a compactifying divisor $\D$. \item For any Lagrangian branes $L_0, L_1$, the wrapped Floer groups $WF^*(L_0,L_1)$ are finitely generated modules over $SH^0(X).$ \end{enumerate}

We then describe applications of this result to mirror symmetry,  the first of which is an ``automatic generation" criterion for the wrapped Fukaya category $\mathcal{W}(X)$.  We also show that,  in the case where $X$ is maximally degenerate and admits a ``homological section",  $\mathcal{W}(X)$ gives a categorical crepant resolution of the potentially singular variety $\Spec(SH^0(X))$.  This provides a link between the intrinsic mirror symmetry program of Gross and Siebert and the categorical birational geometry program initiated by Bondal-Orlov and Kuznetsov. 
\end{abstract}


\title{Intrinsic mirror symmetry and categorical crepant resolutions}
\author{Daniel Pomerleano}
\thanks{D.~P.~was supported by EPSRC, University of Cambridge, and UMass Boston during the development of this project.}
\maketitle

\section{Introduction}
\subsection{Finiteness}

A positive pair $(M,\D)$ consists of a smooth,  complex projective variety $M$ and a strict normal crossings anti-canonical divisor $\D:= D_1 \cup \cdots \cup D_i \cup \cdots \cup D_k$ supporting an ample line bundle $\mathcal{L}$.  For positive pairs,  the complement $X$ is an affine variety and can be equipped with a symplectic structure by taking a K\"{a}hler form $\omega$ associated to (a positive Hermitian metric $|| \cdot ||$ on) $\mathcal{L}$ and restricting this form to $X$.  This symplectic structure is exact and convex at infinity and,  furthermore, independent of the choice of compactification up to a suitable notion of deformation.  In view of this,  one can attach a number of Floer theoretic invariants to $X$.  The most classical of these is symplectic cohomology,  $SH^*(X)$,  which is a Hamiltonian Floer cohomology for exact,  convex symplectic manifolds $X$ developed by Viterbo \cite{Viterbo:1999fk} (building on pioneering work of Cieliebak-Floer-Hofer \cite{Cieliebak:1995fk}).  As with ordinary Hamiltonian Floer cohomology,  it carries a pair-of-pants product which makes it into a unital ring.  In \cite{Abouzaid:2010ly},  Abouzaid and Seidel introduced wrapped Floer cohomology,  which is a Lagrangian intersection version of Viterbo's construction for pairs of (suitably decorated) exact Lagrangian submanifolds which are cylindrical at infinity.  For any two such Lagrangians $L_0,L_1 \subset X$, the wrapped Floer groups $WF^*(L_0,L_1)$ are naturally modules over $SH^*(X)$.  

This paper concerns the study of these invariants in the case that $\D$ is anticanonical; we refer to such pairs $(M,\D)$ as positive Calabi-Yau pairs and to the complements $X:= M\setminus \D$ as affine log Calabi-Yau varieties.  We note that in the log Calabi-Yau case,  $SH^*(X)$ is canonically $\mathbb{Z}$-graded (because there is a preferred trivialization of $\mathcal{K}_X$).  Wrapped Floer invariants on affine log Calabi-Yau varieties have recently attracted a great deal of attention because of their importance in mirror symmetry (see e.g. \cite{Auroux, MR3415066, Pascaleff,  HackingKeating} and references there-in).  This is the prediction that (at least in nice cases) there exists a mirror algebraic variety $X^{\vee}$ so that symplectic invariants on $X$ can be described in terms of algebro-geometric  invariants on $X^{\vee}$.  Under the mirror dictionary,  the symplectic cohomology is expected to correspond to the polyvector field cohomology $HT^*(X^{\vee}):= H^*(X^{\vee}, \Lambda^*TX^{\vee})$,  whose degree zero piece is nothing but $\Gamma(\mathcal{O}_{X^{\vee}})$,  the ring of global functions on $X^{\vee}$.  Furthermore,  for each Lagrangian $L$,  there should exist a corresponding complex of coherent sheaves $E_L \in D^bCoh(X^{\vee})$ so that for any pair of Lagrangians $L_0,L_1$,  $WF^*(L_0,L_1)$ is canonically isomorphic to $ \operatorname{RHom}_{X^{\vee}}^*(E_{L_{0}},E_{L_{1}})$.  

The mirror partner to a given $X$ is in general not another affine variety,  but one expects that it is at least \emph{semi-affine},  meaning it admits a proper map to an affine scheme.  Semi-affineness imposes the following finiteness conditions on $HT^*(X^\vee)$ and $\operatorname{RHom}_{X^{\vee}}^*$:  \begin{enumerate} [label=(\alph*)] \item \label{item:finitelygenB} the ring of global functions $\Gamma(\mathcal{O}_{X^{\vee}})$ is finitely generated over the base field \cite{MR327772}. \item \label{item:finitelygen2B} polyvector field  cohomology $HT^*(X^{\vee})$ is a finitely generated module over  $\Gamma(\mathcal{O}_{X^{\vee}}).$ \item For any $E_0,E_1 \in D^bCoh(X^{\vee})$, $\operatorname{RHom}_{X^{\vee}}^*(E_0,E_1)$ is a finitely generated module over $\Gamma(\mathcal{O}_{X^{\vee}}).$ \end{enumerate} 

Our main result is a direct mirror ``translation" of the above finiteness statements to Floer cohomology on $X$: 

\begin{thm} \label{thm:conj1} For any affine log Calabi-Yau variety and any field $\K$: \begin{enumerate} [label=(\alph*)] \item \label{item:finitelygen} $SH^0(X,\K)$ is a finitely generated $\K$-algebra.  \item \label{item:finitelygen2} $SH^*(X,\K)$ is a finitely generated module over $SH^0(X,\K)$. \item \label{item:finitelygen3} For any $L_0,L_1,$ $WF^*(L_0,L_1)$ is a finitely generated module over $SH^0(X).$ \end{enumerate}  \end{thm}

The actual proof of Theorem \ref{thm:conj1},  whose main ideas we now outline,  is  largely independent of these mirror symmetry heuristics.  Part \ref{item:finitelygen} of Theorem \ref{thm:conj1} follows immediately from a more precise result,   Theorem \ref{thm:maintheorem},  which requires a bit of additional notation to state.  
 

For a vector $\v=(v_i)$ in $(\mathbb{Z}^{\geq 0})^k$,  we define the support of $\v$,  $|\v|$,  to be the set of $i \in \lbrace 1,\cdots,k \rbrace$ such that $v_i \neq 0.$ Let $\mathcal{A}_\K$ be the free $\K$-module given by:
  
 \begin{align}\label{eq: Amodule} \mathcal{A}_{\K}:= \bigoplus_{\v} H^0(D_{|\v|},  \K) \end{align} 
We can equip this vector space with a ring structure which,  in intuitive terms,  records how the different strata of $\D$ intersect.  To do this,  represent homogeneous elements of $\mathcal{A}_{\K}$ by $\alpha_\v$ with $\alpha \in H^0(D_{|\v|},  \K)$.  For any pair $\alpha \in  H^0(D_{|\v_1|},  \K)$, $\beta \in H^0(D_{|\v_2|},  \K)$ define \begin{align} \label{eq: ringstructure}
    \alpha_{\v_{1}} \ast_{\operatorname{SR}} \beta_{\v_{2}} = (i_{\v_1+\v_2,\v_1}^*(\alpha) \cup i_{\v_1+\v_2,  v_2}^*(\beta))_{\v_1+\v_2}
   \end{align} 
where $i_{\v_1+\v_2,\v_1}:D_{|\v_{1}+\v_{2}|} \hookrightarrow D_{|\v_{1}|} $, $i_{\v_1+\v_2,\v_2}: D_{|\v_{1}+\v_{2}|} \hookrightarrow D_{|\v_{2}|}$ denote the natural inclusion maps.  Extending \eqref{eq: ringstructure} linearly defines a commutative algebra structure on $\mathcal{A}_\K$ which depends only on the dual intersection complex of $\D$, $\Delta(\D).$  We will denote the ring $(\mathcal{A}_\K, \ast_{\operatorname{SR}})$ by $\mathcal{SR}_\K(\Delta(\D)).$\footnote{the notation comes from the fact that if $\Delta(\D)$ is a simplicial complex(as opposed to a $\Delta$-complex),  $\mathcal{SR}_\K(\Delta(\D))$ agrees with the Stanley-Reisner ring of $\Delta(\D)$ as studied in combinatorial commutative algebra.} Lastly,  we let $B(M,\D) \subseteq (\mathbb{Z}^{\geq 0})^k$ to be the set of vectors $\v$ such that $D_{|\v|} \neq \emptyset$.  





\begin{thm} \label{thm:maintheorem} Let $(M,\D)$ be a positive Calabi-Yau pair (equipped with a polarizing line bundle $\mathcal{L}$). There is a canonical isomorphism of rings \begin{align} \label{eq:ringiso} gr_{F_w}SH^0(X,\K) \cong \mathcal{SR}_{\K}(\Delta(\D)) \end{align} Moreover, when $\operatorname{char}(\K)=0$, the module $SH^0(X, \K)$ has a canonically defined basis of elements $\theta_{(\v,c)}$ with $\v \in B(M,\D)$ and $c$ is a connected component of $D_{|\v|}$ (and is thus isomorphic to $\mathcal{A}_\K$ as an $\K$-module with fixed basis).     \end{thm}

We also prove a version of this result (Theorem \ref{thm:additiveiso}) for an enhanced version of symplectic cohomology,  $SH^*(X,\underline{\Lambda})$, which is linear over a certain Novikov ring $\Lambda$.  Theorem \ref{thm:maintheorem} builds on  \cite{GP1, GP2} and specifically improves on those papers in two different respects.  \begin{itemize} \item First,  the theorem holds for arbitrary Calabi-Yau pairs (in fact a version of it also holds in the somewhat more general context of ``log nef pairs"; see Definition \ref{defn:lognefdef}) unlike \cite[Theorems 5.31, 5.37]{GP2} which establish an isomorphism of rings of the form \eqref{eq:ringiso} when $M$ is Fano or when $\operatorname{dim}(M)=2$.\footnote{Theorem \cite[Theorem 5.37]{GP2} was previously proven in a completely different way by Pascaleff \cite[Theorem 1.2]{Pascaleff}}  \item Second, in characteristic zero,  it constructs a specific basis of elements, $\theta_{(\v,c)}$, for symplectic cohomology which is conjecturally related to similar bases which appear in Gross-Siebert's intrinsic mirror symmetry programme (\cite{GrossSie2}).  We discuss the connections between our work and their programme a bit more in \S \ref{subsection:ims}.  \end{itemize} 

To explain the new ingredient in the proof of Theorem \ref{thm:maintheorem},  recall from \cite{GP2} that for any positive $(M,\D)$,  the filtration $F_w$ gives rise to a multiplicative spectral sequence converging to the symplectic cohomology ring: \begin{align} \label{eq:specGP2} E_r^{p,q} =>  SH^*(X, \K) \end{align} 

For Calabi-Yau pairs, (a special case) of  \cite[Theorem 1.1]{GP2} provides a canonical identification of rings: 
\begin{equation}\label{eq:sspage1}
\operatorname{PSS}_{log}^{low}: \mathcal{SR}_{\K}(\Delta(\D))  \stackrel{\cong}{\rightarrow} \bigoplus_{p} E_1^{p,-p}.
\end{equation}

The map \eqref{eq:sspage1} is given by counting certain low-energy moduli spaces of solutions (``log PSS solutions") $u: \mathbb{C}P^1\setminus \lbrace 0 \rbrace \to M$ which solve a suitable version of Floer's equation and which intersect the divisor $\D$ with a prescribed multiplicity at $\lbrace \infty \rbrace$.

 The proof of Theorem \ref{thm:maintheorem} is given by constructing, when $\operatorname{char}(\K)=0$, a degree zero (additive) splitting of the spectral sequence \begin{align} \label{eq:PSSlogmap} \operatorname{PSS}_{log}: \mathcal{A}_\K  \stackrel{\cong}{\rightarrow} SH^0(X,\K), \end{align} i.e. a filtered isomorphism (with respect to a natural filtration $F_w$ on  $\mathcal{A}_\K$)  whose associated graded map in a suitable sense is Equation \eqref{eq:sspage1}. This map will be constructed by counting log PSS solutions of aribitrary energy,  as opposed to just low energy solutions.

 The main challenge in doing this is that whereas sphere bubbling is \emph{a priori} excluded in the low energy moduli spaces,  it can occur in the moduli spaces of arbitrary energy,  which potentially interferes with having well-behaved compactifications.  To overcome this,  we use \emph{refined} versions of Gromov compactness in the relative setting.  There are by now several different approaches to this in the literature \cite{Tehrani, Ionel:2011fk, MR3383807}.  We rely on the approach of \cite{Tehrani} because it is phrased in elementary geometric terms and produces a smaller compactification than \cite{Ionel:2011fk}, however any of these approaches would be suitable for proving Theorem \ref{thm:maintheorem}. To regularize these strata, we rely on the technique of stabilizing divisors \cite{CieliebakMohnke} which has become widely used in the symplectic topology literature (see e.g. \cite{CharestWoodward}).  Having constructed the splitting \eqref{eq:PSSlogmap} in characteristic zero, a simple algebraic argument shows that, in arbitrary characteristic, the spectral sequence \eqref{eq:specGP2} degenerates in degree zero thereby completing the proof of Theorem \ref{thm:maintheorem}.  Part \ref{item:finitelygen} of Theorem \ref{thm:conj1} follows from Theorem \ref{thm:maintheorem} because a filtered ring (with positive ascending filtration) is finitely generated iff its associated graded is finitely generated.  


 The proof of parts \ref{item:finitelygen2} and  \ref{item:finitelygen3} are conceptually quite similar.  In part \ref{item:finitelygen2}, we recall that in \cite[Theorem 1.1.]{GP2} we identified the full $E_1$ page of \eqref{eq:specGP2} with a certain logarithmic cohomology group $H_{log}^*(M,\D)$.  This logarithmic cohomology has a certain ``periodic" structure,  containing many copies of the cohomology of torus bundles over various divisor strata indexed by multiplicities $\v$.   Algebraically,   this periodicity is captured by a finitely generated module structure over $\mathcal{SR}_{\K}(\Delta(\D)) \stackrel{\cong}{\rightarrow} \bigoplus_{p} E_1^{p,-p}$.  Because the spectral sequence collapses in degree zero,  the remaining pages are modules over $\mathcal{SR}_{\K}(\Delta(\D))$ as well.  Because the $E_1$ page is finitely generated as a module,  the subsequent pages are finitely generated as well.  For part \ref{item:finitelygen3},  using a result from \cite{McLeanwrapped},  we show that one can deform the two Lagrangians so that the Hamiltonian chords (for suitable choices of Hamiltonians) between them exhibit a similar periodic structure.  There is again a spectral sequence for $WF^*(L_0,L_1)$ each of whose pages are modules over $\mathcal{SR}_{\K}(\Delta(\D))$.  We again show that the $E_1$ page is a finitely generated module over $\mathcal{SR}_{\K}(\Delta(\D))$ (generated by ``short chords").

\subsection{Applications}

We now turn to applications of the above results to wrapped Fukaya categories $\mathcal{W}(X)$ \cite{Abouzaid:2010ly} of affine log Calabi-Yau varieties (all $A_{\infty}$ categories in this section will be linear over the ground field $\K$). For any $A_\infty$ category $\mC$, we will let $\Perf(\mC) \subset \operatorname{Mod}(\mC)$ denote the dg-category of perfect $A_\infty$ modules. 

A key structural fact about the dg-categories $\Perf(\mathcal{W}(X))$ is that they are smooth, Calabi-Yau dg-categories of dimension $n=\dim_\mathbb{C} X$ \cite{Ganatra}. (This in turn relies on the difficult fact that $\Perf(\mathcal{W}(X))$ is split-generated by the Lagrangian ``co-cores" of any Weinstein handlebody presentation \cite{GPS, Ghiggini}.) Combining Theorem \ref{thm:conj1} with these algebraic properties, we obtain the following ``automatic generation" criterion for wrapped categories of log Calabi-Yau varieties:





\begin{prop} \label{cor:autogen}(Proposition \ref{cor:autogenmain}) Let $(M,\D)$ be a Calabi-Yau pair and let $L$ be an object of $\WXp$ such that $$\Perf(\Hom^\bullet (L,L)) \cong \Perf(Y)$$ where $Y$ is a smooth quasi-projective scheme over $\K$. Then $<L>=H^0(\WXp).$  \end{prop} 

The result is reminiscent of recent results \cite[Theorem 1]{Sheel} and especially \cite[Theorem A]{SheridanPerutz} in the case of compact Fukaya categories. However, it should be emphasized that the method of proof of Proposition \ref{cor:autogen} is quite different and consists of two algebraic observations (in addition to Theorem \ref{thm:conj1}).  The first is that a smooth, Calabi-Yau dg-category with connected $HH^0$ does not admit non-trivial semi-orthogonal decompositions---this is a variant of a standard argument for proper Calabi-Yau categories (\emph{c.f. } \cite[Theorem 36]{Sheel}).  The second observation (see Lemmas \ref{lem:admissibility} and \ref{prop:agprop}) is that the module finiteness of wrapped Floer groups implies that any pre-triangulated subcategory of the form $\Perf(Y)$ is automatically admissible,  thereby generating a semi-orthogonal decomposition(which is necessarily trivial in view of the first observation).   

As mentioned above,  it is already known that $\WXp$ has a collection of generators.  Nevertheless, we expect that this result will find applications in situations where there is a collection of potential generators which are not obviously related to these co-cores. As an illustration of this, suppose $L_0$ is an object of $\mathcal{W}(X)$ such that the ``zeroeth order term" of the closed-open map \begin{align} \label{eq:easyiso}  \mathcal{CO}_{(0)}: SH^0(X) \cong WF^*(L_0,L_0) \end{align} is an isomorphism. (In particular, $WF^*(L_0,L_0)=0$ for $* \neq 0.$) We will refer to such an object as a ``homological section." The terminology comes from the fact that general expectations suggest that a section of a putative SYZ fibration on $X$ will be a homological section. However, in examples, it is usually much easier to construct homological sections than Lagrangian fibrations--- in dimension two \cite[Proposition 7.2]{Pascaleff} gives a geometric criterion for a Lagrangian brane to be a homological section and we give a similar criterion in Proposition \ref{prop:weaksection}  which applies in all dimensions. 

If $X$ is equipped with a homological section,  it follows that there is a fully faithful functor \begin{align} \label{eq:pihom} \pi^*: \operatorname{Perf}(\operatorname{Spec}(SH^0(X))) \hookrightarrow \Perf(\mathcal{W}(X)) \end{align} which sends the structure sheaf to $L_0$. In particular, when a homological section exists and $SH^0(X)$ is smooth, Proposition \ref{cor:autogen} immediately yields: 

\begin{cor}\label{cor:strongres} Suppose $SH^0(X)$ is smooth and $X$ admits a homological section, then \eqref{eq:pihom} is an equivalence of categories. \end{cor}

We illustrate this Corollary in the relatively straightforward case of one of the simple local models of Lagrangian fibrations with singularites. Consider specifically the conic bundle:  \begin{align} 
    \label{eq: conicbundleintro} X=\lbrace (z,u,v) \in  (\mathbb{C}^*)^{n-1} \times \mathbb{C}^2 |\; uv= 1+ z_1+ \cdots + z_{n-1} \rbrace 
\end{align}  

 \begin{prop} \label{prop:GrossSiebertintro} (Proposition \ref{prop:GrossSiebert}) Let $\K$ denote a field of characteristic zero and let $\mathcal{A}$ be the ring \begin{align} \label{eq:GSring}  \mathcal{A}:=(\K[u_1,\cdots, u_{n},w_1,w_2]/ (\prod_j u_j = 1+ w_1, w_1w_2=1) \end{align} We have an equivalence of categories  $$ \Perf(\Spec(\mathcal{A})) \cong \Perf(\mathcal{W}(X)) $$ \end{prop}

Proposition \ref{prop:GrossSiebertintro} also implies a mirror symmetry statement for finite (abelian) covers of $X$; Corollary \ref{cor:Gcovers}.

When $\Spec(SH^0(X,\K))$ is singular,  a result of the form of Corollary \ref{cor:strongres} cannot hold because, as discussed above, $\Perf(\mathcal{W}(X))$ is always a smooth dg-category.  This leads to the idea that $\Perf(\mathcal{W}(X))$ could be a categorical resolution of $\operatorname{Spec}(SH^0(X,\K))$ in the sense of Kuznetsov \cite[Definition~3.2 and Definition~3.4]{MR2403307}).  For this to make sense, we need to assume that $\operatorname{Spec}(SH^0(X,\K))$ has the same dimension as $\Perf(\mathcal{W}(X))$ ($=\dim_\mathbb{C} X$). To achieve this, we will assume that $(M,\D)$ is maximally degenerate,  i.e.  the divisor $\D$ has a zero-dimensional stratum.  In all of the subsequent results of this section,  we will further assume that $\operatorname{char}(\K)=0$ and that the strata $D_I$ of $\D$ are all connected.  Theorem \ref{thm:maintheorem} implies that $\operatorname{Spec}(SH^0(X,\K))$ has very mild singularities:

\begin{prop} \label{cor: Gorenstein} Suppose $\operatorname{char}(\K)=0.$ Let $(M,\D)$ be a maximally degenerate positive Calabi-Yau pair of dimension $n$ such that all strata of $\D$ are connected.  Then $Y:=\operatorname{Spec}(SH^0(X,\K)$ is a reduced Gorenstein $n$-dimensional scheme of finite type.  Moreover, $Y$ is Calabi-Yau.   \end{prop}


One can also show (see Proposition \ref{prop: singmain}) that $\operatorname{Spec}(SH^0(X,\K))$ has other nice properties; for example it is Du Bois (a weakening of rationality introduced by Streenbrink in \cite{MR607373} which is important in the theory of moduli of varieties \cite[Chapter~ 6]{MR3057950}).  In any event, using the fact that $\operatorname{Spec}(SH^0(X,\K))$  is Calabi-Yau, the same methods from Corollary \ref{cor:strongres} show that $\Perf(\mathcal{W}(X))$ is a categorical resolution of $\Spec(SH^0(X))$ which is furthermore crepant in the sense of Definition \ref{defn:ccr}.\footnote{The definition of crepancy that we use in this paper is slightly different from the more standard definition \cite[Definition~3.5]{MR2403307}, see Remark \ref{rem:crepe}.} 

\begin{prop} \label{cor:ccr}  Let $(M,\D)$ be as in Corollary \ref{cor: Gorenstein} and suppose that $X$ admits a homological section $L_0$.  Then the pair $(\Perf(\mathcal{W}(X))^{\otimes_{SH}}, \pi^*)$  define a categorical crepant resolution of $\operatorname{Spec}(SH^0(X,\K))$ in the sense of Definition \ref{defn:ccr}. \end{prop}

In the above proposition, $\Perf(\mathcal{W}(X))^{\otimes_{SH}}$ denotes a ``strictified" model of $\Perf(\mathcal{W}(X))$ which is linear over $SH^0(X,\K)$ in the naive sense. In the affine log Calabi-Yau setting, such models always exist (see Corollary \ref{cor:SHlin}) and are unique up to $SH^0(X,\K)$-linear equivalence.  One interesting corollary of Proposition \ref{cor:ccr} is the following refinement of Corollary \ref{cor:strongres}: 

\begin{cor} Let $(M,\D)$ be as in Proposition \ref{cor:ccr}.  For any affine subset $\Spec(B)$ in the regular locus of $\Spec(SH^0(X))$, there is an equivalence of dg-categories:
\begin{align} \Perf(B) \cong \WXp^{\otimes_{SH}} \otimes_{SH^0(X)} B \end{align} 
   \end{cor}

A far reaching conjecture of Kuznetsov (Conjecture \ref{conj:BO}) postulates that categorical crepant resolutions are unique up to equivalence. This conjecture together with Corollary \ref{cor:ccr} would imply that whenever a crepant resolution $Y$ of $\Spec(SH^0(X))$ does exist, we have an equivalence $$ \Perf(Y) \cong \Perf(\mathcal{W}(X)) $$ 
While this conjecture seems currently out of reach, there has been partial progress in some cases.  For example,  using a well-known result of Van Den Bergh \cite{VandenBergh}, we can show that: 

\begin{cor} \label{cor:lowdmirror2} Let $(M,\D)$ be as in Proposition \ref{cor:ccr} with $\dim_\mathbb{C}M \leq 3$.  Suppose that $\operatorname{Spec}(SH^0(X,\K))$ is integral with terminal singularities and that $\Perf(\mathcal{W}(X))$ admits a tilting generator (Definition \ref{defn: tilting}). Then $\operatorname{Spec}(SH^0(X,\K))$ admits a crepant resolution $Y$ and there is a derived equivalence: $$ \Perf (Y) \cong \Perf(\mathcal{W}(X)) $$  \end{cor} 

We expect that Corollary \ref{cor:lowdmirror2} will provide an efficient method for establishing new cases of mirror symmetry.  Nevertheless, from a conceptual point of view, it seems desirable to replace the assumption that $\Perf(\mathcal{W}(X))$ admits a tilting generator with the assumption that it admits a Bridgeland stability condition. This and other questions are discussed in \S \ref{subsection:questions}.



\subsection{Intrinsic mirror symmetry} \label{subsection:ims}

In the above discussion,  mirror symmetry served as heuristic motivation for Theorem \ref{thm:conj1}.  We now examine deeper connections between our work and the existing mirror symmetry literature.  The approach to mirror symmetry that is closest to our work here is the algebro-geometric framework of intrinsic mirror symmetry \cite{GrossSie2} (building on \cite{MR3415066, MR3758151, KY}).  Gross and Siebert have defined for any Calabi-Yau pair (not necessarily positive) a commutative ring which encodes the counts of certain rational curves in $M$ with incidence conditions along $\D$.  When $(M,\D)$ is maximally degenerate, Gross and Siebert propose that the spectrum of their ring should be viewed as a mirror partner to the pair $(M,\D).$ This construction is \emph{intrinsic} to the pair $(M,\D)$ without relying on any further degenerations (or choice of Lagrangian fibration). 

 For positive pairs, the Gross-Siebert ring is defined directly on the vector space $\mathcal{A}_\Lambda$ (the obvious variant of \eqref{eq: Amodule} over the Novikov ring $\Lambda$),  respects the filtration $F_w$ and the associated graded gring is again $\mathcal{SR}_{\Lambda}(\Delta(\D))$. Theorem \ref{thm:maintheorem} therefore suggests that $SH^0(X)$ and  $(\mathcal{A}_\Lambda,\ast_{\operatorname{GS}})$ are isomorphic as rings. Establishing such an isomorphism would be very interesting because $(\mathcal{A}_\Lambda,\ast_{\operatorname{GS}})$ is relatively computable (and in many cases can be computed entirely combinatorially \cite{Mandel}). We will take up this question in \cite{sequel}. 

The focus of the present paper is rather different and we temporarily set aside this question by (re-)defining the intrinsic mirror as $\operatorname{Spec}(SH^0(X,\K)).$ Indeed, Theorem \ref{thm:maintheorem} is sufficient to guarantee that most of the general properties of $(\mathcal{A}_\Lambda,\ast_{\operatorname{GS}})$ also hold for $SH^0(X).$ On the other hand, the advantage of defining the intrinsic mirror in terms of symplectic cohomology is that there is a direct point of contact between $SH^0(X)$ and $\mathcal{W}(X)$. For example, Corollary \ref{cor:strongres}, should allow one to establish HMS in many cases once the relationship between symplectic cohomology and the intrinsic mirror construction has been ironed out. In the general case when $\operatorname{Spec}(SH^0(X,\K))$ has singularities, Conjecture \ref{conj:BO} and Proposition \ref{cor:ccr} provide a potential path towards establishing Kontsevich's HMS conjecture (though one that seems out of reach for the moment). As hinted at above, it seems likely that one should also incorporate Bridgeland stability conditions into this picture.  On the other hand,  it is also worth noting that starting in complex dimension 4,  we expect that there are examples where no crepant resolution (even ``stacky") of $\Spec(SH^0(X,\K))$ exists.  If such examples do exist,  this would be evidence that the framework of categorical resolutions may be relevant to understanding Kontsevich's conjecture in higher dimensions.

\subsection{Outline of the paper}

The paper is organized as follows.  In  \S \ref{section:SHtor},  we review the necessary background on normal crossings compactifications,  Hamiltonian Floer cohomology, and wrapped Floer cohomology. This material is mostly standard or contained in \cite{GP2}. The heart of the paper is \S \ref{section:PSSmods} and \S \ref{section:maintheorem}.  In \S \ref{section:PSSmods}, we  adapt Tehrani's compactness analysis to described the compactified stable log PSS moduli spaces.  The main goal of \S \ref{section:maintheorem} is to  prove Theorem  \ref{thm:conj1} (and along the way Theorem \ref{thm:maintheorem}). Section \ref{sect: actionspec} describes the passage from Hamiltonian Floer cohomology at a fixed slope to symplectic cohomology and reviews from the spectral sequence from \cite{GP2}.  In \S \ref{subsection:relativeFloer},  we explain how to use stabilization to regularize the compactified PSS moduli spaces. In \S \ref{subsection: stabilizedPSS},  we begin by proving Theorem \ref{thm:additiveisominor} which suffices for proving parts \ref{item:finitelygen} and \ref{item:finitelygen2} of Theorem \ref{thm:conj1}.  We also explain how to upgrade everything to ``$\Lambda$-twisted coefficients" in  Theorem \ref{thm:additiveiso}.  The proof of Part \ref{item:finitelygen3} of Theorem \ref{thm:conj1} follows in \S \ref{subsection:proper}. 
  
We turn to the applications of these Floer theoretic results in \S \ref{section:applications}.  Section \ref{subsection:ha} develops the necessary homological algebra language and proves our ``automatic generation" result.  In this section,  we also prove  Lemma \ref{lem:topsection},  which gives a geometric criterion for a Lagrangian to be a homological section,  and discuss the extended example of a particular local model of singularities of Lagrangian fibrations.  \S \ref{sec:ccr} is where we turn to discussing categorical crepant resolutions as well as some interesting open questions which naturally arise from this work.  Throughout the applications section,  a number of algebraic results are deferred to Appendices \ref{section:appendixA} and \ref{section:appendixB}.  These are a mix of standard results and more novel results whose proof would disrupt the flow of the paper. The paper concludes with Appendix \ref{section:appendixC},  which describes a particular intrinsic mirror family with no smooth fiber.

\subsection*{Acknowledgements}
  Theorem \ref{thm:maintheorem} was worked out during the academic year 2017-2018 (when the author was a research associate at University of Cambridge) and this result was announced at several conferences around that time.  In the interim, the picture has fleshed out considerably, resulting in a delay in releasing this paper (for which I apologize).  I thank Mark Gross for sponsoring my time in Cambridge and for many helpful conversations about intrinsic mirror symmetry.  This paper was heavily influenced by collaborations and discussions with Sheel Ganatra.  Not only is Theorem \ref{thm:maintheorem} an extension of the techniques developed in \cite{GP1,GP2},  but we also had several useful discussions about different approaches to proving automatic generation in non-compact settings.  I also thank Mohammed Abouzaid for the suggestion that there should be a purely algebraic argument for automatic generation in the affine case. 

I would also like to acknowledge the considerable technical help that I received from a number of other mathematicians.  Namely, the proofs of various algebraic lemmas were explained to me by Hailong Dao (Lemma \ref{prop:DuBoisgen}), Dan Halpern-Leistner (Lemma \ref{prop:agprop}), and Dmitry Vaintrob (Lemma  \ref{lem:rectification}). Denis Auroux also helped me with the construction of the monotone torus which appears in Lemma \ref{lem: nonzeroproducty} and Bernhard Keller answered my questions about Calabi-Yau categories.  I am grateful to all of them.

\section{Symplectic Preliminaries} \label{section:SHtor} 



\subsection{Regularizations and normal crossings geometry}

 As a preliminary step in our analysis,  we will deform the convex symplectic structure on an affine variety $X$ to one which admits nice  ``models" near a compactifying divisor $\D$ and hence is suitable for studying symplectic cohomology (this technique for studying symplectic cohomology originates from \cite{McLean:2012ab} and is a key ingredient in \cite{GP1,GP2}).  
 
  Describing the models precisely requires some notation and terminology.  To begin, recall that a Hermitian line bundle $(L,  \rho, \nabla)$ over a smooth manifold $Z$ consists of a complex line bundle $\pi: L \to Z$ over  $Z$,  a Hermitian metric $\rho$,  and $\nabla$ a $\rho$-compatible connection.  A Hermitian structure on a real-oriented rank-two line bundle $L$ over $Z$ is a pair $(\rho, \nabla)$ so that $(L, \rho, \nabla)$ is a Hermitian line bundle,  where $L$ is given the complex structure $\mathfrak{i}_{\rho}$ determined by the Riemannian metric $Re(\rho)$.  In a slight abuse of notation,  given a Hermitian structure on some $L$ we will also use $\rho(v):=\rho(v,v)$ to refer to the norm-squared function on $L$.   Given a Hermitian structure on $L$,  we can associate a connection 1-form $\theta_{e} \in \Omega^1(L \setminus Z)$ which vanishes on the horizontal tangent spaces and which restricts to the ``angular one-form," $d\varphi= d\operatorname{log}(\rho)\circ \mathfrak{i}_\rho$ on the fibers of $\pi$.  

Suppose now that $Z$ is symplectic --- equipped with a symplectic form $\omega_Z$.  Let $L$ be a real-oriented rank-two bundle over $Z$ with a Hermitian structure $(\rho, \nabla).$ We can associate a 2-form on $L \setminus Z$ by the formula: 
 \begin{align} \label{eq: normalbform}
    \omega_{(\rho,\nabla)}= \pi^* \omega_Z + \frac{1}{2} d (\rho \theta_{e})
 \end{align} 
This form extends over the zero section $Z$ and is in fact symplectic in a neighborhood of $Z$.  Similarly,  given
 a collection of Hermitian line bundles $\{L_i = (L_i, \rho_i, \nabla_i)\}_{i
 \in I}$, we have connection 1-forms $\{\theta_{e,i}\}_{i \in I}$. Setting
 $\pi_{I, j}: \oplus_{i \in I} L_i \to L_j$ to be the projection we can form 
 \begin{align} 
     \omega_{(\rho_i,\nabla_i)}= \pi^* \omega_Z + \frac{1}{2} \sum_i \pi_{I,i}^*d (\rho_i \theta_{e,i})   
 \end{align}

Let $(M,\omega)$ denote a compact symplectic manifold and let $Z \subset M$ be a symplectic submanifold.  We will let $NZ$ denote the normal bundle of $Z$ inside of $M$.  In the case where $Z$ is a codimension two symplectic submanifold $Z \subset M$, we can equip $NZ$ with a Hermitian structure $(\rho, \nabla)$.  It follows from Weinstein's tubular neighborhood theorem that there is an embedding of a relatively compact open set
   $\psi: U \subset NZ \to M$, such that for every $z \in Z$, the (normal
   component of the) derivative $ D\psi_z: NZ_z \to NZ_z$ is the identity map and   
\begin{align} \psi^*(\omega)=\omega_{(\rho,\nabla)} \end{align} where $\omega_{(\rho,\nabla)}$ is defined as in Equation \eqref{eq: normalbform}. We will refer to such an embedding $\psi: U \subset NZ \to M$ as a (symplectic) regularization of the submanifold $Z$. 

We say that a collection $\{Z_i\}_{i \in S}$  of codimension two symplectic submanifolds is transverse if any subcollection $\{Z_i\}_{i \in I}$ of the submanifolds meet transversely. We will need to extend the notion of regularization to such transverse collections. The correct generalization has been given in \cite{MTZ}.  As before, we may equip each normal bundle $NZ_i$ with a Hermitian structure $(\rho_i, \nabla_i)$ and consider Weinstein tubular neighborhoods $\psi_i: U_i \to M.$  For every non-empty stratum $Z_I$, we also require that the overlaps $\cap_{i \in I} \psi_i(U_i)$ be covered by tubular neighborhoods of $Z_I$, $\psi_I: U_I \subset NZ_I \to M$ such that \begin{align} \psi^*(\omega)=\omega_{(\rho_i,\nabla_i)} \end{align}

 The maps $\psi_I: U_I \to M$ are required to satisfy a number of natural compatibility conditions; as these compatibility conditions will play a limited role in the work we carry out in this paper, we will not spell them out in detail here but refer the reader to \cite[Definition 2.11,  Definition 2.12]{MTZ}.  One thing that is worth noting is that if a transverse collection of divisors $\{Z_i\}$ admits a regularization, then they must be symplectically orthogonal. Thus, in contrast to the case of a single divisor, regularizations do not exist for arbitary transverse collections. However,  \cite[Theorem 2.13]{MTZ} (see also \cite{McLean:2012ab}*{Lemma 5.4, 5.15}) shows that given a transverse collection $\{Z_i\}$ in a compact symplectic manifold which intersect ``positively" in the sense of \cite[Definition 5.1]{McLean:2012ab },  we can deform our symplectic structure so that a regularization exists:
 
\begin{thm}\label{thm: MTZ} \cite[Theorem 2.13]{MTZ} 
Given a positively intersecting transverse collection of symplectic divisors $\{Z_i\}_{i=1}^k$ in  a compact symplectic manifold $(M,\omega_0)$,  there is a deformation of symplectic structures $\omega_t$,  $t \in [0,1]$ such that:
\begin{itemize}
\item $[\omega_t] = [\omega_0] \in H^2(M)$,
  \item  the divisors $\{Z_i\}$ are symplectic submanifolds for all $\omega_t$,
\item the deformation is supported in an arbitrarily small
neighborhood of the singular strata of the divisors $\{Z_i\}$,
\end{itemize} and such that the transverse collection of symplectic divisors
$\{Z_i\}_{i=1}^k$ admits an $\omega_1$-regularization.  
\end{thm}  

Throughout this paper,  our transverse collections of submanifolds will all come from algebraic geometry:
 
\begin{defn}
    \label{def:logsmooth} 
    A {\em positive pair}  is a pair $(M,\mathbf{D})$ with $M$ a smooth, projective $n$-dimensional variety and $\mathbf{D} \subset M$ a divisor satisfying
    \begin{align}
    &\textrm{The divisor $\mathbf{D}$ is normal crossings in the strict sense, e.g.,}\\
     &\nonumber\mathbf{D} := D_1 \cup \cdots \cup D_i \cup \cdots \cup D_k\ \textrm{where $D_i$ are smooth components of $\mathbf{D}$; and}\\
     &\label{eq:kappai}\textrm{There is an ample line bundle $\mathcal{L}$ on $M$ together with a section $s \in H^0(\mathcal{L})$ whose }\\
     &\nonumber\textrm{divisor of zeroes is $\sum_i \kappa_i D_i$ with $\kappa_i>0$}.
    \end{align}
\end{defn}

For the remainder of this section,  we fix a positive pair $(M,\D = D_1 \cup \cdots \cup
D_k)$. We let $X$ denote the affine complement $X:= M \setminus \mathbf{D}$.  For any $I \subset \{1, \ldots, k \}$, we set \begin{equation}
    D_I:= \cap_{i \in I} D_i.
\end{equation}

 Turning to symplectic structures,  equip $M$ with a symplectic form $\omega_\mathcal{L}$ which is K\"{a}hler for some positive Hermitian metric $|| \cdot ||$ on $\mathcal{L}$.   Consider the potential $h: X \to \mathbb{R}$ defined by the formula  $$h=- \operatorname{log}||s||,$$ where $s$ is the section given in \eqref{eq:kappai}.  Over $X$,  we have that  $\omega_\mathcal{L}:=-dd^c h$ and hence $\theta_\mathcal{L} = -d^c h$ is a primitive for $\omega_\mathcal{L}$ (i.e.  $d\theta_\mathcal{L}= \omega_\mathcal{L}$).
The tuple $(X, \omega_\mathcal{L},  \theta_\mathcal{L})$ equips $X$ with the structure of a finite-type convex symplectic manifold (see e.g. \cite[\S A]{McLean:2012ab} for the definition) which, up to deformation,  is independent of the compactification or the choice of ample line bundle $\mathcal{L}$ (\cite[\S 4]{Seidel:2010fk}).

We are now in a position to explain how we want to deform the finite type convex symplectic structure $(X, \omega_\mathcal{L},  \theta_\mathcal{L})$.   First,  using Theorem \ref{thm: MTZ},  we deform the symplectic form $\omega_{\mathcal{L}}$ to a form $\omega$ for which our divisors $\mathbf{D}$ admit a regularization.  We next choose a primitive $\theta$ for $\omega$ that has nice local models with respect to the regularization.  In what follows,  except when necessary we will drop the parameterizations $\psi_i$ from our notation and identify the source $U_i$ with its image in $M$.   To state the next result,  we let   
\begin{equation}
            \pi_I: U_I \ra D_I
        \end{equation}

denote the natural projection from the tubular neigborhood to the divisor stratum.
        
\begin{thm}\label{thm:niceprimitive} \cite{McLean:2012ab}*{Lemma 5.14}
    There exists a primitive $\theta$ for the restriction of the regularized symplectic form $\omega$ to $X$ so that:  
    \begin{itemize} \item $(X,  \omega, \theta)$ is a finite type convex symplectic manifold which is deformation equivalent to $(X, \omega_\mathcal{L},  \theta_\mathcal{L})$.
   \item  After possibly shrinking the neighborhoods $U_i$,  we have that on each $U_I$,  $\theta$ restricted to a fiber of $\pi_I$ agrees with
   \[
       \sum_{i \in I} (\frac{1}{2}\rho_i - \frac{\kappa_i}{ 2\pi}) d\varphi_i,
   \]
   \end{itemize} 
 \end{thm}

There is some $\epsilon_0$, such that $\rho_i^{-1}[0,\epsilon_0^2] \subset U_i$ for
all $i$.  For any $\epsilon_1$ such that $\epsilon_1^2 \leq \operatorname{inf}_i\frac{2\pi \epsilon_0^2} {\kappa_i},$ we set $U_{i,\epsilon_{1}}$ to be the region where $\frac{\rho_i}{\kappa_i/2\pi} \leq \epsilon_	1^2.$ We then set $UD_{\epsilon_{1}}:=
\cup_i U_{i,\epsilon_{1}}$ and \begin{equation} \label{eq:manicorners} \begin{aligned} \hatLio := M\setminus UD_{\epsilon_{1}} \\ \hatX:=\partial(\hatLio). \end{aligned} \end{equation} Let $\epsilon_2, \epsilon_3$ be two additional small parameters and let $\vec{\epsilon}$ denote the tuple $\vec{\epsilon}=(\epsilon_1,\epsilon_2,\epsilon_3).$ The space $\hatLio$ is a manifold with corners and in \cite[\S 2]{GP2},  we considered particular roundings of $\hatX$,  $\Sigma_{\vec{\epsilon}}$,  depending on a choice of tuple,  $\vec{\epsilon}$.  When $\epsilon_2$ is sufficiently close to $\epsilon_1$ and $\epsilon_3$ is sufficiently close to 0,  the compact submanifold of $X$ bounded by $\Sigma_{\vec{\epsilon}}$ is a Liouville domain $\bar{X}_{\vec{\epsilon}}$ whose boundary is $C^0$ close to $\hatX.$ For any $x \in X\setminus \bar{X}_{\vec{\epsilon}}$, let $R^{\vec{\epsilon}}(x)$ be the \emph{Liouville coordinate} defined as 
\begin{equation}\label{eq:inducedliouvillecoordinate}
    R^{\vec{\epsilon}}(x) = e^t,
 \end{equation}
 where $t$
is the time it takes to flow along the Liouville vector field $Z$ from the hypersurface $\partial
\bar{X}_{\vec{\epsilon}}$ to $x$.  Flowing for some small negative time $t_0^{\vec{\epsilon}}$ defines a collar neighborhood of $\partial \bar{X}_{\vec{\epsilon}} = \Sigma_{\vec{\epsilon}}$,
\[
    C(\Sigma_{\vec{\epsilon}}) \subset \bar{X}_{\vec{\epsilon}}.
\] 
Thus,  letting $X_{\vec{\epsilon}}^{o}$ denote the complement of this collar in the domain $\bar{X}_{\vec{\epsilon}}$
\[
    X_{\vec{\epsilon}}^o := \bar{X}_{\vec{\epsilon}} \setminus C(\Sigma_{\vec{\epsilon}}),
\]
 $R^{\vec{\epsilon}}$ may be viewed as a function 
$R^{\vec{\epsilon}}: X\setminus X_{\vec{\epsilon}}^{o} \to \mathbb{R}.$
In fact,  it was shown in \cite[Lemma 3.15]{GP1} that $R^{\vec{\epsilon}}$ extends smoothly across the divisors $\mathbf{D}$ and hence can be viewed as a function on $M \setminus X_{\vec{\epsilon}}^{o}$: \begin{equation}
    R^{\vec{\epsilon}}: M\setminus X_{\vec{\epsilon}}^{o} \to \R.  \end{equation} This is in turn a consequence of the fact (\cite[Lemma 3.14]{GP1}) that in $U_I \setminus \cup_{j\notin I} U_j$,  the function $R^{\vec{\epsilon}}$ only depends on the variables $\rho_i$ for $i \in I$ i.e.  over this region we have that  \begin{align} R^{\vec{\epsilon}}=  R^{\vec{\epsilon}}(\rho_i). \end{align} (This latter fact depends crucially on $\theta$ having the form described in Theorem  \ref{thm:niceprimitive}  and the specific choice of Liouville domain.)

 \subsection{Floer cohomology}

Here we explain our setup for Floer cohomology.  We begin by laying down some notation.  Let $H: M \times S^1 \to \mathbb{R}$ be a time-dependent Hamiltonian whose Hamiltonian flow preserves $\D.$ Time-one orbits of $X_H$ of such a Hamiltonian either lie entirely in $X$ or entirely in $\D$. We will refer to the orbits
contained in $\D$ as divisorial orbits and denote them by  $\mathcal{X}(\D; H)$ and denote all other orbits by  $\mathcal{X}(X; H)$.  For every orbit $x_0 \in  \mathcal{X}(X; H)$, we can define the action of $x_0$, $A_{H} (x_0),$  by the formula: 
\begin{align} \label{eq:generalaction} A_{H} (x_0)= -\int_{x_{0}} x_0^*(\theta) + \int_{0}^{1} H(t,x_0(t))dt \end{align} 

In order to obtain a well-behaved Floer theory on a noncompact space such as $X$, we need to restrict to a nice class of Hamiltonians. 

\begin{defn} \label{defn:linearf}  Fix a real number $\lambda>0.$ We say that a function $h^\lambda(R): \R \to \R$ is {\em admissible of slope $\lambda$} if the following conditions hold: 
\begin{enumerate}
    \item $h^{\lambda}= -\hbar_{\operatorname{min}}$ for $R \leq e^{-t^{\vec{\epsilon}}_0}$ where $\hbar_{\operatorname{min}}$ is a small non-negative constant (and  $t^{\vec{\epsilon}}_0$ is as before). 
    \item $(h^{\lambda}) ' \geq 0$; 
    \item \label{item:secondd} $(h^{\lambda}) '' \geq 0$  ; as well as 
    \item For some $K_{\vec{\epsilon}}$ satisfying $\operatorname{min}_{\D} R^{\vec{\epsilon}} >K_{\vec{\epsilon}}>1,$
    \begin{align} \label{eq:linearity}
            h^{\lambda}(R)=\lambda(R-1) \quad  \forall R \geq K_{\vec{\epsilon}}.
        \end{align}  
\end{enumerate}
\end{defn}

Note that because of condition (1), the composition $\hla \circ R^{\vec{\epsilon}}$ can be extended smoothly to a Hamiltonian on all of $M$, which we also call $\hla$: 
\begin{equation} \label{eq:Rveceq} 
   \hla: M \to \mathbb{R}.
\end{equation}

Fix the $K_{\vec{\epsilon}}$ for which \eqref{eq:linearity} as well as a second constant 
$\mu_{\vec{\epsilon}} \in (K_{\vec{\epsilon}}, \operatorname{min}_{\D}
R^{\vec{\epsilon}})$. This latter constant enables us to define open neighborhoods of $\D$
$V_{0,\vec{\epsilon}} = (R^{\vec{\epsilon}})^{-1}(\mu_{\vec{\epsilon}}, \infty)$,
$V_{\vec{\epsilon}} = (R^{\vec{\epsilon}})^{-1}(K_{\vec{\epsilon}}, \infty).$ For later use, we note that Equation \eqref{eq:linearity} implies that along the slice $R^{\vec{\epsilon}}=K_{\vec{\epsilon}},$  \begin{align} \label{eq:linearity2} H^{\lambda}=\theta(X_H)-\lambda \end{align}




Equation \eqref{eq:Rveceq} implies that the Hamiltonian flow of $h^{\lambda}$ preserves $\D$.  Hence,  the time-one orbits of this flow are either completely contained in $\D$ or completely
contained in $X$ (in fact $X\setminus V_{\vec{\epsilon}}$).  For our purposes, it will be sufficient to give the following brief description of $\mathcal{X}(X; h^{\lambda})$(see \cite[\S 2.2]{GP2} for a more detailed discussion).  The orbits $x_0 \in \mathcal{X}(X; h^{\lambda})$ come in two types of families,  the first being the set of constant orbits. This set of orbits (a submanifold with
boundary)  is given by the complement $\mathcal{F}_{\emptyset}:=
M\setminus \lbrace R_m^{\ell} \geq R_{0} \rbrace$, where $R_{0}$ be the largest value of $R^{\vec{\epsilon}}$
for which $h^{\lambda}(R^{\vec{\epsilon}})=-\hbar_{\operatorname{min}}$.
The second type of Hamiltonian orbit corresponds to (possibly multiply-covered) circles in the fiber where
\begin{align} \label{eq:hamvec}
    X_{h^{\lambda}}=\sum_{i } -2 \pi v_i\partial_{\phi_i}
\end{align}  
for an integer vector $\v \in \mathbb{N}^n$ which is strictly
supported on $I$.  We will assume that along these orbits,  the second derivative of $h^{\lambda}$ is strictly positive.   These orbits come in connected families,  denoted by
$\mathcal{F}_{\v}$,  which are manifolds with corners.  We define the weighted winding number of (a connected family
of) orbits $x_0$ to be $w(x_0)=\sum_i \kappa_i v_i(x_0)$.

\S \ref{section:maintheorem}  describes a careful choice of $C^2$ small time-dependent perturbation $H^{\lambda}: M \times S^1 \to \mathbb{R}$ which enables us to make all of the orbits (divisorial and otherwise) nondegenerate and has several other desirable properties.  For the moment,  the two most important properties of this perturbation is the following: 
\begin{itemize} 
    \item The perturbation is disjoint from $V_{\vec{\epsilon}} \setminus V_{0,\vec{\epsilon}}$ and inside of $M\setminus V_{\vec{\epsilon}}$, it is supported in the disjoint union of small isolating neighborhoods $U_\v$ of the orbit sets $\mathcal{F}_{\v}$.  

    \item The Hamiltonian flow of $\Hla$  preserves each divisor $D_i$.
\end{itemize}

At this stage, we assume for simplicity that $M$ has an anti-canonical divisor supported on $\D$,
\begin{align} \label{eq:volumeform}  \Omega_M \cong \mathcal{O}(\sum_i -a_i D_i) \end{align}

(This enables us to ensure that the Floer cohomology groups we define below may be $\mathbb{Z}$-graded.) For each Hamiltonian orbit $x \in \mathcal{X}(X; \Hla)$, the induced trivialization $\gamma$ of $x^*(TX)$ determines a 1-dimensional real vector space  $\mathfrak{o}_x$, the {\em determinant line} associated to a local Cauchy-Riemann operator $D_\gamma$. The {\em $\K$-normalization} of any vector space $W$, denoted $|W|_\K$ is the free $\K$ module generated by the set of orientations of $W$, modulo the relation that the sum of the orientations vanishes.  When the coefficient field is understood, we will sometimes just drop the $\K$ subscripts and just write $|W|$. In the case when $W=\mathfrak{o}_x$, the $\K$-normalization $|\mathfrak{o}_x|_\K$ is known as the orientation line.


Define the Floer complex to be the vector space generated by these orientation lines:
\begin{align} 
    CF^*(X \subset M;\Hla) := \bigoplus_{x \in \mathcal{X}(X; \Hla)} |\mathfrak{o}_x|_\K. 
\end{align} 

We grade the subspaces $|\mathfrak{o}_x|$ by $\deg(x)$,  the index of the local operator $D_x$ associated to $x$ (all of this defined with respect to the trivialization of the holomorphic volume form on $X$).  This index is in turn equal to $n - CZ(x)$,  where $CZ(x)$ is the Conley-Zehnder index of
$x$ \cite{Floer:1995fk}.  The next step is to equip the Floer complex with a  differential.  To do this,  we must specify the class of almost complex structures that we wish to work with.

 \begin{defn}\label{defn:complexint} 
Define $\mathcal{J}(M,\D)$ to be the space of $\omega$-tamed almost complex structures $J$ which preserve $\D.$  
\end{defn}

\begin{defn} \label{defn:contact} For any choice of Liouville domain and shells $\bar{X}_{\vec{\epsilon}}, V_{\vec{\epsilon}}, V_{0,\vec{\epsilon}}$,
    define $\mathcal{J}(\bar{X}_{\vec{\epsilon}}, V) \subset \mathcal{J}(M,\D)$ to be the space of
    $\omega$-compatible almost complex structures which are of contact type on
    the closure of $V_{\vec{\epsilon}} \backslash V_{0,\vec{\epsilon}}$, meaning on this region 
    \begin{equation}
        \theta \circ J = -dR^{\vec{\epsilon}}.
    \end{equation}
Finally, let $\mathcal{J}_F(\bar{X}_{\vec{\epsilon}},V)$ denote the space of $S^1$ dependent complex structures, $\mathcal{C}^\infty(S^1; \mathcal{J}(\bar{X}_{\vec{\epsilon}}, V))$. \end{defn} 
 Fix a pair of orbits $x_0,x_1 \in  \mathcal{X}(X; \Hla)$ and some $J_F \in \mathcal{J}_F(\bar{X}_{\vec{\epsilon}},V).$ A Floer trajectory is a   solution to the PDE:
\begin{align} \label{eq:FloerRinv}
\left\{
\begin{aligned}
 & u \colon \mathbb{R} \times S^1 \to X,   \\
& \lim_{s \to -\infty} u(s, -) = x_0\\
& \lim_{s \to +\infty} u(s, -) = x_1 \\
 &  \partial_s u + J_{F}(\partial_tu-X_{\Hla})=0.
 \end{aligned}
\right.
\end{align}
 We denote the space of solutions to \eqref{eq:FloerRinv} by $\widetilde{\mathcal{M}}(x_0,x_1)$.  There is an induced $\mathbb{R}$-action on the moduli space $\widetilde{\mathcal{M}}(\tilde{x}_0,\tilde{x}_1)$ given by translation in the $s$-direction.  Whenever $\deg(x_0) - \deg(x_1) \geq 1$ and $J_F$ is generic,  the quotient space 
\begin{equation}\label{eq:modulispacetraj}
    \mathcal{M}(x_0, x_1) := \widetilde{\mathcal{M}}(x_0,x_1)/ \mathbb{R}
\end{equation}
is a manifold of dimension \[ 
    \deg(x_0) - \deg(x_1)-1. 
\] 
Whenever $\deg(x_0) - \deg(x_1) = 1$(and $J_F$ generic), Gromov-Floer compactness implies that  the moduli space \eqref{eq:modulispacetraj} is compact of dimension 0.  Moreover,
the theory of ``coherent orientations" associates, to every rigid element 
$u \in \mathcal{M}(x_0,x_1)$ an isomorphism $$ \mu_u: |\mathfrak{o}_{x_{1}}| \cong |\mathfrak{o}_{x_{0}}|.$$
For any $z \in |\mathfrak{o}_{x_{1}}|$,  we can define the differential applied to $z$ as follows:  
\begin{equation} 
    \partial_{CF}(z)  =\sum_{x_0} \sum_{u \in \mathcal{M}(x_0,x_1)} \mu_u(z)  
\end{equation} 
where $x_0$ ranges over all orbits with $\deg(x_0) = \deg(x_1) + 1$.  We define $HF^*(X\subset M; \Hla)$ to be the cohomology of the complex 
$(CF^*(X \subset M; \Hla),\partial_{CF})$.  To define symplectic cohomology,  it remains to note that for $\lambda_2 > \lambda_1$,  there are continuation maps (see e.g.  \cite[\S 3]{Seidel:2010fk}) $$HF^*(X \subset M;H^{\lambda_1}) \to HF^*(X \subset M;H^{\lambda_2}). $$
The symplectic cohomology is then defined to be the direct limit:
\begin{align} \label{eq:standardSH} SH^*(X):= \varinjlim_{\lambda} HF^*(X\subset M; \Hla) \end{align}

For the moment,   we return to the Hamiltonians of a fixed slope $\lambda$ and explain the connection between the action filtration on $HF^*(X\subset M; \Hla)$ and the winding numbers $w(x)$. The key computation is the following:

\begin{lem} \label{lem: sharpactions} \cite[Lemma 2.10]{GP2}
Fix a slope $\lambda>0$.  Suppose that $\epsilon_1$ is sufficiently small and that $\Sigma_{\vec{\epsilon}}$ is sufficiently $C^0$ close to $\hatX$.  Then by taking: \begin{itemize} \item $K_{\vec{\epsilon}}$ is sufficiently close to 1, \item and $t^{\vec{\epsilon}}_0$,  $\hbar_{\operatorname{min}}$,  $||H^{\lambda}-h^{\lambda}||_{C^2}$ all sufficiently small \end{itemize} we can make the action (see \eqref{eq:generalaction}) of each orbit  $x_0 \in U_\v$ arbitarily close to: 
    \begin{align} \label{eq: action} 
        A_{H^{\lambda}} (x_0) \approx -w(x_0)(1- \epsilon_1^2/2) 
    \end{align}  
\end{lem}

Meanwhile,  for any Floer trajectory $u: \mathbb{R} \times S^1 \to X$, we can define the topological energy of $u$, $E_{top}(u)$ as
\begin{align} E_{top}(u)= \int u^*\omega- d(u^*H^{\lambda}dt) \end{align}  

It is elementary to see that $E_{top} \geq 0$ and it follows from Stokes' theorem that \begin{align} E_{top}(u)=A_{H^{\lambda}}(x_0)-A_{H^{\lambda}}(x_1) \end{align}

Thus,  the action of Hamiltonian orbits induces a filtration on the Floer complex.  Lemma \ref{lem: sharpactions} shows that in our situation the action filtration can be described in terms of the winding number $w(x)$.  For any $w$, let $\FwCF$ denote the complex generated by (orientation lines associated to) orbits $x \in  \mathcal{X}(X; H^{\lambda})_{\leq w}$ with $w(x) \leq w$: 

\begin{align} F_wCF^*(X \subset M; H^{\lambda}):= \bigoplus_{x \in \mathcal{X}(X; H^{\lambda})_{\leq w}}  |\mathfrak{o}_x| \end{align}  

\begin{cor} Fix a slope $\lambda>0$. Suppose that $\epsilon_1$ is sufficiently small and that $\Sigma_{\vec{\epsilon}}$ is sufficiently $C^0$ close to $\hatX$. Then by taking $K_{\vec{\epsilon}}$ is sufficiently close to 1 and $t^{\vec{\epsilon}}_0$, $||H^{\lambda}-h^{\lambda}||_{C^2}$ sufficiently small, we can ensure that the Floer differential preserves the filtration by $w(x_0)$. \end{cor} 
\begin{proof} Under the above assumptions,  Equation \eqref{eq: action} implies that the weight filtration agrees with the action filtration on the Floer complex, up to reversing the sign.   \end{proof} 

 For the rest of this paper, we will always choose our $\epsilon_1$, $\Sigma_{\vec{\epsilon}}$, and $H^{\lambda}$ so that this filtration exists.  More care is needed to construct the filtration on the direct limit \eqref{eq:standardSH} because as $\lambda$,  increases,  this entails taking roundings which are ``sharper" (meaning $C^0$ closer to $\hatX$).  We postpone discussion of this to \S \ref{sect: actionspec}.  
 
 We let $F_wHF^*(X\subset M; H^{\lambda})$  denote the filtration on $HF^*(X\subset M; H^{\lambda})$ induced from the cochain level filtration $F_wCF^*(X \subset M; H^{\lambda})$.  We also define the low energy Floer cohomology of weight $w$, $HF^*(X\subset M; H^{\lambda})_w$, by the formula
\begin{align} \label{eq: lowenergyFloergroup}
    HF^*(X\subset M; H^{\lambda})_w :=H^*(\frac{F_wCF^*(X \subset M; H^{\lambda})}{F_{w-1}CF^*(X
    \subset M; H^{\lambda})}) 
\end{align} 

We consider the corresponding descending filtrations:\footnote{This is done so that our conventions for cohomological spectral sequences match \cite{McCleary}.} $$F^{p}CF^*(X \subset M; H^{\lambda}):= F_{-p}CF^*(X \subset M; H^{\lambda}).$$ The filtration $F^p$ on the cochain complex gives rise to a spectral
sequence 
\begin{align}  
    \label{eq:Spec1} \lbrace E_{H^{\lambda},r}^{p,q}, d_r \rbrace
    \implies HF^*(X\subset M; H^{\lambda})  
\end{align} 

where the first page is by definition
identified with 
\begin{align} 
    \bigoplus_q E_{H^{\lambda},1}^{p,q}: = HF^*(X \subset M;
    H^{\lambda})_{w=-p}   
\end{align} 

 As is customary,  we set $E_{H^{\lambda},1}=\bigoplus_{p,q} E_{H^{\lambda},1}^{p,q}.$ We end this subsection with a lemma concerning Floer trajectories that will be crucial to ruling out undesirable Hamiltonian breaking along orbits in the divisor $\D$ in our compactness arguments for PSS moduli spaces:

\begin{lem} \label{lem: nocylinder} Suppose $x_0 \in \mathcal{X}(X; H^\lambda)$ and  $x_1 \in \mathcal{X}(\D; H^\lambda)$. Then there is no broken Floer trajectory (in $M$) with input $x_1$ and output $x_0$ such that \begin{align} \label{eq:assumptionlem} E_{top}(u) - A_{H^{\lambda}}(x_0)  < \lambda \end{align}  \end{lem}  
\begin{proof} This argument is contained in \cite[Lemma 4.14]{GP1}.  We summarize it here for completeness.  Given such a Floer trajectory $u: C:= \mathbb{R} \times S^1 \to M,$  We let $\bar{C}$ denote the piece of this curve which lies above $R=K_{\vec{\epsilon}}$ and $\underline{C}$ to be the piece of this curve which lies below this level set. Then we have that  \begin{align} E_{top}(\bar{C})= E_{top}(u)-E_{top}(\underline{C}) \\=  E_{top}(u)-A_{H^{\lambda}}(x_0)+\int_{\partial \bar{C}}u^*\theta- H^{\lambda}dt \end{align} where in the first line we have used the additivity of the topological energy and in the second line we have used the fact that by Stokes' theorem, we have that $E_{top}(\underline{C})= A_{H^{\lambda}}(x_0)-(\int_{\partial \bar{C}}u^*\theta- H^{\lambda}dt).$ In view of our assumption \eqref{eq:assumptionlem}, we have that \begin{align}  E_{top}(u)-A_{H^{\lambda}}(x_0)+\int_{\partial \bar{C}}u^*\theta- H^{\lambda}dt<\lambda + \int_{\partial \bar{C}}u^*\theta- H^{\lambda}dt \\ = 
 \lambda + \int_{\partial \bar{C}}u^* \theta-u^*\theta(X_{H^{\lambda}}) dt  + \int_{\partial \bar{C}}\lambda dt \\ = 
 \int_{\partial \bar{C}}u^* \theta-u^*\theta(X_{H^{\lambda}})dt      \end{align}
where in the second line we have used  Equation \eqref{eq:linearity2} and in the third line we have used that $\int_{\partial \bar{C}}\lambda dt=-\lambda$ by Stokes' theorem.  The rest proceeds as in the proof of the integrated maximum principle \cite[Lemma 7.2]{Abouzaid:2010ly} to conclude that $E_{top}(\bar{C}) \leq 0$,  which is a contradiction (\emph{c.f. } Equations (4.43)-(4.45) of \cite{GP1}). 
 \end{proof}

\subsection{Review of wrapped Floer cohomology} We next review the pieces of wrapped Floer cohomology and open closed maps that we will need to prove Theorem \ref{thm:properness} and  Proposition \ref{prop:weaksection} below.  In particular, while some of our results are phrased in terms of the wrapped Fukaya category of $X$,  all of the geometric arguments in our paper take place at the cohomological level and only make use of constructions described below.  Let $L$ be a properly embedded Lagrangian in $X$. We say that $L$ is: \begin{itemize} \item \emph{exact} if $\theta_{|L}=df_L$ for some function $f_L \to \mathbb{R}.$ \item \emph{cylindrical} (at infinity) if outside of a compact set in $X$, $L$ is invariant under the Liouville flow,  i.e. $Z$ is tangent to $L$.   \end{itemize}  

 Let $L_0,L_1$ be two exact Lagrangians and let $H: [0,1] \times X \to \mathbb{R}$ be a time dependent Hamiltonian. We let $\mathcal{X}(L_0,L_1; H)$ denote the (time-one) chords of the Hamiltonian vector field $X_H$ from $L_0$ to $L_1.$  Given a chord $x_0$, we define the action functional to be given by 
\begin{align} A_H^{L}(x_0):= f_{L_{1}}(x_0(1))-f_{L_{0}}(x_0(0))-\int_{x_{0}} x_0^*(\theta) + \int_{0}^{1} H(t,x_0(t))dt \end{align}  

\begin{rem} \label{rem:normalizing} It is worth mentioning explicitly that if $L$ is an exact, cylindrical Lagrangian, the function $f_L$ above is locally constant at infinity.  So, assuming $L$ is connected at infinity, one can normalize it to be zero on this region. \end{rem}

In order to define the Lagrangian version of Floer cohomology over the integers (instead of a field $\K$ of characteristic two), we will assume all of our Lagrangians are Spin and equip them with a choice of Spin structure. Gradings in wrapped Floer theory similarly require a choice of grading on each Lagrangian submanifold $L$ (recall that because we have assumed a natural class of volume form \eqref{eq:volumeform}). 

\begin{defn} A Lagrangian brane is an exact, cylindrical Lagrangian equipped with a grading and Spin structure. \end{defn} 

Let $\bar{X}_{\vec{\epsilon}}$ be one of the Liouville domains from \S \ref{section:SHtor}. Given two Lagrangian branes $L_0,L_1$, let $H^{\lambda}: [0,1] \times X \to \mathbb{R}$ be a compactly supported (time-dependent) small perturbation 
of an admissible Hamiltonian of slope $\lambda$ such that all time-one Hamiltonian chords, $\mathcal{X}(L_0,L_1; H^{\lambda})$, from (the Lagrangian underlying) $L_0$ to (the Lagrangian underlying) $L_1$ are nondegenerate.  The Floer differential is defined on the $\K$-module: \begin{align} \label{eq:LFcomplex} CF^*(L_0, L_1; H^{\lambda}):= \bigoplus_{x \in \mathcal{X}(L_0,L_1; H^{\lambda})} |\mathfrak{o}_{L_0,L_1,x}|_{\K} \end{align}  

where $\mathfrak{o}_{L_0,L_1,x}$ is a certain one-dimensional real vector space associated to the chord $x$ using index theory \cite{Seidel_PL}. Parallel to \eqref{eq:FloerRinv}, we let $\tilde{\mathcal{R}}(x_0,x_1)$ denote the space of solutions: \begin{align} \label{eq:FloerLag}
\left\{
\begin{aligned}
 & u \colon \mathbb{R} \times [0,1] \to X, \\
& u(s,0) \in L_0, \quad u(s,1) \in L_1   \\
& \lim_{s \to -\infty} u(s, -) = x_0\\
& \lim_{s \to +\infty} u(s, -) = x_1 \\
&  \partial_s u + J_{F}(\partial_tu-X_{\Hla})=0.
 \end{aligned}
\right.
\end{align} and we let $\mathcal{R}(x_0,x_1)$ denote the quotient by the $\mathbb{R}$-translation action. Because our Lagrangians are equipped with a grading, each chord $x$ can be assigned a degree $\deg(x)$. If $x_0, x_1$ are two chords with $\deg(x_0)>\deg(x_1)$, we have that for generic $J_t$ the moduli space $\mathcal{R}(x_0,x_1)$ is a manifold of dimension $\deg(x_0)-\deg(x_1)-1$ which has a Gromov compactification $\overline{\mathcal{R}}(x_0,x_1)$. In particular, when $\deg(x_0)-\deg(x_1)=1$ and $J_t$ is generic $\mathcal{R}(x_0,x_1)$ consists of a finite number of rigid solutions. Moreover, any rigid solution gives rise to an isomorphism of orientation lines, $\mu_u: \mathfrak{o}_{L_0,L_1,x_1}  \cong \mathfrak{o}_{L_0,L_1,x_0}.$ (In an abuse of notation we use the same notation to denote the induced map on $\K$-normalizations.) For any $z \in |\mathfrak{o}_{x_{1}}|_\K$, we can therefore set:
\begin{align} \partial_{LF}(z)= \sum_{x_{0}} \sum_{u \in \mathcal{R}(x_{0},x_{1})}\mu_u(z) \end{align}   

which gives a differential on the complex \eqref{eq:LFcomplex}.  We let $HF^*(L_0,L_1; H^{\lambda})$ be the cohomology of this complex. (In the case when $L_0=L_1$, we abbreviate this as $HF^*(L_0; H^{\lambda}).$) As with Hamiltonian Floer cohomology, choices of interpolating Hamiltonians $H_{s,t}$ and (generic) almost complex structures $J_{s,t}$ give rise to continuation maps: 
$$ \mathfrak{c}_{\lambda_1,\lambda_2}: HF^*(L_0,L_1 ; H^{\lambda_1}) \to HF^*(L_0,L_1; H^{\lambda_2}) $$
whenever $\lambda_1 \geq \lambda_2.$ We can define \begin{align} \label{eq:wrappedlim}  WF^*(L_0,L_1) := \varinjlim_\lambda HF^*(L_0,L_1; H^{\lambda}).  \end{align}

One can formulate versions of \eqref{eq:FloerLag} over more general Riemann surfaces with boundary.  As we will only use very special cases of this construction,  we will only recall the features of this that will be important for us,  referring the reader to \cite[Ch.  8]{Seidel_PL} for more details.  Let $(\bar{\Sigma}, \lbrace \zeta_i \rbrace)$ be a pair consisting of a (temporarily compact) Riemann surface with boundary $\bar{\Sigma}$ and boundary marked points $\lbrace \zeta_i \rbrace$ which have been split into two groups (``positive" and ``negative") $$\lbrace \zeta_i \rbrace:= \lbrace \zeta^{+} \rbrace \cup \lbrace \zeta^{-} \rbrace.$$  Let $\Sigma:=\bar{\Sigma} \setminus \cup_i \zeta_i $ and equip $\Sigma$ with Lagrangian labels, that is to say a collection of Lagrangian branes $\lbrace L_C \rbrace$ indexed by boundary components $\partial \Sigma_C$ of $\Sigma.$ The choice of Lagrangian labels means that each boundary puncture $\zeta_i$ has a pair of Lagrangian branes $(L_{0,\zeta_i}, L_{1,\zeta_i})$ associated to it by the two edges surrounding the puncture.  (The convention is that if $\zeta_i$ is a negative puncture,  $L_{1,\zeta_i}$ is assigned to the edge which comes before $\zeta_i$ according to the boundary orientation, while the opposite holds at a negative puncture).  We also assume that $\Sigma$ is equipped with the following additional choices: 
\begin{itemize} \item strip-like ends (\cite[ \S 8.d]{Seidel_PL}) along the boundary punctures (positive or negative according to the partition of the boundary points).  
\item Floer data $(H_{t,i}, J_{t,i})$ at each boundary marked point (\cite[ \S 8.e]{Seidel_PL}). 
\item Perturbation data $(K, J_\Sigma)$ compatible with the previous choice of Floer data (\cite[ \S 8.e]{Seidel_PL}).  
\end{itemize}

 A Floer solution is then a map $u: \Sigma \to X$ such that  \begin{align}
 \label{eq:generalizedFloery}
  \left\{
\begin{aligned} & u (\partial\Sigma_C) \subset L_C,  \\ 
& (du-X_K)^{0,1}=0. 
\end{aligned}
\right.
\end{align} 
and so that $u$ is asymptotic to some $x_i \in \mathcal{X}(L_{0,\zeta_i},L_{1,\zeta_i}; H_{t,i})$ at each boundary puncture $\zeta_i$.

A very special case of this construction allows one to construct a $\K$-linear graded category,  $H^*(\mathcal{W}(X))$,  whose objects are Lagrangian branes.  (In the literature this is sometimes referred to as the Donaldson category.) The morphisms between two Lagrangian branes $L_0, L_1$  in this category are the above wrapped Floer groups i.e. $$ \Hom_{H^*(\mathcal{W}(X))}(L_0,L_1):= WF^*(L_0,L_1). $$   The composition in this category is given by the ``triangle product"  which we summarize as follows: Consider a disc with three boundary punctures $(\zeta_0,\zeta_1,\zeta_2)$ ordered counter-clockwise.  Equip $\zeta_0$ with a negative strip like end and $\zeta_1,\zeta_2$ with positive strip like ends.  Consider a closed one form $\beta$ which restricts to $dt$ along the positive strip-like ends,  $2dt$ along the negative strip-like end, and such that $\beta_{|\partial S}=0.$ Taking $K$ to be a suitable perturbation of the perturbation one form $ h^{\lambda}\otimes \beta$, we can define a map $$HF^*(L_1,L_2 ; H^{\lambda}) \otimes HF^*(L_0, L_1 ; H^{\lambda}) \to HF^*(L_0,L_2 ; H^{2\lambda}) $$ by counting rigid solutions (with respect to a generic surface dependent almost complex structure) of  \eqref{eq:generalizedFloery} which are \begin{itemize} 
\item asymptotic along the strip-like ends to the appropriate Hamiltonian chords.  
\item map the component of the boundary between $\zeta_1$ and $\zeta_2$ to $L_1$,  the component between $\zeta_2$ and $\zeta_0$ to $L_2$ and the component between $\zeta_0$ and $\zeta_1$ to $L_0$. 
\end{itemize} 
 This operation is easily seen to pass to direct limits and defines the composition in the category. 

The last piece of structure we will need is a map connecting the Lagrangian and Hamiltonian flavours of Floer cohomology: \begin{align} \label{eq:COmainlambda} \mathcal{CO}_{(0)}: HF^*(X \subset M;H_F^\lambda) \to  HF^*(L; H_L^{\lambda}) \end{align} 

($H_F^{\lambda}$ and $H_L^{\lambda}$ are used to distinguish the \emph{a priori} different perturbations used to define the two Floer theories.) To define these,  consider the ``chimney domain," a disc with one boundary puncture (equipped with a negative strip like end) and one interior puncture (equipped with a positive cylindrical end).  Over the chimney domain,  we consider a closed one-form $\beta$ which vanishes along the boundary and which restricts to $dt$ along the (cylindrical and strip-like) ends and let $K$ be a suitable perturbation of $h^{\lambda}\otimes \beta$.  As before, \eqref{eq:COmainlambda} is defined by counting solutions to  \eqref{eq:generalizedFloery} which mapping the boundary to $L$ and are asymptotic to prescribed chords (see \cite[\S 6.14]{Ritter} for more details).  It is perhaps also worth noting that \eqref{eq:COmainlambda} fits into a more general TQFT framework where studies \eqref{eq:generalizedFloery} over Riemann surfaces with both interior punctures and boundary.

The maps \eqref{eq:COmainlambda} are easily seen to be compatible with continuation maps, giving rise to a map:  \begin{align} \label{eq:COmain} \mathcal{CO}_{(0)}: SH^*(X) \to  WF^*(L) \end{align}  It is not hard to check that this map is a ring homomorphism \cite[Theorem 6.17]{Ritter}. Thus, composing \eqref{eq:COmain} with the triangle product gives the wrapped Floer groups the structure of a (graded-) module over $SH^*(X)$.


\section{PSS moduli spaces} \label{section:PSSmods}

Let us first specify the type of pairs that we wish to work with: 

\begin{defn} \label{defn:lognefdef} We say that a pair $(M,\D)$ is \emph{log nef} if there is an isomorphism:
 \begin{align}\label{eq:volnef} K_M \cong \mathcal{O}(\sum_i -a_iD_i), \end{align} with all of the $a_i \leq 1$. \end{defn} 

\begin{defn} \label{defn:BMD} We let $B(M,\D) \subseteq (\mathbb{Z}^{\geq 0})^k$ to be the set of vectors $\v$ such that \begin{itemize} \item $D_{|\v|} \neq \emptyset$. \item $\sum (1-a_i) v_i =0.$ \end{itemize}  \end{defn}

\textbf{Note:} Throughout this section and \S \ref{section:maintheorem}, any vector $\v$ will always be assumed to be in $B(M,\D)$.  \vskip 5 pt

\subsection{Properties of relative maps} \label{subsection:relmaps}

Before turning to our specific geometric context, we begin by recalling a few facts and constructions concerning $J$-holomorphic maps $u: (\Sigma, \mathfrak{j}) \to M$ which apply very generally.  These concepts will be crucial to defining (stable) log PSS moduli spaces in \S \ref{subsection:PSSopen} and  \S \ref{section:stablePSS}. Throughout this subsection, we let $J$ be an almost complex structure in $\mathcal{J}(M,\D)$. Moreover, we will assume that $J$ is ``integrable in the normal directions to the divisors $D_i$" or more precisely that: 

\begin{enumerate}    \item for any $i \in \lbrace 1,\cdots, k \rbrace$, $p \in D_i$, and tangent vectors $\eta_1,\eta_2 \in T_pM$, the Nijenhuis tensor $N_J(\eta_1,\eta_2) \in T_p D_i$. \end{enumerate}

 We say that a $J$-holomorphic map $u: \Sigma \to M$ has \textbf{depth} $I_u \subset \lbrace 1,\cdots, k \rbrace $ if $u(\Sigma) \subset D_{I_{u}}$ and $u(\Sigma) \not \subset D_j$ for $j \notin I_u$. If $i \notin I_u$ (i.e. $u(\Sigma) \not \subset D_i$), it is well-known that for each point $z_0 \in \Sigma$ with $u(z_0) \in D_i$, there is a well-defined intersection multiplicity $\ord_i(z_0)$ with $D_i$ (see e.g.  \cite[Lemma 6.4]{CieliebakMohnke} or \cite[\S 3]{Ionel:2003kx}). We can extend $\ord_i$ to other points of $\Sigma$ by setting it equal to zero at points with $u(z_0) \notin D_i$; observe that with this definition $\ord_i(z)$ is non-vanishing at only finitely many points.  For defining log moduli spaces, we will need a slight refinement of this idea which allows for taking into account holomorphic jets at these non-vanishing points.  Namely, as before, we use the regularization to identify $U_i$ with its image in $M$. Let $z_0$ be a point with $\ord_i(z_0) \neq 0$ (in particular $u(z_0) \in D_i$) and fix a trivialization of $ND_i$ over a coordinate neighorhood $W \subset D_i$ centered about the point $u(z_0):$  \begin{align} \label{eq: trivND} ND_{i}|_{W} \cong  W \times ND_{i,u(z_0)}  \end{align} 
We then have an expansion (\cite[Lemma 3.4]{Ionel:2003kx}, \cite[Eq. 6.1]{TZ}) over a sufficiently small coordinate chart $\Delta \subset \mathbb{C}$ (centered about $z_0$ and with coordinate $z$)  \begin{align} \label{eq:expansion} \pi_\mathbb{C} \circ u|_{\Delta} = \eta_{i,z_0} z^{\ord_i} + O(|z|^{\ord_i+1}) \end{align}

where $\eta_{i,z_0} \in ND_{i,u(z_0)}$ is non-vanishing. We say that $\eta_{i,z_0}$ is the $\ord_i$-th \emph{jet evaluation} at the point $z_0$. This depends on the coordinate chart $\Delta$ (more precisely the choice of coordinate),  however we suppress this from the notation.

One of the main ideas of relative Gromov-Witten theory is to enhance the moduli space of stable maps in such a way that the notion of intersection multiplicity (and jet evaluations) extends to the case where $u(\Sigma) \subset D_i.$ The approach to this that we take here follows \cite{Tehrani}, which has the advantage of being both efficient and phrased in terms of standard differential geometric objects (holomorphic line bundles, meromorphic sections, etc.) If $i \in I_u$, the operator $D_u \bar{\partial}$ descends to a first-order differential operator \begin{align} \label{eq:delbar1} D_u^{ND_i} \bar{\partial}: \Gamma(\Sigma, u^*ND_i) \to  \Gamma(\Sigma, \Omega_{\Sigma,\mathfrak{j}}^{0,1} \otimes_\mathbb{C} u^* ND_i) \end{align}

It follows from the fact that $J$ is integrable in the normal direction that this is a complex linear operator and hence gives $u^*ND_i$ the structure of a holomorphic line bundle (\cite[Lemma 2.1]{Tehrani}). 

As $u^*ND_i$ is a holomorphic line bundle over $\Sigma$, we can consider the space of meromorphic sections  $\Gamma_{\operatorname{mero}}(\Sigma, u^*(ND_i))$. Notice that the complex Lie group $\mathbb{C}^*$ acts on $\Gamma_{\operatorname{mero}}(\Sigma, u^*(ND_i))$ by rescaling. We will use the notation $[\zeta] \in \Gammam(\Sigma, u^*(ND_i))/\mathbb{C}^*$ to denote the equivalence class of a meromorphic section $\zeta.$ The following is the key definition:

\begin{defn} A log curve $(\Sigma, u, [\zeta])$ consists of Riemann surface $\Sigma$, a $J$-holomorphic map $u: \Sigma \to M$, together with a collection of non-constant equivalence classes of meromorphic sections \begin{align} [\zeta]=([\zeta_i])_{i \in I_u} \in \prod_{i \in I_u} \Gammam(\Sigma, u^*(ND_i))/\mathbb{C}^* \end{align} \end{defn}

 \begin{rem} We use the terminology ``log" in order to be consistent with \cite{Tehrani}. However, in algebraic geometry, log schemes form a category and log stable maps as considered in \cite{MR3011419} are morphisms internal to this category. In this sense, the theory of exploded manifolds as developed by \cite{MR3383807} is perhaps a closer analogue of log geometry in the differential geometric setting. \end{rem} 

 Assume that $(\Sigma,u, [\zeta])$ is a log curve. Then we can construct a well-defined \emph{contact order function}: \begin{align*} 
    \operatorname{ord}_u: \Sigma \to \mathbb{Z}^k \\ z \to \lbrace \operatorname{ord}_{u,i}(z)\rbrace
\end{align*}

\begin{itemize} 
\item For $i \in I_{u}$, we set $\operatorname{ord}_{u,i}(z)$
to be the order of any zero or pole of any lift $\zeta_i \in \Gammam(\Sigma, u^*(ND_i))$ at $z$. This function is
non-vanishing at at most finitely many points $z \in \Sigma$. Note that the function $\operatorname{ord}_{u,i}$ is independent of the given lift $\zeta_i \in \Gammam(\Sigma, u^*(ND_i))$.

\item For $i \notin I_{u}$, we let $\ord_{u,i}= \ord_i(z)$
to be the order of contact with $D_i$ as defined previously(which is again
non-vanishing at at most finitely many points). 
\end{itemize}

The notion of jet evaluation extends as well using the analogue of Equation \eqref{eq:expansion}  applied to lifts $\zeta_i$ for each equivalence class $[\zeta_i]$ involved in the definition of our log curve. In more detail, assume we are at a marked point with $\ord_i(z_0) \neq 0$ for some $i \in I_u.$ Choose a lift $\zeta_i$ for each equivalence class $[\zeta_i]$, and a small chart $\Delta$ near $z_0$. Over $\Delta$, one may choose a holomorphic trivialization of $u^*ND_i:$ \begin{align} \label{eq: trivialization2} u^*ND_i \cong \Delta \times ND_{i,u(z_0)} \end{align} over $\Delta$. Then after possibly shrinking $\Delta$ we have \begin{align}  \label{eq: expansion2} \zeta_i|_{\Delta}= \eta_{i,z_0} z^{\ord_i} + O(|z|^{\ord_i+1})  \end{align} 

for some $\eta_{i,z_0} \neq 0 \in ND_{i,u(z_0)}.$ As before, the element $\eta_{i,z_0}$ manifestly depends on both the choice of lifts $\zeta_i$ as well as the given coordinate chart $\Delta$, however we again suppress this from the notation. 


\subsection{(Log) PSS moduli spaces} \label{subsection:PSSopen}

In this section, we review the log PSS moduli spaces from \cite{GP1,GP2} (see also \cite{SLef,  MR3974686} for related constructions in the case of a single smooth divisor) and we begin by describing their classical analogues from \cite{Piunikhin:1996aa}. These $\PSS$-moduli spaces are certain Floer theoretic moduli spaces defined on the domain $S= \mathbb{C}P^1 \setminus \lbrace 0 \rbrace$. We view this domain as having a distinguished marked point at $z_0=\infty$ and equip it with a negative cylindrical end about $z= 0$ defined by 
\begin{align*} \varepsilon: \mathbb{R} \times S^1 \to S \\
\varepsilon : (s,t) \to e^{2\pi(s+it)}. \end{align*} 

We let $\rho(s)$ be a monotone (non-increasing) cutoff function such that  $\rho(s)=0$ for $s>>0$ and $\rho(s)=1$ for $s<<0$ and set $\beta= \rho(s)dt.$  
In the next definition, we let $J_S$ be a surface dependent almost complex structure which is surface independent in a neighborhood of the point $z_0$ and which agrees with some time-dependent almost complex structure $J_t$ along the cylindrical ends. 


\begin{defn} Let $x_0 \in \mathcal{X}(X; H^\lambda)$ be a Hamiltonian orbit in $X$.  Recall that a $\mathbf{\PSS}$ solution asymptotic to $x_0$ is a map $u: S \rightarrow M$
satisfying the following variant of Floer's equation:
\begin{equation} \label{eq:PSSeq}
    (du - X_{H^{\lambda}} \otimes \beta)^{0,1} = 0
\end{equation}
(where $({0,1})$ is taken with respect to $J_S$) such that 
\begin{align}\label{eq:limitingcondition}
    \lim_{s \rightarrow -\infty} u(\varepsilon(s,t)) &= x_0
    \end{align}
\end{defn} 
 Let $\mathcal{M}(x_0)$ denote the moduli space of PSS solutions asymptotic to $x_0$. \vskip 5 pt

Note that all of the orbits $x \in \mathcal{X}(X; H^{\lambda})$ are contractible when viewed as orbits in $M$ and hence admit capping discs (in $M$). As is standard, we declare that two capping discs $u_1$, $u_2$ are equivalent if $$ [u_1\#(-u_2)] =0 \in \hatH $$ We let $\tilde{\mathcal{X}}(X; \Hla)$ denote the set of pairs $\tilde{x}=(x,[u_x])$, with $x \in \mathcal{X}(X; \Hla)$ and $[u_x]$ is an equivalence class of capping disc for $x.$ We separate the curves according to their relative homology classes as follows: 

\begin{defn} \label{defn:cappedPSS} For any capped orbit $\tilde{x}_0=(x_0,[u_{x_{0}}]),$ let $\mathcal{M}(\tilde{x}_0) \subset \mathcal{M}(x_0)$ denote those log $\PSS$ solutions $u$ such that $-u$ is equivalent to the capping disc $u_{x_{0}},$ $$-[u] \simeq [u_{x_{0}}].$$ \end{defn} 

We naturally have \begin{align} \label{eq:homologysplit} \mathcal{M}(x_0):= \bigsqcup_{[u_{x_{0}}]} \mathcal{M}((x_0,[u_{x_{0}}])). \end{align}

 For later use, it is convenient to record an alternative view on solutions to \eqref{eq:PSSeq} 
 (this alternative point of view is often called the ``Gromov trick"). Let $M_S:= S \times M$ and let $\pi_S: M_S \to S$ denote the projection to $S$. Equip $M_S$ with the almost complex structure complex structure $J_{M_{S}}$ defined by :  \begin{equation} \label{eq:Gromovcs} \begin{bmatrix} 
j_S & 0 \\
(X_{H^{\lambda}} \otimes \beta)\circ j_S - J_M \circ (X_{H^{\lambda}} \otimes \beta) & J_M
\end{bmatrix} \end{equation}  
 Specifying a solution to \eqref{eq:PSSeq} is then the same as specifying a \emph{pseudo-holomorphic} map (considered up to domain reparameterization) \begin{align} \label{eq:Gromovtricky} \tilde{u}: S \to S \times M \end{align} 

such that the projection $\pi_S \circ \tilde{u}$ is a bi-holomorphism and $\pi_S \circ u(z_0)= \lbrace \infty \rbrace$. For any $u \in \mathcal{M}(x_0)$, we define \begin{align} \label{eq:PSSEnergy} E_{top}(u)= \int_S u^*\omega - d(u^*H^{\lambda}\beta) \end{align}

\begin{lem} \label{lem:hbar} Suppose $H^{\lambda} \geq -\hbar$ for some constant $\hbar>0.$ Then for any PSS solution $u$,  $E_{top}(u) \geq - \hbar.$  \end{lem}
\begin{proof} The geometric energy $$E_{geo}(u) = \int_S u^*\omega - u^*(dH^{\lambda})\wedge\beta $$ is always non-negative.  We have that $$E_{top}(u) = E_{geo}(u) - \int_S H^{\lambda}d\beta. $$ 
By assumption,  $\int_S H^{\lambda}d\beta \leq \hbar.$ Therefore,  $$E_{top}(u) \geq E_{geo}(u)-\hbar \geq -\hbar.  $$
  \end{proof}  

For the next construction, we choose a surface dependent $J_S$ which preserves $\D$ i.e.  $J_z \in \mathcal{J}(M,\D)$ for all $z \in S$.

\begin{defn} \label{defn:PSSlog} A \emph{log} $\mathbf{\PSS}$ (with multiplicity $\v$) solution is a $\PSS$ solution such that \begin{itemize} \item $u$ does not intersect $\D$ anywhere except for at $z_0$. 
\item The intersection multiplicity of u with $D_i$ at $z_0$ is $v_i$:
\begin{align} (u \cdot D_i) = \v \end{align} 
\end{itemize} 
  \end{defn} 

We denote the moduli space of log PSS solutions with multiplicity $\v$ and asymptote $x_0$ by $\mathcal{M}(\v,x_0).$ It follows from standard arguments (see e.g. \cite[Section 4]{GP1}) that 
\begin{align} \label{eq:PSSvdim}  \operatorname{vdim}(\mathcal{M}(\v, x_0)) = \degr(x_0) \end{align} 
 
Furthermore, for generic $J_S$ with $J_z \in \mathcal{J}(M,\D)$ for all $z \in S$, the moduli space is cut out transversely. Another application of Stokes' theorem shows that for any $u \in \mathcal{M}(\v,x_0)$, the topological energy from \eqref{eq:PSSEnergy} is given by \begin{align} \label{eq:logenergy} E_{top}(u)=w(\v)+A_{H^{\lambda}}(x_0) \end{align}

We record the following lemma for later use: 

\begin{lem} \label{lem:sharpactionlem} Fix a slope $\lambda>0$. Suppose that $\epsilon_1$ is sufficiently small and that $\Sigma_{\vec{\epsilon}}$ is sufficiently $C^0$ close to $\hatX$. Then by taking $K_{\vec{\epsilon}}$ is sufficiently close to 1 and $t^{\vec{\epsilon}}_0$, $||H^{\lambda}-h^{\lambda}||_{C^2}$ sufficiently small,  we have that for any $\v$ and $x_0 \in \mathcal{X}(X; H^\lambda)$,  and $u \in \mathcal{M}(\v,x_0)$,  $E_{top}(u)$ can be made arbitrarily close to \begin{align} \label{eq:sharpPSSac} E_{top}(u) \approx w(\v) -w(x_0)(1- \epsilon_1^2/2) \end{align} 
 \end{lem}
\begin{proof} This follows immediately by combining \eqref{eq:logenergy} and  \eqref{eq: action}. \end{proof}    

It is sometimes useful to keep track of the relative homology class of the log PSS solution: 

\begin{defn} \label{defn:cappedlog} For any capped orbit $\tilde{x}_0=(x_0,[u_{x_{0}}]),$ let $\mathcal{M}(\v, \tilde{x}_0) \subset \mathcal{M}(\v, x_0)$ denote those log $\PSS$ solutions $u$ such that $-u$ is equivalent to the capping disc $u_{x_{0}},$ $$-u \simeq u_{x_{0}}.$$ \end{defn} 

We have \begin{align} \mathcal{M}(\v, x_0) :=  \bigsqcup_{[u_{x_{0}}]} \mathcal{M}(\v,(x_0,[u_{x_{0}}]))  \end{align}
where $\tilde{x}_0$ ranges over all the (homology classes of) capping discs for $x_0$.  \vskip 10 pt

\subsection{Stable log PSS moduli spaces} \label{section:stablePSS}

The purpose of this subsection is to show that given a triple $\v$, $H^{\lambda}$, $x_0 \in \mathcal{X}(X; H^{\lambda})$, with $\lambda>w(\v)$ and $||H^{\lambda}-h^{\lambda}||_{C^2}$ sufficiently small (and our complex structures suitably chosen), there is a well-behaved compactification $\overline{\mathcal{M}}(\v,x_0)$ of $\mathcal{M}(\v,x_0)$ (see Theorem \ref{thm: compactation} below). We first recall that the moduli spaces  $\mathcal{M}(\tilde{x}_0)$ have appropriate Gromov compactifications $\overline{\mathcal{M}}(\tilde{x}_0)$ which incorporate: 
\begin{enumerate} \item sphere bubbling \item breaking along cylindrical ends. \end{enumerate} 

The spaces $\overline{\mathcal{M}}(\tilde{x}_0)$ can be given a topology (we recall what we need about this topology in Definition \ref{defn:ordinarygromov}) which makes them into compact Hausdorff spaces. As in \eqref{eq:homologysplit}, we have: 

\begin{align} \label{eq:homologysplit2} \overline{\mathcal{M}}(x_0):= \bigsqcup_{[u_{x_{0}}]} \overline{\mathcal{M}}((x_0,[u_{x_{0}}])). \end{align} 

As is usual in pseudoholomorphic curve theory, the combinatorics of curves in our compactification are governed by graphs (in our case trees as we are working in genus zero) equipped with additional structure. We begin by specifying our notation/conventions about these trees: \vskip 5 pt 

\textbf{Notation for trees:}
\begin{enumerate} 
\item  The collection of vertices of a tree $\underline{\Gamma}$ will be denoted by $\VG$. Individual vertices will usually be denoted by $\nu$ and if  $\underline{\Gamma}$ is a rooted tree, the root vertex will be denoted by $\nu_{root}$.
\item  Internal or bounded edges of $\underline{\Gamma}$ will be denoted by $\EG$, and $\oEG$ will denote the set of oriented internal edges. For a vertex $\nu$, $\mathbb{E}(\nu)$ shall be the set of internal edges which neighbor $\nu$  and $\vec{\mathbb{E}}(\nu)$ shall denote the oriented edges. The oriented edge from $\nu$  to $\nu'$ will be denoted by $e_{\nu,\nu'}$.  \item Trees may also come with unbounded edges which we call legs. A \emph{PSS tree} will be a rooted tree with one leg $l_0$. 
\end{enumerate} 

 \begin{defn} \label{defn:bubblePSS} Let $y_0$ be a Hamiltonian orbit of $H^{\lambda}$ and fix a surface dependent almost complex structure $J_S$. A (parameterized) bubbled $\PSS$ solution modelled on $\GammaPSS$ and asymptotic to $y_0$ is an assignment of a marked curve to each vertex $\nu \in \VGPSS$ as follows.  At the root vertex $\nu_{root}$, assign a PSS solution $u_{root}$  asymptotic to $y_0$ with marked points $z_e$, $e \in \mathbb{E}(\nu_{root})$. We impose that the marked point $z_e$ corresponding to the edge closest to $l_0$ is equal to $z_0$. For any other vertex in  $\VGPSS \setminus \nu_{root},$ we assign a map  $u_\nu: \mathbb{C}P^1 \to M.$ We require the following conditions to hold: \begin{itemize} \item For any vertex $\nu$, let $C_\nu$ be the domain of the map corresponding to that vertex. For any pair of vertices, let $z_{e_{\nu,\nu'}} \in C_{\nu}$, $z_{e_{\nu',\nu}} \in C_{\nu'}$ be the marked points determined by $e_{\nu,\nu'}$ and $e_{\nu',\nu}$ respectively. We require $$ u_\nu(z_{e_{\nu,\nu'}}) = u_{\nu'}(z_{e_{\nu',\nu}}). $$ \item Moreover, for any $ \nu \neq \nu_{root}$, let $e \in \mathbb{E}(\nu_{root})$ be the edge closest to this vertex. We require $u_\nu$ is $J_{S,z(e)}$-holomorphic. 
  \end{itemize} 
 \end{defn}

We wish to consider bubbled $\PSS$ solutions up to isomorphism of maps $(C,u) \cong (C', u')$ ---that is to say \begin{itemize} \item an isomorphism of rooted trees $\tau: \GammaPSS \to \GammaPSS'$ \item suitable isomorphisms $\phi_\nu: C_{\nu} \cong C'_{\tau(\nu)}$ of punctured sphere (thimble) or marked $\mathbb{C}P^1$ for each vertex such that $u_\nu = u'_{\tau(\nu)} \circ \phi_\nu$. \end{itemize}

A bubbled $\PSS$ solution is stable if its automorphism group is finite. We let $\mathcal{M}(\GammaPSS, y_0)$ denote the moduli space of stable PSS solutions modelled on $\GammaPSS$ up to isomorphism.  

\begin{rem} \label{rem:breaking} As mentioned above, the full Gromov compactification also incorporates cylindrical breaking into Floer trajectories (with further bubbles attached to these). The combinatorics of this is also naturally described by trees. However, cylindrical breaking in our setting will be highly simplified and we therefore don't describe the somewhat involved notation needed for the general case. \end{rem}

To have compact stable log PSS moduli spaces, it will be necessary to impose further restrictions on the complex structures that we use. In each tubular neighborhood $ \pi_I: U_I \to D_I$, the connections $\nabla_i$ chosen for the regularization induce a canonical direct sum decomposition of the tangent space to any point $p \in U_I$ 
\begin{align} \label{eq:complexdecomp} T_p U_I \cong  \pi_I^*TD_I \oplus \pi_I^*ND_I  \end{align}

\begin{defn} \label{defn: split} 
    We say that $J_{\operatorname{std}} \in \mathcal{J}(M,\D)$ is {\em standard} if over each $\pi_I: U_I \to D_I$
\begin{itemize} 
    \item For every point p, the complex structure respects the decomposition \eqref{eq:complexdecomp}.
    \item On the first factor $\pi_I^*TD_I$, the almost complex structure is pulled back from one on $D_I$. 
    \item On the second factor $\pi_I^*ND_I \cong \pi_I^*(\bigoplus_i ND_i)$, the complex structure agrees with the sum $$ \pi_I^*(\bigoplus_i \mathfrak{i}_{i,I})$$ where $\mathfrak{i}_{i,I}$ is the complex structure on $ND_{i,|D_{I}}$ induced from the regularization. 
    \end{itemize}
We denote the space of complex structures which are standard in some neighborhood of $\D$ by $AK(M,\D).$
\end{defn} 

\begin{defn} Let $\mathcal{J}_S(M,\D)$ denote the space of surface dependent almost complex structures $J_S$ such that \begin{itemize} \item $J_{S,s} \in \mathcal{J}(M,\D)$ for all $s \in S$ \item In a neighborhood of $z_0$, $J_{S,s}=J_0$ for some surface independent $J_0$ which lies in $AK(M,\D).$ \item Along the cylindrical end, $J_S \in  \mathcal{J}_F(\bar{X}_{\vec{\epsilon}},V).$ for $s<<0$ (recall Definition \ref{defn:contact}). \end{itemize} \end{defn} \vskip 5 pt

\textbf{Note:} For the remainder of the paper, all surface dependent complex structures $J_S$ are taken to lie in $\mathcal{J}_S(M,\D)$.  \vskip 5 pt

 The moduli space $\overline{\mathcal{M}}(\v,x_0)$ will be a compact Hausdorff space equipped with a canonical topological embedding $\overline{\mathcal{M}}(\v,x_0) \hookrightarrow \overline{\mathcal{M}}(x_0)$ and the compactifying strata $\overline{\mathcal{M}}(\v,x_0) \setminus \mathcal{M}(\v,x_0)$ will have virtual codimension two (or higher).  Just as in the classical case, the boundary of the compactification will consist of moduli spaces of  ``stable log PSS solutions." In order to define these stable log PSS moduli spaces, we must introduce a bit more combinatorial notation. Let $\mathcal{P}$ denote the set of subsets of $\lbrace 1,\cdots, k \rbrace$. Given a PSS tree $\GammaPSS$ :


\begin{itemize} \item
  a depth function is an assignment  \[
     I^\nu: \VGPSS \to \mathcal{P},\\
     I^e: \EGPSS \to \mathcal{P}
 \]  
such that $I^{\nu_{root}}=\emptyset$ and for any $e$ which is bounded by vertices $\nu, \nu'$, $I^\nu \cup I^{\nu'} = I^e.$
 \item  a contact function is a map \begin{align} \v(-) : \oEGPSS \to \mathbb{Z}^k \end{align} 
such that for any pair of neighboring vertices $\nu,\nu'$ \begin{align} \label{eq:ordersmatch} \v(e_{\nu,\nu'})=-\v(e_{\nu',\nu}) \end{align} 
and the support $|\v(e_{\nu,\nu'})| \subset I^{e_{\nu,\nu'}}$(here we mean the depth function of the edge underlying the oriented edge).
\end{itemize} 

\begin{defn} \label{defn:PSStrees} A log $\PSS$ tree $\Gamma=(\GammaPSS, I^{\bullet}, \v(-))$ is a triple consisting of \begin{itemize} \item a $\PSS$ tree $\GammaPSS$ so that $\GammaPSS \setminus \nu_{root}$ is connected.  \item a depth function $I^\nu, I^e.$ \item and a contact function $\v(-)$. \end{itemize} \end{defn}

The condition that $\GammaPSS \setminus \nu_{root}$ is connected will correspond geometrically to the fact that for stable $\PSS$ solutions all sphere bubbles will attach at the marked point $z_0 \in S$. We set\begin{align*} \bD(\Gamma) := \mathbb{Z}^{\EGPSS}\oplus \bigoplus_{\nu \in \VGPSS}\mathbb{Z}^{I^{\nu}}. \\
       \T(\Gamma):= \bigoplus_ {e \in \EGPSS} \mathbb{Z}^{I^{e}} \end{align*} 
For the next construction, we require a choice of an arbitrary orientation on the set of edges $\EGPSS$(we let $\tilde{e}$ denote the oriented edge assigned to $e$). After making such a choice, we define a natural homomorphism \begin{align} \label{eq: linearparm} \rho: \bD(\Gamma) {\rightarrow} \T(\Gamma)\end{align}  by the following rule.
For vectors $1_e$ in the first component, we set \begin{align} \label{eq:rhoonee} \rho(1_e)=\v(\tilde{e}) \in \mathbb{Z}^{I^e}. \end{align} Given a vertex $\nu$, $i \in I^{\nu}$, and $e$ an edge, we set  \begin{align} \label{eq:rhodef} \tau_{\nu,i,e}=: \begin{cases} \vec{1}_{e,i} \in \mathbb{Z}^{I^e} &\mbox{if } \tilde{e}=e_{\nu,\nu'} \\  -\vec{1}_{e,i} \in \mathbb{Z}^{I^e} &\mbox{if } \tilde{e}=e_{\nu',\nu}\\ 0 &\mbox{otherwise} \end{cases} \end{align}  Finally, we set 
 $$\rho(1_{\nu,i})= \sum_e \tau_{\nu,i,e} \in \bigoplus \mathbb{Z}^{I^e} $$
This defines the map \eqref{eq: linearparm} on generators and we extend linearly to define the map on all elements. Let $\Lambda$ be the image of $\rho$ and let $\Lambda_\mathbb{C}$ denote its complexification. 

\begin{defn} We set $\mathcal{G}(\Gamma)$ to be the quotient group \begin{align} \label{eq:Gammadef} \prod_e(\mathbb{C}^*)^{I^e}/\operatorname{exp}(\Lambda_{\mathbb{C}}). \end{align} \end{defn}

With the preceding understood, we turn to defining moduli spaces of holomorphic curves with incidence conditions modelled on log $\PSS$ trees. Observe that for any $u:= (C_\nu,u_\nu)_{\nu \in \VGPSS} \in \mathcal{M}(\GammaPSS,x_0)$ with $x_0$ a Hamiltonian orbit in $X$, we can define a function $I_{(C_\nu,u_\nu)}^{\bullet}$ to $\Gamma$ as follows. \begin{itemize} \item For any vertex $\nu \in \VGPSS $, let
 $I_{(C_\nu,u_\nu)}^{\nu}:=I_{u_{\nu}}.$ ( $I_{u_{\nu_{root}}}=\emptyset$ corresponding to the fact that the broken PSS solutions pass through points in $X$.)  
\item For any edge $e \in \EGPSS$, let $I_{(C_\nu,u_\nu)}^e$ to be the depth of the point $u_\nu(z_{\tilde{e}})$; i.e. the subset corresponding to the deepest stratum containing $u_\nu(z_{\tilde{e}})$.  \end{itemize}

This is not quite a depth function for general $u:= (C_\nu,u_\nu)_{\nu \in \VGPSS} \in \mathcal{M}(\GammaPSS,x_0)$ because it only satisfies $I^\nu \cup I^{\nu'} \subset I^e$ (as opposed to equality), however, in the following definition we restrict our attention to those solutions for which this assignment does define a depth function. 

\begin{defn}  \label{defn:prelogmods} 
  Let $\Gamma$ be a log $\PSS$ tree (equipped with a depth function $I^\bullet$ and contact function $\v(-)$).
 A {\em pre-log}  (stable) map of multiplicity $\v$, modelled on $\GammaPSS$, and asymptotic to $x_0 \in \mathcal{X}(X; H^\lambda)$ consists of a collection 
$(C_\nu,u_\nu, [\zeta]_\nu)_{\nu \in \VGPSS}$ where \begin{itemize} 
\item $(C_\nu,u_\nu)_{\nu \in \VGPSS}$ define
a stable $\PSS$ solution $u_{PSS}: C \to M$ asymptotic to $x_0$ and modelled on $\GammaPSS$. We require that the function $I_{(C_\nu,u_\nu)}^{\bullet}$ associated to  $(C_\nu,u_\nu)_{\nu \in \VGPSS}$ be equal to $I^\bullet.$ 

\item  For each $\nu \in \VGPSS \setminus \nu_{root} $ 

\begin{align} [\zeta]_\nu =([\zeta_i])_\nu \in \{ \Gammam(\mathbb{C}P^1, u_{\nu}^*(ND_i))/ \mathbb{C}^*\}_{i \in
    I^{\nu}} \end{align} 

is a collection of meromorphic sections (defined up to scale) such that the associated functions
    $\{\operatorname{ord}_{\nu}\}_{\nu \in \VGPSS}$ satisfy: 
    \begin{itemize} 
        \item[(i)] $\operatorname{ord}_\nu$ is non-vanishing only at the marked points corresponding to $l_0$ and to edges $e_{\nu,\nu'} \in \oEGPSS.$ 
        \item[(ii)] We have that $\operatorname{ord}_{\vec{\nu}}(z_{l_{0}})=\v$ and for any vertex $\nu$, 
        \begin{align}
            \operatorname{ord}_{\nu}(z_e)= \v(e_{\nu,\nu'}),
        \end{align} 
        where $z_e$ is the marked point on $C_\nu$ corresponding to $e_{\nu,\nu'}.$ 
    \end{itemize} 
\end{itemize} 
\end{defn} 

An isomorphism of pre-log maps $(C,u, \lbrace [\zeta] \rbrace) \cong (C',u', \lbrace [\zeta'] \rbrace)$ is an isomorphism of PSS solutions such that $$ \phi_\nu^*([\zeta'_i]_{\tau(\nu)}) =[\zeta_i]_\nu. $$ 
Let $\tilde{\mathcal{M}}^{\operatorname{plog}}(\v, \Gamma, x_0)$ denote the moduli space of stable pre-log maps (with the natural $C^\infty_{loc}$ topology) and $\mathcal{M}^{\operatorname{plog}}(\v, \Gamma, x_0)$ denote the moduli space of stable pre-log maps up to isomorphism. 

\begin{prop} The expected dimension of a stratum $\mathcal{M}^{\operatorname{plog}}(\v, \Gamma, x_0)$ is given by: \begin{align} \label{eq:plogdim} \operatorname{vdim}(\mathcal{M}^{\operatorname{plog}}(\v, \Gamma, x_0)) = \degr(x_0)+ 2(\sum_{e \in \EGPSS} (|I^e|-1)-\sum_{\nu \in \VGPSS}|I^\nu|) \end{align}  \end{prop}
\begin{proof} Let $e_{root} \in \mathbb{E}(\nu_{root})$ denote the edge which corresponds to $\lbrace \infty \rbrace.$ Then the expected dimension of the spherical components before imposing the incidence condition is $$ \sum_{\nu \neq \nu_{root}} (2n-2|I^{\nu}|-6) + \sum_{e \neq e_{root}} 4 + 4$$  
where the second term comes from the fact that each edge corresponds to two marked points and the last term corresponds to the two additional marked points on the spherical components. Combining this with \eqref{eq:PSSvdim} shows that the expected dimension of all components before imposing the incidence condition is: \begin{align} \degr(x_0) + \sum_{\nu \neq \nu_{root}} (2n-2|I^{\nu}|-6) + \sum_{e \in \EGPSS} 4  \end{align} The matching condition at the marked points is a codimension  $\sum_{e \in \EGPSS} 2n-2|I^e|$ condition. Thus the expected dimension of  $\mathcal{M}^{\operatorname{plog}}(\v, \Gamma, x_0)$ is given by $$ \degr(x_0) + \sum_{\nu \neq \nu_{root}} (2n-2|I^{\nu}|-6) + \sum_{e \in \EGPSS} 4-\sum_{e \in \EGPSS} 2n-2|I^e| $$ which is easily seen to be equal to  \eqref{eq:plogdim} (recalling that $I^{\nu_{root}}=\emptyset)$. \end{proof}

When the divisor $\D$ is not smooth, the quantity $$ \sum_{e \in \EGPSS} (|I^e|-1)-\sum_{\nu \in \VGPSS}|I^\nu|$$ can potentially be non-negative (and in fact positive) for certain nodal pre-log curves and so we must impose further conditions on our curves to obtain a well-behaved compactification. The key to doing this is the following construction. Suppose we are given a pre-log map $(C,u,\lbrace [\zeta] \rbrace)$ with underlying graph $\GammaPSS$ and let $e \in \EGPSS$ be an edge connecting two vertices $\nu,\nu'$ (i.e. $e \in \mathbb{E}(\nu) \cap \mathbb{E}(\nu')$).  As in \eqref{eq:rhodef}, we arbitrarily choose an orientation $\tilde{e}= e_{\nu,\nu'}$ on $e$. Suppose that $C_\nu$ and $C_\nu'$ are equipped with meromorphic lifts $\zeta_i \in \Gammam(u^*(ND_i))$ of each equivalence class $[\zeta_i]$.  Choose coordinate charts $\Delta, \Delta'$ centered about $z_{e_{\nu,\nu'}}$ and $z_{e_{\nu',\nu}}$ respectively so that the expansions \eqref{eq:expansion} and \eqref{eq: expansion2} hold. For each $i \in I^e$, we set:  \begin{align} \label{eq:ratio} \eta_{e,i}= (\eta_{i,z_{e_{\nu,\nu'}}}/\eta_{i,z_{e_{\nu',\nu}}}) \in \mathbb{C}^* \end{align} 
Letting $e$ range over all edges $e \in \EGPSS$ then determines an element \begin{align} \label{eq:obstruction} \prod_{e \in \EGPSS}  (\eta_{e,i})_{i \in I^e} \in \prod_{e \in \EGPSS} (\mathbb{C}^*)^{I^{e}}  \end{align}

\begin{lem} The construction \eqref{eq:obstruction} descends to a well-defined map  \begin{align} \label{eq:obstruction20} ob_{\Gamma}: \mathcal{M}^{\operatorname{plog}}(\v, \Gamma, x_0) \to \mathcal{G}(\Gamma) \end{align} Moreover, \eqref{eq:obstruction20} is continuous in the Gromov topology.   \end{lem} 
\begin{proof} The action of the subgroup $\operatorname{exp}(\Lambda_\mathbb{C})$ correspond to changing the coordinate chart as well as lifts $\zeta_i$ of the equivalence classes $[\zeta_i]$ which are involved in the construction of \eqref{eq:obstruction}. For example, changing the coordinate chart at a marked point rescales  \eqref{eq:ratio} by an element of the one-dimensional subgroup generated by a vector $\rho(1_e)$ from \eqref{eq:rhoonee}. Thus, the element is well-defined modulo the action of the subgroup $\operatorname{exp}(\Lambda_\mathbb{C})$. The continuity of \eqref{eq:obstruction20} follows from the fact that, in the absence of bubbling (which occurs in deeper strata), the jet maps vary continuously in the $C_{loc}^{\infty}$ topology.  \end{proof}

\begin{defn} \label{defn:logmap} A log map (modelled on $\Gamma$) consists of a pre-log map $$(C_\nu,u_\nu, [\zeta]_\nu)_{\nu \in \VGPSS} \in \mathcal{M}^{\operatorname{plog}}(\v, \Gamma, x_0)$$ satisfying the following additional conditions: \begin{enumerate}[label=(\roman*)] \item There exists functions $\v_\nu: \VGPSS \to \mathbb{R}^k, \lambda_e: \EGPSS \to \mathbb{R}^{+}$ such that \begin{enumerate}[label=(\alph*)] \item $\v_\nu \in (\mathbb{R}^+)^{I^{\nu}} \times \lbrace 0 \rbrace^{k-|I^{\nu}|}$ \item for every pair of vertices $\nu, \nu'$ connected by an (oriented) edge $e_{\nu,\nu'}$, \footnote{Below $\lambda_e(e_{\nu,\nu'})$ denotes the value of $\lambda_e$ on the edge underlying the oriented edge. } $$\v_{\nu}-\v_{\nu'}=\lambda_e(e_{\nu,\nu'})\v(e_{\nu,\nu'}).$$  \end{enumerate} \item We have that  \begin{align} \label{eq:obstruction2} ob_{\Gamma}(u)=[1] \in \mathcal{G}(\Gamma) \end{align}  \end{enumerate} We let $\mathcal{M}(\v, \Gamma, x_0)$ denote the moduli space of log maps (modelled on $\Gamma$) up to isomorphism.  \end{defn} 

The next proposition shows that once we have imposed these additional conditions, the strata $\mathcal{M}(\v, \Gamma, x_0)$ all have expected dimension at most $\degr(x_0)-2:$ 

\begin{prop} For any log PSS tree $\Gamma$, let $\mathbb{K}_\mathbb{C}(\Gamma))$ denote the complexification of the kernel of \eqref{eq: linearparm}. The expected dimension of a stratum $\mathcal{M}(\v, \Gamma, x_0)$ is: \begin{align} \operatorname{vdim}(\mathcal{M}(\v, \Gamma, x_0)) = \degr(x_0)-2(\operatorname{dim}_\mathbb{C} \mathbb{K}_\mathbb{C}(\Gamma)) \end{align} Moreover, the only non-empty stratum with $\operatorname{dim}_\mathbb{C} \mathbb{K}_\mathbb{C}(\Gamma))=0$ is the stratum where $\Gamma_{PSS}$ is a point (``the main stratum").  \end{prop}
\begin{proof} We have that the dimension of $\mathcal{G}(\Gamma)$ is given by $$ \operatorname{dim}(\mathcal{G}(\Gamma)) = 2(\sum_e (|I^e|-1) -\sum_v |I_v| + \operatorname{dim}_\mathbb{C} \mathbb{K}_\mathbb{C}(\Gamma)).$$ (This is just an exercise in elementary linear algebra.) Combining this with \eqref{eq:plogdim} gives the statement on expected dimensions. The element $\v_{\nu}$ defines an element of  $\mathbb{K}_\mathbb{C}(\Gamma)$ which is non-trivial whenever $\underline{\Gamma}_{PSS}$ is, demonstrating the non-triviality of the kernel. \end{proof}

\begin{lem} \label{lem: onelift} Given a stable bubbled PSS solution (as in Definition \ref{defn:bubblePSS}) there is at most one (``log") lift to a stable log PSS solution (as in Definition \ref{defn:logmap}).  \end{lem}
\begin{proof} This follows exactly as in the proof of Lemma 3.14 of \cite{Tehrani}.  \end{proof} 

As noted in Remark \ref{rem:breaking}, to obtain a compact moduli space, we also need to account for breaking along the cylindrical end: 

\begin{defn} Fix a multiplicity $\v$ and orbit $x_0 \in \mathcal{X}(X;H^\lambda).$ A stable, broken log $\PSS$ solution $u$ consists of a collection of orbits $x_i \in \mathcal{X}(X;H^{\lambda})$ for $i \in \lbrace 1,\cdots, \ell_u \rbrace$ together with: \begin{itemize} \item a sequence of Floer trajectories $u_{Floer,i} \in \mathcal{M}(x_{i-1},x_i)$ (in particular these curves lie in $X$ and have no sphere bubbles attached.) \item a stable log PSS solution $u_{PSS}$ with output $x_{\ell_{u}}$. \end{itemize}   \end{defn} 

We let $\overline{\mathcal{M}}(\v,x_0)$ denote the moduli space of stable, broken log $\PSS$ solutions. In view of Lemma \ref{lem: onelift}, we can topologize $\overline{\mathcal{M}}(\v,x_0)$ as a subspace of $\overline{\mathcal{M}}(x_0)$. Note that it is therefore a priori Hausdorff. The remainder of this section will be devoted to demonstrating compactness of these moduli spaces (see Corollary \ref{thm: compactation}). We first recall quickly what it means for a sequence of stable PSS solutions to Gromov converge to stable, broken PSS solutions (in the classical sense):\footnote{As in Remark \ref{rem:breaking}, the definition below does not account for the case where sphere bubbles attach on the Floer trajectories. That is for brevity and because these will not arise in our setup.}

\begin{defn} \label{defn:ordinarygromov} Recall that if $u_n$ is a sequence of $stable \PSS$ solutions, modelled on PSS trees $\underline{\Gamma}_{PSS}^{(n)}$, we say that $u_n$ Gromov converges to a stable, broken solution $u_\infty$ modelled on $(\underline{\Gamma}_{PSS}, \underline{\Gamma}_i)$ (where all $\underline{\Gamma}_i$ are trivial trees with a single vertex) if there exists contractions of rooted trees $\tau_n: \underline{\Gamma}_{PSS} \to \underline{\Gamma}_{PSS}^{(n)}$, such that :
\begin{itemize} 
\item for each vertex $\nu \in \VGPSS \setminus \nu_{root}$, there exists Moebius transforms $\phi_{\nu,n}$ such that $u_{\nu,n}= u_{\tau_{n}(\nu)} \circ \phi_{\nu,n}$ converges to $u_{\nu,\infty}$ on compact sets $K \subset C_\nu \setminus \cup_{e \in \mathbb{E}(\nu)} z_e$ (here $z_e$ are the nodal points on $C_\nu$).
\item $u_{\nu_{root},\infty}=\operatorname{lim}_{n \to \infty} u_{\nu_{root},n}$ in the $C^\infty_{loc}$ topology. 
\item for every $i$, there exist a sequence of constants $s_n^{i} \to -\infty$ such that $u_{\nu_{root},n}(s+s_n^{i},t)$ converges to the corresponding Floer trajectory $u_{Floer,i}$. 
\end{itemize} 
\end{defn} 

As has already been discussed in \S \ref{subsection:relmaps}, a central feature of the relative setting is that components may ``fall into $\D$." For example, we may have a sequence of stable log $\PSS$ maps $(C_n, u_n, \lbrace [\zeta] \rbrace_n)$, where none of the components lie in $D_i$ and yet the limit contains a component which does lie in $D_i$. (Of course, one must consider intermediate cases where some of the components already lie in $D_i$ and a new component falls into the divisor.) Thus this component $u_\nu$ should inherit a equivalence classes of meromorphic sections $[\zeta_i] \in \Gamma(u_\nu^*(ND_i))/\mathbb{C}^*.$  The key construction is the following rescaling construction: \vskip 5 pt

\textbf{Rescaling construction(compare \S 3.2 of \cite{Tehrani}):} \vskip 5 pt 
Fix $J \in AK(M,\D).$ Suppose we have a sequence of $J$-holomorphic maps $(C_n,u_n)$ such that $u_n \not \subset D_i$ converging to $u_\infty$ which \emph{does} contain a component $u_\nu$ with $u_\nu (C_\nu) \subset D_i$.(Recall the notations $\tau_n$ and $\phi_{\nu,n}$ from Definition \ref{defn:ordinarygromov}.) For any $t \in \mathbb{C}^*,$ set \begin{align} R_t: ND_i \to ND_i, \quad R_t(v)=tv \end{align} 

Let $\psi_i: U_i \to M$ be the regularizing map and set \begin{align} \psi_t= \psi_i \circ R_t: R_t^{-1}(U_i) \to M \end{align} First consider a fixed compact subset $K \subset C_\nu$ disjoint from the nodal points.  We can assume that for $n$ large, that \begin{align} u_{n,K}:=\psi^{-1} \circ u_{\tau_n(\nu)} \circ \phi_{\nu,n}(K) \subset U_i \end{align} and is therefore described by an (a priori just smooth) section $\zeta_{n,K}$ of the normal bundle, the norm of which converges to 0 as $n \to \infty$. For a sequence of non-zero complex numbers $(t_{n,K})$, set \begin{align} u_{t_{n,K}}:=\psi_{t_{n,K}}^{-1} \circ u_{\tau_n(\nu)} \circ \phi_{\nu,n}. \end{align}  Fix a positive real number $c_K$ and choose a sequence of  $(t_{n,K})$ so that $\operatorname{sup} ||t_{n,K}^{-1} \psi_{n,K}||= c_K. $ Let $$ND_{i,c_k} \subset ND_i:= \lbrace v \in ND_i, ||v|| \leq c_k \rbrace. $$ For $n$ sufficiently large, the sequence of maps  $u_{t_{n,K}}$ defines a sequence of holomorphic maps $u_{t_{n,K}}: K \to ND_{i,c_K}$ for a standard complex structure $J_{\operatorname{std}}$ pulled back to $ND_{i,c_K}$ (this uses the fact that our complex structure is in $AK(M,\D)$). Then by Gromov compactness, this converges to a holomorphic section $\zeta_{\nu,K} \in \Gamma(K, u_\nu^* (ND_i)).$ 

Of course, we must construct this section consistently over an exhaustive sequence of compact sets $$ K_1 \subset K_2 \subset \cdots . $$ To do this, choose a reference point $p$ and assume $K_j$ all contain this point. We may, after possibly rescaling $c_{K_i}$ and the $(t_{n,K})$ and the vectors $\zeta_{\nu,K_i}(p)$ agree with a fixed vector $v_p \in ND_{i,u_\nu(p)}$. As a result, we get a holomorphic section over the open part of $C_\nu$ which extends meromorphically over all of $C_\nu$ \cite[Proposition 3.10]{Tehrani}. \vskip 5 pt

 This construction leads naturally to a definition for convergence of log maps:

\begin{defn} \label{defn:logPSSGromov} We say that a sequence of log $\PSS$ maps $(C_n, u_n, \lbrace [\zeta] \rbrace_n)$ modelled on log $\PSS$ trees $\Gamma'_n$ converges to $(C, u_\infty, \lbrace [\zeta] \rbrace$ modelled on $\Gamma$ if $u_n$ converges to $u_\infty$ as ordinary $\PSS$ solutions (recall the notations $\tau_n$ and $\phi_{\nu,n}$ from Definition \ref{defn:ordinarygromov}) and moreover there exists  a subsequence $u_{g(n)}$ ($g: \mathbb{N} \to \mathbb{N}$ is a strictly increasing function) such that:
\begin{itemize} \item the log trees $\Gamma'_{g(n)}$  have fixed underlying graph $\underline{\Gamma}'$ and depth function $(I')^{\bullet}$. \item Let $\tilde{\zeta}_n=\zeta_{g(n),i,\tau_{g(n)}(\nu)}.$ For a given vertex $\nu \in \VG$ and $i \in (I')^{\tau_{g(n)}(\nu)}$ there exists a sequence $(t_{\nu,n,i})$ of non-zero complex numbers such that for every compact set $K \subset C_\nu \setminus \cup_{e \in \mathbb{E}(\nu)} z_e$, $$t_{\nu,n,i}^{-1} \cdot \tilde{\zeta}_{n} \circ \phi_{\nu,n}|K$$ converges uniformly to $\zeta_{i,\nu}.$  \item For a given vertex $\nu \in \VG$  and $i \in I^{\nu} \setminus (I')^{\tau_{g(n)}(\nu)}$, there exists a sequence $(t_{\nu,n,i})$ of non-zero complex numbers so that for any $K \subset C_\nu \setminus \cup_{e \in \mathbb{E}(\nu)} z_e$, \footnote{the quantity in \eqref{eq:rescaling} is only defined for sufficiently large $n$} \begin{align} \label{eq:rescaling} \psi_{t_{\nu,n,i}}^{-1} \circ u_{\tau_n(\nu)} \circ \phi_{\nu,n}|K \end{align} converges uniformly to $\zeta_{i,\nu}.$ \end{itemize} 
\end{defn} 

The key compactness statement is then the following: 

\begin{thm} \label{thm: sequentialcomp} Suppose that $\lambda>w(\v)$ and that $||H^{\lambda}-h^{\lambda}||_{C^2}$ is sufficiently small. Then given a sequence of stable log $\PSS$ curves $(C_n,u_n, [\zeta]_n)$ there a subsequence which converges in the sense of Definition \ref{defn:logPSSGromov} to a stable log $\PSS$ curves $(C, u_\infty, \lbrace [\zeta] \rbrace).$ \end{thm}
\begin{proof} Given the sequence of curves $(C_n,u_n, [\zeta]_n)$, after passing to a subsequence, ordinary Gromov compactness produces a limit $(C,u_\infty)$. In view of Lemma \ref{lem:hbar} and \eqref{eq:logenergy} we have that for any Floer trajectory component $u_{Floer}$ of $u_{\infty}$, \begin{align} E_{top}(u_{Floer}) \leq w(\v)+A_{H^{\lambda}}(x_0) +\hbar \end{align}  where $\hbar$ can be made arbitrarily small by taking $||H^{\lambda}-h^{\lambda}||_{C^2}$ sufficiently small.  Lemma \ref{lem: nocylinder} therefore shows that $u_{Floer}$ cannot have output in $X$ and input in $\D$ from which it follows that the sequence $u_n$ does not break along $\D$.

     Let us first discuss the sphere bubbles which attach onto the PSS component $u_{\infty, \nu_{root}}$ at $z_0 \in S$ ($=\lbrace \infty \rbrace \in \mathbb{C}P^1 \setminus \lbrace 0 \rbrace$). Then by the ``rescaling construction" reviewed just before Definition \ref{defn:logPSSGromov} 
we conclude that each of these sphere bubbles may be equipped with equivalence classes of meromorphic sections $[\zeta_i]_{i \in I^{\nu}}$; Notice that if $\tau_n(\nu)$ is the root vertex and $u_{\tau_n(\nu)}$ is a $\PSS$ solution, we can assume that $\phi_{\nu,n}(K)$ lies in the region where $\beta=0$ ($u_{\tau_n(\nu)}$ is genuinely holomorphic) and so the construction and the arguments below apply to these sphere bubbles as well without modification. Let $\mathcal{A}(\Gamma)$ denote the set of vertices on the tree connecting $l_0$ to $\nu_{root}$ (including $\nu_{root}$).  Equation \ref{eq:ordersmatch} holds for any pair of vertices $\nu, \nu' \in \mathcal{A}(\Gamma)$ by the proof of Theorem 4.9 of \cite{Tehrani}. By construction, we also have that
\begin{align} \label{eq:intz0} \v = \sum_{\nu \in \mathcal{A}(\Gamma) \setminus \nu_{root}}([u_{\infty,\nu}] \cdot D_i) + \ord(z_0) \end{align} 
where $[u_{\infty, \nu}] \in H_2(M,\mathbb{Z})$ denotes the homology class of the sphere bubble component $u_{\infty,\nu}.$

Now we rule out the possibility of sphere bubbling occuring at some other point
in $S$ or along a Floer cylinder. Let $\mathcal{B}$ denote the set of sphere bubble components which attach at a point other than $z_0$. By positivity of symplectic area, there must exist a divisor $D_i$ for which \begin{align} \sum_{b \in \mathcal{B}} [u_{\infty,b}] \cdot D_i >0. \end{align} By positivity of
intersection with $\D$, any Floer cylinder or $\PSS$ solution must intersect $D_i$ with non-negative multiplicity. This however contradicts \eqref{eq:intz0} which implies that the sum of all intersections with $D_i$ away from $z_0$ must be zero.  

This implies that the collection $(C,u_{\infty}, \lbrace [\zeta] \rbrace)$ defines a pre-log PSS solution. The fact that is a log PSS solution is therefore equivalent to satisfying Conditions (i) and (ii) of Definition \ref{defn:logmap}. But this follows from the proof of Theorem 4.10 of \cite{Tehrani}(which also applies without modification).

 \end{proof}

\begin{cor} \label{thm: compactation} Suppose that $\lambda>w(\v)$ and that $||H^{\lambda}-h^{\lambda}||_{C^2}$ is sufficiently small. Then $\overline{\mathcal{M}}(\v,x_0)$ is a (sequentially) closed subspace of $\overline{\mathcal{M}}(x_0)$ and is therefore (sequentially) compact. \end{cor}  

\begin{defn} \label{defn:logPSScomp} Let $\v$ be as in Corollary \ref{thm: compactation} and let $c$ be a connected component of $D_{|\v|}$.  For any orbit $x_0 \in \mathcal{X}(X; H^\lambda)$,  define $\mathcal{M}(\v,x_0,c) \subset \mathcal{M}(\v,x_0)$ to be the moduli space of log PSS solutions with $$u(z_0) \in D_{|\v|,c}.  $$ Let $\overline{\mathcal{M}}(\v,x_0,c)$ denote the closure of this moduli space in $\overline{\mathcal{M}}(\v,x_0)$.   \end{defn} 

\section{Proof of Main Theorem} \label{section:maintheorem} 
 In this section,  we prove Theorem \ref{thm:conj1}.  From here until \S \ref{subsection:proper}, we continue working in the context of log nef pairs (recall the beginning of \S \ref{section:PSSmods}).  For any log nef pair,  we will set $\Delta(\D)_{(0)}$ to denote the sub $\Delta$-complex of the dual intersection complex generated by those divisors $D_i$ along which the chosen holomorphic volume form in  \eqref{eq:volnef} has a pole of order one.  We also set \begin{align}  \mathcal{A}_\K:= \bigoplus_{\v \in B(M,\D)} H^0(D_{|\v|}, \K). \end{align} 

This $\K$-module has a basis $\theta_{(\v,c)}$,  where $\v \in B(M,\D)$ and $c$ is a connected component of $D_{|\v|}$.  Equip $\mathcal{A}_\K$  with the product defined by extending \eqref{eq: ringstructure} $\K$-linearly.  The ring $(\mathcal{A}_\K,  \ast_{\operatorname{SR}})$ will be denoted by $\mathcal{SR}_{\K}(\Delta(\D)_{(0)}).$ It is immediate that these definitions all generalize the corresponding definitions for Calabi-Yau pairs given just before Theorem \ref{thm:maintheorem}. \vskip 5 pt 
 
 From \S \ref{subsection:proper} onwards, we restrict our attention to Calabi-Yau pairs. 

\subsection{Filtered symplectic cohomology} \label{sect: actionspec}
All of the constructions from \S \ref{section:PSSmods} are for Hamiltonians of a fixed slope $\lambda$. To transfer them to symplectic cohomology, we need to take the direct limit over Hamiltonians of higher slopes.  To carry out our argument,  we will choose our sequence of Hamiltonians appearing in the direct limit a bit more carefully(as compared to \eqref{eq:standardSH})  to ensure that \begin{itemize} \item the filtration by $F_w$ extends to the direct limit 
\item certain (``low energy") log PSS solutions used to define the map  
\eqref{eq:Speciso2} have sufficiently small topological energy.  
\end{itemize}

This is accomplished by taking a sequence of Liouville domains $\bar{X}_{m}:=\bar{X}_{\vec{\epsilon}_m}$ (indexed by $m \in \mathbb{N}^{>0}$) closer and closer to the divisor $\D$ as the slope increases.  We will assume that $\bar{X}_{m_{1}} \subseteq \bar{X}_{m_{2}}$ whenever $m_1<m_2$. Denote by 
$\Sigma_{m}:=\Sigma_{\vec{\epsilon}_{m}}$ to be the corresponding boundary $\partial\bar{X}_m$ of $\bar{X}_m.$ We also let $w_m \in \mathbb{N}^{>0}$ be a sequence of positive integers with $w_{m+1}>w_m$ and choose for each $w_m$ a real number $\lambda_m \notin \operatorname{Spec}(\bar{X}_{m})$ with $w_m<\lambda_m<w_{m+1}.$ Let $R^m$ be
the Liouville coordinate for each $m$ (which as before extends smoothly over $\D$) and choose an adapted Hamiltonian of slope $\lambda_m$, $\hlm(R^{m})$.  As before,  this extends
smoothly to a Hamiltonian on all of $M$, which we also call $\hlm$: 
\begin{equation}
\hlm:= \hlm(\rlm): M \to \mathbb{R}^{\geq 0}.
\end{equation}
  We also fix the regions $V_m$ and $V_{0,m}$ as well as the perturbed, nondegenerate Hamiltonians $\Hlm$ (see \ref{subsection:relativeFloer} for more details on these perturbations).  Provided all of this data is suitably chosen,  it is shown in \cite[Lemma 2.16]{GP2} that there are filtration preserving continuation maps  
 \begin{align} \label{eq:filteredcontinuations}
     \mathfrak{c}_{m_1,m_2}: F_wCF^{*}(X \subset
     M; H^{m_1}) \to F_wCF^{*}(X \subset M; H^{m_2}) 
 \end{align}

whenever $m_2 \geq m_1$.  We may therefore define: 
\begin{align}
    SH^*(X,\K):= \varinjlim_{m} HF^*(X \subset M; H^{m}) 
\end{align} 


 By counting pairs of pants satisfying a suitable variant of 
Floer's equation, 
we may also equip $SH^*(X, \K)$ with a filtration preserving product. 
 Using a standard cochain-level direct limit construction,  it is not hard to lift these limits to a cochain complex $SC^*(X,\K)$ such that $SH^*(X,\K) \cong H^*(SC^*(X, \K))$ in such a way that the filtrations by $w(\v)$ gives it the structure of a filtered complex.  To do this, we define  
 \begin{align} \label{eq: chainlevel} SC^*(X,\K):=
     \bigoplus_{m}CF^*(X \subset M; \Hlm)[q] 
 \end{align} 
 where $q$ is
 a formal variable of degree $-1$ such that $q^2=0$.  For $a+qb \in CF^*(X
 \subset M; \Hlm)[q]$, the differential on this complex is given by the
 formula 
\begin{align} \label{eq: directlimitchain}
    \partial(a+qb)=(-1)^{\deg(a)}\partial(a) +(-1)^{\deg(b)}(q\partial(b)+\mathfrak{c}_{m,m+1}(b)-b)  
\end{align}  Moreover, the corresponding descending filtration $F^p SC^*(X, \K)$ is bounded above and exhaustive,  which enables us to make use of the standard machinery of spectral sequences:

\begin{lem}  \label{lem:spectralsequ}
 The descending filtration $F^p SC^*(X,\K)$ gives rise to a convergent multiplicative spectral sequence 
 \begin{align} E_r^{p,q} => SH^*(X,\K)  \end{align} whose first page can be identified with the direct limit
\begin{align} \label{eq:E1concrete} \bigoplus_q E_{1}^{p,q}: = \varinjlim_m HF^*(X \subset M;
    \Hlm)_{w=-p}   
\end{align} 
where the groups at each stage of the direct limit are the low energy Floer groups from \eqref{eq: lowenergyFloergroup} and the maps are induced by \eqref{eq:filteredcontinuations}. 
\end{lem} 

We now turn to recalling the definition of the canonical map  \begin{align} 
 \label{eq:Speciso2} \PSSlog^{low} : \mathcal{A}_\K  \to \bigoplus_{p} E^{p,-p}_{1}  
\end{align} 

constructed in \cite{GP2}.  Recalling that $\mathcal{A}_\K = \bigoplus_{\v \in B(M,\D)} \K \cdot \theta_{(\v,c)}$,  we will define this map on each generator $\theta_{(\v,c)}$ and extend linearly.  For any $(\v,c)$,  choose $m$ so that  $w(\v) \leq w_m$ and let $x_0 \in \mathcal{X}(X; H^m)$ be an orbit such that $\deg(x_0)=0$ and $w(x_0)=w(\v)$.  The moduli spaces $\mathcal{M}(\v,x_0,c)$ (recall Definition \ref{defn:logPSScomp}) for such orbits are called low energy log PSS moduli spaces.  This is because (provided our Hamiltonians are taken as in Lemma \ref{lem:sharpactionlem}) it follows from Equation \ref{eq:sharpPSSac} that the energies of such solutions are so small that no sphere bubbling arises in the compactification of these moduli spaces.  After taking $J_S$ generic,  the fact there is no bubbling enables us to define an element $ \widehat{\PSS}_{log}^{m}(\theta_{(\v,c)}) \in HF^0(X \subset M; H^{m})_{w(\v)}$ by counting such low energy solutions: \begin{align} \label{eq:PSSlowe} \widehat{\PSS}_{log}^{m}(\theta_{(\v,c)})= \sum_{x_0,w(x_0)=w(\v)} \sum_{u \in \mathcal{M}(\v, x_0,c)} \mu_u  \end{align}
Where $\mu_u \in |\mathfrak{o}_{x_{0}}|$ is again an element determined by standard gluing/orientation theory.  For any $m' \geq m$, we have that:
\begin{align} \mathfrak{c}_{m,m'} \circ \widehat{\PSS}_{log}^{m}(\theta_{(\v,c)})= \widehat{\PSS}_{log}^{m'}(\theta_{(\v,c)}) \end{align}
and thus it follows from \eqref{eq:E1concrete} that this defines a well-defined element $\PSSlog^{low}(\theta_{(\v,c)}) \in \bigoplus_{p} E^{p,-p}_{1}.$ This defines the map \eqref{eq:Speciso2} on basis elements after which we extend linearly.  For the next result,  we equip $\mathcal{A}_\K$ with the ring structure $\ast_{\operatorname{SR}}$ from \eqref{eq: ringstructure}. The following is a special case of \cite[Theorem 1.1]{GP2}:
 
\begin{thm} \label{lem:GP2lemma} The map \eqref{eq:Speciso2} is an isomorphism of rings.  \end{thm} 

 
\subsection{Degree zero $SH^*$ relative to a (stabilizing) divisor} \label{subsection:relativeFloer}

As mentioned in the introduction,  we will use the approach of stabilizing divisors to achieve regularity \cite{CieliebakMohnke}.  This approach is somewhat \emph{ad hoc} as it involves auxillary choices compared to the more canonical virtual techniques, but avoids having to prove analytically involved gluing theorems dictating how the different strata of the moduli space fit together.  The basic definitions are the following: 

\begin{defn} A stabilizing divisor for the pair $(M,\D)$ is another ample normal crossings divisor $\E=\cup_j E_j$, $j \in \{1,\cdots,k' \}$
 which meets $\D$ transversely,  so that $\D \cup \E$ is also a normal crossings divisor. \end{defn}

\begin{defn} \label{defn:stabdata} A stabilization datum $(M,\D)$ consists of a stabilizing divisor $\E=\cup_j E_j$ together with a
 a regularization for the pair $(M,\D \cup \E)$ such that for generic $J \in \mathcal{J}(M,\D \cup \E)$, and any subset $K \subset \{1, \ldots, k' \}$, there are no non-constant $J$-holomorphic curves $u: \mathbb{C}P^1 \to E_K$ which intersect $\cup_{j \notin K} (E_j \cap E_K)$ in two or fewer distinct points.
\end{defn} 

Note that as per our usual convention, we include the case $I = \emptyset$ above,
with $D_{\emptyset} = M$.  It is easy to construct stabilized regularizations by taking a $\E$ to be union of $2n+1$ generic divisors in the linear system of $\mathcal{L}^{\otimes q}$ for some sufficiently large $q$.  The basic point behind the definition is that any non-constant curves contained in $\E$ will be stable and thus we will be able to use surface dependent perturbations on these components to regularize them.  

We now introduce ``$\E$-marked" (=relative to $\E$) versions of our Floer theoretic constructions.  Observe that a regularization for $(M,\D \cup \E)$ gives one a regularization for $(M,\D).$ Whenever we have a stabilizing divisor, we can assume,  after changing possibly changing the regularization of $(M,\D)$, that the chosen regularization is compatible with the one used to define $SH^*(X)$ (Liouville domains $\bar{X}_{m}$, Hamiltonian functions $h^{m}$ etc.). We may also arrange it so that in tubular neighborhood $U_{I,K}$ of $D_I \cap E_K$ we have the Liouville one form $\theta$ restricts to the fiber of $\pi_{I,K}: U_{I,K} \to D_I \cap E_K$ as: \begin{align} \label{eq:reltheta}   \sum_{i \in I} (\frac{1}{2}\rho_i - \frac{\kappa_i}{ 2\pi}) d\varphi_i + \sum_{j \in J} \frac{1}{2}\rho_j d\varphi_j \end{align} 

 We will do all of this going forward.  Fix one of these Liouville manifolds $\bar{X}_{\vec{\epsilon}}$ and consider the set of complex structures:  \begin{align} \label{eq:complexrelE} \mathcal{J}( \bar{X}_{\vec{\epsilon}},\E) := \mathcal{J}( \bar{X}_{\vec{\epsilon}},V) \cap \mathcal{J}(M,\D \cup \E) \end{align}

By construction, the Reeb flow on $\partial \bar{X}_{\vec{\epsilon}}$ preserves $\E$ and it follows from equation \eqref{eq:reltheta} that the Liouville flow also preserves $\E.$ It is therefore not difficult to construct to show that $\mathcal{J}( \bar{X}_{\vec{\epsilon}},\E)$ is non-empty.  For a given admissible $h^{\lambda}$(of fixed slope $\lambda$), the flow preserves $\E.$ We almost must consider our perturbed functions $H^\lambda$. We need to choose them so that they have the following properties (in addition to those needed in Section 2): \vskip 5 pt

\emph{Properties of $H^{\lambda}$}: 
\begin{enumerate}[label=(\alph*)] \item \label{item:hama} $X_{H^{\lambda}}$ preserves $E_j$ for all $j$. \item \label{item:hamb} All Hamiltonian orbits $x_0$ which lie in any $E_j \cap X$ have $\deg(x_0) \geq 2$.  \end{enumerate}

\vskip  5 pt

Property \ref{item:hama} will be needed so that Floer trajectories not entirely contained in $\E$ intersect each divisor $E_j$ with positive multiplicity.  Our constructions in this paper will only involve orbits of index $\leq 1$ (because we're studying degree zero $SH^*(X)$).  Property \ref{item:hamb} ensures that all of the relevant Hamiltonian orbits are disjoint from $\E$,  meaning that none of the Floer curves that connect them can be entirely contained in $\E$.  To explain why we can choose Hamiltonians with these properties,  we need to recall a bit more of the perturbations from \cite[Section 4.1]{GP2}: \vskip 10 pt

\textbf{Isolating neighborhoods:} We first specify the isolating neighborhoods $U_\v$ of the critical sets $\mathcal{F}_\v.$ For nonzero values of $\v$, the orbit sets $\mathcal{F}_\v$  occur at points $U_I$ where $\rho_i= \rho_{i,\v}$ for some $\rho_{i,\v} \in \mathbb{R}$ and $i\in I$ and have boundary and
corner strata where $\rho_i=\rho^{c}_{\v,i}$ for some $\rho^{c}_{\v,i} \in \mathbb{R}$ and $i \notin I$. We let $c_0$ be a sufficiently small positive real number and let $D^{c_0}_\v$ denote the open manifold $D^{c_0}_{\v} := D_{|\v|} \setminus \cup_{i \notin |\v|} (U_{i,\rho^{c}_{\v,i}-c_0} \cap D_{|\v|})$ (where the sets $U_{i,\rho^{c}_{\v,i}-c_0}$ are defined just below Theorem \ref{thm:niceprimitive}). Let $S^{c_0}_{\v}$
denote the induced $T^{|\v|}$ bundle over $D^{c_0}_\v$ where
$\rho_i=\rho^{c}_{\v,i}$. We set the isolating sets to be the neighborhoods $U_{\v} \subset U_I$ such that
$\pi_I(x) \in D^{c_0}_{\v}$  and $\rho_{i,\v}-c_0 <\rho_i <\rho_{i,\v}+c_0, \ i
\in I$.  For $\mathcal{F}_{\mathbf{0}}$,
choose our isolating neighborhood $U_{\mathbf{0}}$ to be the complement of neighborhood where 
\[U_{\mathbf{0}}:= M \setminus \lbrace R^{\vec{\epsilon}} \geq R_{0}+c_0 \rbrace \] After possibly shrinking $c_0$, these neighborhoods are pairwise disjoint, i.e. $U_\v \cap U_{\v'}=\emptyset$ for $\v \neq \v'$. We
let $U_{\v}'$ to be slightly smaller subsets such that $U_{\v}' \subset U_{\v}$
which are of the same form (to construct them just take choose a constant
$c_0'$ which is slightly smaller than $c_0$). 
\vskip 10 pt

\textbf{Hamiltonian perturbations:} All of the $\mathcal{F}_\v$ are Morse-Bott in their interiors $\mathcal{F}_\v \setminus \partial \mathcal{F}_\v$ (see Step 2 of the proof of Theorem 5.16 of \cite{McLean2})
and thus Morse-Bott type perturbations are required to make
the orbits non-degenerate. These perturbations all take place locally in each isolating set $U_\v.$ When $\v=0$, we will perturb by a Morse function
$h_{\mathbf{0}}$  which near $R^{\vec{\epsilon}}=R_{0}$ is a
function of $R^{\vec{\epsilon}}$ with positive derivative.  To carry this out when $\v \neq 0$, it is convenient to recall the ``spinning" construction (see e.g. \cite[Proof of Prop. B.4]{KoertKwon} for a nice reference). This enables us to (locally in each $U_\v$) convert our Hamiltonian system to an equivalent one where the orbits in $U_\v$ are constant (after which we can then use a suitable Morse function to perturb the flow).  For
non-constant orbits, observe that on all of the orbit sets $\mathcal{F}_\v$, the Reeb flow generates an $S^1$-action on
$\mathcal{F}_\v$ which extends canonically to an $S^1$ action on $U_\v$. This action is Hamiltonian, with some Hamiltonian $K_{S^{1}}$, and 
we let $g_t$ denote the flow of the Hamiltonian vector field $X_{K_{S^{1}}}$ on $U_\v$. On $U_\v$, the function \begin{align} \label{eq:hamunwind} \hat{h}^{\lambda} := h^{\lambda}_{|U_{\v}} -K \end{align} 

has constant Hamiltonian orbits along $\mathcal{F}_\v$.  Next, choose a Morse function $\hat{h}_\v^{base}: S^{c_0}_{\v}  \to \mathbb{R}$ such that near the corners the function $\hat{h}_I$ point outwards along the boundary and let $\hat{h}_\v$ denote the pull-back of $\hat{h}_\v^{base}$ to $U_\v$. Finally, we choose cutoff functions $\rho_\v$ supported in $U_\v$ (for simplicity we take to be $S^1$ equivariant in $U_\v$) such that
\begin{itemize} 
    \item $\rho_\v(x)=0, x\in M\setminus U_{\v}$  
    \item $\rho_\v(x)=1, x \in U_{\v}'$ 
\end{itemize}  Set $$ \hat{H}^{\lambda}:= \delta_\v \rho_\v \hat{h}_\v +\hat{h}^\lambda$$

Then for $\delta_\v$ sufficiently small, vall of the orbits of $\hat{H}^{\lambda}$ are in bijection with critical points of $\hat{h}_\v$ along $\mathcal{F}_\v$ (see \cite[Lemma 4.2]{GP2} based on \cite[Proof of Prop. B.4]{KoertKwon}). One can now translate this perturbation back to our original Hamiltonian system by setting $H^{\lambda}: S^1 \times U_\v \to \mathbb{R}$ to be the function \begin{align} \label{eq:pertuv} H^{\lambda}:= \delta_\v \rho_\v h_\v + h^\lambda \end{align} where $h_\v: S^1 \times U_\v \to \mathbb{R}$ denotes the time dependent function given by $h_\v:=\hat{h}_\v(t,g^{-1}_t(x))$ (the ``spinning of $\hat{h}_\v$"). Orbits of $H^{\lambda}$ are then non-degenerate perturbations of the initial orbits in $\mathcal{F}_\v$.  Moreover, if $\hat{J}_t:=dg_t^{-1} \circ J(g_tx) \circ dg_t$, the assignment \begin{align} \label{eq:Seidel} u \to g_t \circ u \end{align}

takes transverse Floer trajectories $u: \mathbb{R} \times S^1 \to U_\v$ for the pair $(\hat{H}^{\lambda},\hat{J}_t)$ to transverse Floer trajectories for the pair $(H^{\lambda}, J_t)$ (see \cite[Lemma 4.2]{MR1487754} for a much more general result). 

To globalize this,  notice that a similar construction (spelled out in \cite[\S 4.1]{GP1}) yields analogous functions $h_\D$, $\rho_\D$ supported near $\D$ which can be used to perturb the divisorial orbit sets. For sufficiently small constants $\delta_\v$ and $\delta_\D$ define:  
\begin{align} 
    \label{eq: Hpert2} H^{\lambda}=  \sum_ \v \delta_\v \rho_\v h_\v + h^\lambda +\delta_\D  \rho_\D h_\D 
\end{align} 

The crucial property for us that the function $H^{\lambda}$ is non-degenerate and that the orbits which lie in $U_\v$ are in bijection with critical points of the corresponding $\hat{h}_\v^{base}$.  

\begin{lem} \label{lem: keyhamlem} There exist perturbed Hamiltonians $H^{\lambda}$ satisfying Properties \ref{item:hama}, \ref{item:hamb} above. \end{lem}  
\begin{proof}  If the Hamiltonian flow of the functions $h_\v$,  $h_\D$,  $\rho_\v$,   $\rho_\D$ appearing in Equation \eqref{eq: Hpert2} all preserve the divisors $E_i$,  then $X_{H^{\lambda}}$ preserves $E_i$.  It is easy to construct the last three functions with the desired property.  However,  more care is needed in the construction of $h_\v$ to ensure that Property \ref{item:hamb} also holds.  

The Hamiltonian flow of $h_\v$ preserves $E_i$ if the Hamiltonian flow of $\hat{h}_\v$,  the pull back of the Morse function $\hat{h}_\v^{base}$, preserves $E_i. $ As noted above,  the orbits of the perturbed Hamiltonian which lie in $U_\v$ are in bijection with critical points of $\hat{h}_\v^{base}$.  Moreover, the degree of the corresponding Hamiltonian orbit is simply $\deg(x_{crit}) + 2\sum_i (1-a_i)$, where $\deg(x_{crit})$ denotes the Morse index of the corresponding critical point. Because we are assuming our pair is log nef,  it suffices to construct a Morse function on $\mathcal{F}_\v$ which preserves $\E \cap \mathcal{F}_\v$ and such that all critical points which lie in $\E \cap \mathcal{F}_\v$ have index at least two. The existence of such a function follows from the usual Thom isomorphism in Morse theory--- namely, it suffices to assume that near $E_J \cap \mathcal{F}_\v$ the function $\hat{h}$ has the form $\sum_{j \in J} G(\rho_j) + \pi_J^*\tilde{h}$, where $G$ is a function with a maximum at $\rho_j=0$ and $\tilde{h}$ is a function on $E_J \cap \mathcal{F}_\v$.
   \end{proof}

\begin{defn} \label{eq:FloerRinv2} Here we record some variants of the moduli spaces defined in \eqref{eq:FloerRinv}.
\begin{itemize} \item For pairs of capped orbits $\tilde{x}_0:= (x_0, u_{x_{0}})$, $\tilde{x}_1:= (x_1, u_{x_{1}}) \in  \mathcal{X}(X; \Hla)$, the moduli space of Floer trajectories $\mathcal{M}(\tilde{x}_0,\tilde{x}_1)$ is the space of solutions to \eqref{eq:FloerRinv} satisfying $$ u \# u_{x_{0}} \simeq u_{x_{1}} $$
considered up to reparameterization. 
\item Similarly, we let $\mathcal{M}(x_0, x_1,[u_{x_{0},x_{1}}])$ denote the moduli space of solutions in a given relative homology class $[u_{x_{0},x_{1}}]$. 
\end{itemize} 
\end{defn}

Consider a pair of orbits $x_0,x_1 \in (X \setminus \E)$ and note that elements of $\mathcal{M}(x_0,x_1,[u_{x_{0},x_{1}}])$ have a well-defined total intersection number \begin{align} \label{eq:Floerintersection} [u_{x_{0},x_{1}}] \cdot E_j \end{align} with each divisor $E_j$. Suppose that we are given a Floer trajectory $u: (C,z) \to X$ connecting orbits together with a point $z \in C$ such $u(z) \in E_j$. Then the identification \eqref{eq:Gromovtricky}(or rather its obvious variant for Floer trajectories) allows one to define a positive intersection multiplicity, $\operatorname{ord}_j(z)$ of $u$ with $E_j$ at the point $z$. The intersection number \eqref{eq:Floerintersection} is then the sum of all of these local multiplicities and is hence non-negative.  We next turn to defining $\E$-marked Floer moduli spaces $u: (C, \underline{z}) \to X$, which are roughly speaking solutions to \eqref{eq:modulispacetraj} where we allow our curves to pass through the divisors $E_i$ at additional marked points $\underline{z}$. 

\begin{defn} Let $F$ be a finite set and $e_j$, $j \in \lbrace 1,\cdots, k_F \rbrace$ a collection of positive integers such that $\sum_j e_j= |F|$. An $\lbrace e_j \rbrace$-labelling function is  an assignment $b: F \to \lbrace 1,\cdots k_F \rbrace$ which assigns to $j$ exactly $e_j$ elements of $F$. \end{defn}

For the following definition, we fix two capped orbits $x_0,x_1$ in $X \setminus \E$ as well as a relative homology class $[u_{x_{0},x_{1}}].$ We abbreviate  \eqref{eq:Floerintersection} by $e_j$ and let \begin{align} r([u_{x_0,x_1}]) = \sum_j e_j. \end{align} We further set $F:= \lbrace 1,\cdots, r([u_{x_0,x_1}])  \rbrace$,\footnote{By convention $F=\emptyset$ if $r([u_{x_0,x_1}]) =0.$} and $b$ to denote an $\lbrace e_j \rbrace$-labelling function $b: F \to \lbrace 1,\cdots k' \rbrace$. 

\begin{defn} Choose a time dependent almost complex structure $J_t$ such that each $J_t$ lies in \eqref{eq:complexrelE}. Set $\mathcal{M}_\E(x_0,x_1,[u_{x_{0},x_{1}}],b)$ to be the moduli space of Floer curves $u: (C, \underline{z}_f) \to X$ (as usual considered up to $\mathbb{R}$-translation) with marked points indexed by $f \in F$ such that \[ u(z_f) \in E_{b(f)}\] with $\operatorname{ord}_{b(f)}(z_f)=1$ (e.g. the intersection is transverse). We also let $$\mathcal{M}_\E(x_0,x_1,[u_{x_{0},x_{1}}]):=\bigsqcup_b \mathcal{M}_\E(x_0,x_1,[u_{x_{0},x_{1}}],b)$$
\end{defn} 

\begin{lem} Let  $x_0,x_1$ be two orbits with $\deg(x_0)=0$, $\deg(x_1)=1$ (hence lie in $X\setminus E$). and let $[u_{x_0,x_1}]$ be a relative homology class. Then for a generic choice of $J_t$, $\mathcal{M}_\E(x_0,x_1,[u_{x_{0},x_{1}}])$ is compact. \end{lem} 

\begin{proof} There are essentially two things that might potentially lead the moduli spaces to be non-compact:  \vskip 5 pt

\emph{Cylindrical breaking:} To rule out cylindrical breaking, we consider two cases. If there is no breaking along $\E$, this follows from standard transversality. If there is breaking in $\E$, there must be a component going from an orbit in $\E$(which has to degree at least two) to an orbit of degree at most one in $X$. Because such a curve cannot lie entirely in $\E$(because the outgoing asymptote lies in $X$), we can still obtain transversality by perturbing $J_t$. \vskip 5 pt 

\emph{Marked points colliding:} Similarly, it is not difficult to see that marked points collide generically in codimension two, which means they do not appear(again assuming $J_t$ generic) as limits of our one dimensional moduli spaces.   \end{proof}

Define a map \begin{align} \partial_{E}: CF^0(X \subset M,H^{\lambda}) \to CF^1(X \subset M, H^{\lambda}) \end{align}
by the formula \begin{equation} 
    \partial_{CF}(z)  =\sum_{x_{0}} \sum_{[u_{x_{0},x_{1}}]} \sum_{u \in \mathcal{M}_\E(x_0,x_1,[u_{x_{0},x_{1}}])} \frac{1}{ r([u_{x_{0},x_{1}}])!} \mu_u 
\end{equation} 
where $z \in |\mathfrak{o}_{x_{1}}|$ and as before $\mu_u$ is the isomorphism of orientation lines assigned to each $u$ by the theory of coherent orientations. Define \begin{align} \label{eq:relFloerdif} HF_{\E}^0(X \subset M, H^{\lambda}):= \operatorname{ker}(\partial_E) \end{align}  to be the kernel of this map.  Note that, in the present situation, the signed counts of curves in $\mathcal{M}_\E(x_0,x_1,[u_{x_{0},x_{1}}],b)$ are manifestly independent of $b$, and hence there is some redundancy in \eqref{eq:relFloerdif}. However, the present formulation will be convenient in the coming sections. 

We can similarly define continuation maps between the groups  $HF_{\E}^0(X \subset M, H^{m})$.  We then define the (degree zero) $\E$-marked symplectic cohomology to be the direct limit: \begin{align} SH^0_{\E}(X) := \varinjlim HF_{\E}^0(X \subset M, H^{m})  \end{align} 

\begin{lem} \label{lem:caniso} We have that  \begin{align} HF_{\E}^0(X \subset M, H^{\lambda}) \cong HF^0(X \subset M, H^{\lambda}) \end{align} \begin{align} \label{eq:obviso} SH^0_{\E}(X) \cong SH^0(X, \K) \end{align} \end{lem}
\begin{proof} As noted above, all of the moduli spaces $\mathcal{M}_\E(x_0,x_1,[u_{x_{0},x_{1}}],b)$ are identified for different $b$ and are in fact isomorphic to $\mathcal{M}(x_0,x_1,[u_{x_{0},x_{1}}])$.Thus the differentials and continuation maps are identified and the result follows. \end{proof} 

Having defined an $\E$-marked version of Floer cohomology,  we are now in position to define $\PSS$ moduli spaces with $\E$-markings.  First,  the requisite combinatorial definitions:  

\begin{defn} An $r$-marked PSS tree $\underline{\Gamma}_{PSS}^{(r)}$ is a rooted tree with a distinguished leg $l_0$ and a collection of additional legs indexed by $\lbrace 2,\cdots , r+1 \rbrace$\footnote{The reason for this choice of indexing will be apparent in Definition \ref{defn:PSSstab}.} so that for any vertex $\nu \neq \nu_{root}$,  \begin{align} |\mathbb{E}(\nu)| + |\mathbb{L}(\nu)| \geq 3. \end{align} Here,  $\mathbb{E}(\nu)$ denotes the set of internal edges bounding $\nu$ and $\mathbb{L}(\nu)$ denotes the set of legs bounding $\nu$.   \end{defn} 
Analogously to Definition \ref{defn:bubblePSS}, we can define bubbled $\PSS$ curves with $r+1$-marked points which are modelled on some $\underline{\Gamma}_{PSS}^{(r)}.$ If $C$ is such a marked PSS solution,  we let $\bar{C}$ be the domain given by compactifying the thimble domain $S$ to $\mathbb{C}P^1$.  When dealing with marked $\PSS$-solutions,  we will use perturbations that depend on the modulus of $\bar{C}$ (this will be needed to achieve transverality for stable log $\PSS$-moduli spaces in the presence of sphere bubbles).  Given an integer $r \geq 1$, let $$\bar{\mathcal{C}}_{0,r+2} \to \bar{\mathcal{M}}_{0,r+2}$$ denote the universal curve over the moduli space of stable genus zero curves with $r+2$ marked points. For the next definition, fix some background $J_S$ which as usual  is surface independent, i.e. $J_z$ agrees with some fixed almost complex structure $J_0$, in some neighborhood $V_S$ of $z=\infty$. We let $U_S$ be some open, pre-compact subset of $V_S \cap \operatorname{supp}(\beta).$ 

\begin{defn} Equip $TM$ with the almost complex structure $J_0$.  A ($J_S$-compatible, for some fixed $r \geq 1$)  perturbation datum will be a one-form \begin{align} \label{eq: ruantianform} \upsilon \in C^{\infty}(\bar{\mathcal{C}}_{0,r+2} \times M_S, \Omega_{\bar{\mathcal{C}}_{0,r+2}/\bar{\mathcal{M}}_{0,r+2}} ^{0,1} \otimes_\mathbb{C} TM) \end{align} 
which is supported in $\bar{\mathcal{C}}_{0,r+2} \times U_S \times M$ and away from marked points and nodes. 
\end{defn}

We will assume that our perturbation data are coherent with respect to boundary strata in the sense of \cite[Section 3]{CieliebakMohnke} or \cite[Chapter 9]{Seidel_PL}. Coherent perturbation data exist by a gluing argument similar to Lemma 9.5 of \cite{Seidel_PL}.  For any capped orbit $\tilde{x}_0$ with $x_0 \in X \setminus \E$,  any $\PSS$ solution $u \in \mathcal{M}(\tilde{x}_0)$ has (positive) intersection $[u_{x_{0}}]\cdot E_j$ with each divisor $E_j.$ We set 
\begin{align} r(\tilde{x}_0) = \sum_j [u_{x_{0}}]\cdot E_j  \end{align}

\begin{defn} \label{defn:PSSstab} Let $\tilde{x}_0:=(x_0,[u_{x_{0}}])$ be a capped orbit with $x_0 \in X \setminus \E$ and $r(\tilde{x}_0) \geq 1.$ Let $\underline{\Gamma}_{PSS}^{(r(\tilde{x}_0))}$ be a marked PSS tree.  A perturbed bubbled $\PSS$-solution asymptotic to $\tilde{x}_0$ and modelled on $\underline{\Gamma}_{PSS}^{(r(\tilde{x}_0))}$  is a domain $C$ modelled on $\underline{\Gamma}_{PSS}^{(r(\tilde{x}_0))}$   together with:
\begin{itemize} \item A holomorphic map $\bar{C}$ onto a fiber $\tau: \bar{C} \hookrightarrow \bar{\mathcal{C}}_{0,r+2}$ which: \begin{itemize} \item is either 1) an isomorphism or 2) contracts the root vertex and is an isomorphism away from the root vertex.  \item sends the distinguished marked point corresponding to     $l_0$ to $z_0$ and the marked point corresponding to the output orbit to $z_1.$ \end{itemize} 

 \item A map $\tilde{u}: C \to M_S:= S \times M$ satisfying \begin{align} (d\tilde{u})^{0,1}-(\tau, \tilde{u})^*\upsilon=0 \end{align} such that the projection $\pi_S \circ \tilde{u}$ is a bi-holomorphism at the root vertex and constant for the remaining vertices with $\pi_S \circ \tilde{u}(z_{l_{0}})=\lbrace \infty \rbrace$.  Here $(d\tilde{u})^{0,1}$ is understood with respect to the almost-complex structure \eqref{eq:Gromovcs}.

\end{itemize} 
 \end{defn} 

We denote the remaining marked points on $C$ by $z_q, q \in \lbrace 2,\cdots r+1 \rbrace.$ We now turn to giving the definitions of log versions of the perturbed moduli spaces. The definitions are very similar to those in \S \ref{section:stablePSS} and we describe the necessary modifications:  \vskip 5 pt

\textbf{Contact functions and log trees:} \vskip 5 pt

 Let $\hat{\mathcal{P}}$ denote the collection of subsets of $\lbrace 1,\cdots, k \rbrace \cup \lbrace 1,\cdots, k' \rbrace.$ Given an $r$-marked PSS tree $\underline{\Gamma}_{PSS}^{(r)}$:

 \begin{itemize} \item
  a stablized depth function is an assignment  \[
     I^\nu:  \mathbb{E}(\underline{\Gamma}_{PSS}^{(r)}) \to \hat{\mathcal{P}},\\
     I^e: \mathbb{E}(\underline{\Gamma}_{PSS}^{(r)})  \to \hat{\mathcal{P}}
 \]  
such that $I^{\nu_{root}}=\emptyset$ and for any $e$ which is bounded by vertices $\nu, \nu'$, $I^\nu \cup I^{\nu'} = I^e.$
 \item  a contact function is a map \begin{align} \v(-) : \vec{\mathbb{E}}(\underline{\Gamma}_{PSS}^{(r)}) \to \mathbb{Z}^{k} \times \mathbb{Z}^{k'} \end{align} 
such that for any pair of neighboring vertices $\nu,\nu'$ \eqref{eq:ordersmatch} holds and the support $|\v(e_{\nu,\nu'})| \subset I^{e_{\nu,\nu'}}$(here we mean the depth function of the edge underlying the oriented edge).
\end{itemize}

\begin{defn} \label{defn: stablizedPSStree} (Compare Definition \ref{defn:PSStrees}) An  r-marked log $\PSS$ tree $\Gammas=(\GammaPSS^{(r)}, I^{\bullet}, \v(-))$ is a triple consisting of \begin{itemize} \item an $r$-marked $\PSS$ tree $\GammaPSS^{(r)}$ so that if we remove all of the legs attached to $\nu_{root}$(including $\nu_{root}$ itself), the result is connected.  \item a depth function $I^\nu, I^e.$ \item and a contact function $\v(-)$. \end{itemize} \end{defn} \vskip 5 pt


\textbf{Floer data:} \vskip 5 pt

We choose our almost complex structures $J_S$ as follows: 
\begin{itemize} 
\item   for each $z \in S$, $J_z \in \mathcal{J}(\bar{X}^{\vec{\epsilon}},\E)$ (recall \eqref{eq:complexrelE}). 
\item $J_z=J_0$ is surface independent (i.e. $J_z$ agrees with some fixed almost complex structure $J_0$) in some neighborhood $V_S$ of $z=\infty$ for some $J_0 \in AK(M,\D \cup \E).$ 
\end{itemize} 

For log moduli spaces, we restrict to $\upsilon$ which are compatible with the regularization in the sense of \cite[Definition 3.7]{Tehrani2}.  This compatibility is easiest to explain using a differential-geometric version of the logarithmic tangent bundle $TM(-\operatorname{log}(\D))$,  which depends on a choice of a regularization (see \cite[Eq. (2.11)]{Tehrani2}).  For our purposes, we only need recall that if $\pi_I:U_I \to D_I$ is a regularizing tubular neighborhood, then over $U_I^o:=\pi_I^{-1}(D_I\setminus \cup_{j\notin I} D_j)$,  we have that there is a decomposition \begin{align} \label{eq: logdecomp} TM(-\operatorname{log}(\D))_{|U_{I}^o} \cong \pi_I^*(T_{D_{I}}) \oplus \mathbb{C}^{I} \end{align} As one would expect, $TM(-\operatorname{log}(\D))$ equipped with a canonical map $$ TM(-\operatorname{log}(\D)) \to TM $$ which is an isomorphism away from $\D.$ To describe this map over $U_I^o$, we identify (using the regularization) points $\hat{x} \in U_I^o$ with tuples $(x, (\mathbf{n}_i)_{i \in I}$ where $x$ is a point in $D_I$ and $(\mathbf{n}_i)_{i \in I} \in ND_{i,x}$ is a tuple of normal vectors.  We then set: 
$$  (v_{D_I}, (\mathbf{c}_i)_{i \in I})_{(x, (\mathbf{n}_i)_{i \in I})} \to (v_{D_I}, \sum_i c_i \mathbf{n}_i )   $$ 
where $v_{D_I}$ denotes a tangent vector in $TD_{I,x}$ and the right hand side uses the decomposition  \eqref{eq:complexdecomp} given by the regularizing connections. 
\begin{defn} 
A perturbing one-form $\upsilon$ is compatible with the regularization if: \begin{itemize} \item $\upsilon$ lifts to a section $$ \upsilon_{log} \in C^{\infty}(\bar{\mathcal{C}}_{0,r+2} \times M_S, \Omega_{\bar{\mathcal{C}}_{0,r+2}/\bar{\mathcal{M}}_{0,r+2}} ^{0,1} \otimes_\mathbb{C} TM(-\operatorname{log}(\D))$$  
With respect to the decomposition, \eqref{eq: logdecomp}, $$ \upsilon:= \pi_I^*(\upsilon_I) \oplus \pi_I^*((\upsilon_{F,i})_{i \in I}) $$ 

where in the horizontal component, $\upsilon_I \in  C^{\infty}(\bar{\mathcal{C}}_{0,r+2} \times M_S, \Omega_{\bar{\mathcal{C}}_{0,r+2}/\bar{\mathcal{M}}_{0,r+2}} ^{0,1} \otimes_\mathbb{C} TD_I)$ is a suitable $(0,1)$ form on $D_I$, and on the vertical component, $(\upsilon_{F,i})_{i \in I} \in  C^{\infty}(\bar{\mathcal{C}}_{0,r+2} \times M_S, \Omega_{\bar{\mathcal{C}}_{0,r+2}/\bar{\mathcal{M}}_{0,r+2}} ^{0,1} \otimes_\mathbb{C} \mathbb{C}^I)$ is a tuple of one forms on $\mathbb{C}.$
\end{itemize} 
\end{defn}
Let $u:\mathbb{CP}^1 \to M$ be a component of a bubbled perturbed PSS solution which attaches at $\lbrace \infty \rbrace \in \mathbb{C}P^1 \setminus \lbrace 0 \rbrace$. As in \eqref{eq:delbar1}, if $u \subset D_i$, then $D_u(\bar{\partial}-\upsilon)$ descends to a modified $\bar{\partial}$-operator:
\begin{align} \label{eq:delbar2} D_u^{ND_i}(\bar{\partial}-\upsilon): \Gamma(\Sigma, u^*ND_i) \to  \Gamma(\Sigma, \Omega_{\Sigma,\mathfrak{j}}^{0,1} \otimes_\mathbb{C} u^* ND_i) \end{align}

 which can be used to be define meromorphic sections. \vskip 5 pt

\textbf{Pre-log and log moduli spaces:} \vskip 5 pt

 Let $b$ be a labelling $b: \lbrace 2,  \cdots, r+1 \rbrace \to \lbrace 1, \cdots,  k' \rbrace.$ Because we have an induced $\bar{\partial}$-operator on the normal bundle of curves contained in divisor strata \eqref{eq:delbar2}, we can define $\E$-marked pre-log maps $$(C_\nu,u_\nu, [\zeta]_\nu)_{\nu \in \VGPSS}  \in \mathcal{M}^{\operatorname{plog}}(\v,\Gammas, \tilde{x}_0,b)$$ by directly mimicking Definition \ref{defn:prelogmods}. To this end, let $\v_{E_{j}} \in \mathbb{Z}^{k} \times \mathbb{Z}^{k'}$ denote the vector: $$\v_{E_{j}} =: \vec{0} \times \underbracket{(0,\cdots, 1, \cdots 0)}_{\text{1 in j-th position and 0 elsewhere}} $$ In addiiton to the obvious modifications of Definition \ref{defn:prelogmods}, we additionally require that at $z_q$, $q \in \lbrace 2,\cdots,r+1 \rbrace$, $ u(z_q) \in E_{b(q)}$ and $$\operatorname{ord}_u(z_q)=\v_{E_{b(q)}}. $$ Just as in \eqref{eq:obstruction20},  after choosing an arbitrary orientation $\tilde{e}:=e_{\nu,\nu'}$ on each edge $e$,  there is a well-defined map \begin{align} ob_{\Gammas}: \mathcal{M}^{\operatorname{plog}}(\v,\Gammas, \tilde{x}_0,b) \to \mathcal{G}(\Gammas) \end{align} 

\begin{defn}  \label{defn:stabilizedlogmap}(Compare Definition \ref{defn:logmap}) A $\E$-marked log map (modelled on $\Gammas$) consists of a pre-log map $$(C_\nu,u_\nu, [\zeta]_\nu)_{\nu \in \VGPSS} \in \mathcal{M}^{\operatorname{plog}}(\v, \Gammas, \tilde{x}_0,b)$$ satisfying the following additional conditions: \begin{enumerate}[label=(\roman*)] \item There exists functions $\v_\nu: \mathbf{V}(\GammaPSS^{r(\tilde{x}_0)}) \to \mathbb{R}^{k+k'}, \lambda_e: \mathbf{E}(\GammaPSS^{r(\tilde{x}_0)})  \to \mathbb{R}^{+}$ such that \begin{enumerate}[label=(\alph*)] \item $\v_\nu \in (\mathbb{R}^+)^{I^{\nu}} \times \lbrace 0 \rbrace^{k+k'-|I^{\nu}|}$ \item For every pair of vertices $\nu, \nu'$ connected by an (oriented) edge $e_{\nu,\nu'}$,  $$\v_{\nu}-\v_{\nu'}=\lambda_e(e_{\nu,\nu'}) \v(e_{\nu,\nu'}).$$ \end{enumerate} \item We have that  \begin{align} \label{eq:obstruction2} ob_{\Gammas}(u)=[1] \in \mathcal{G}(\Gammas) \end{align}  \end{enumerate} We let $\mathcal{M}(\v, \Gammas, \tilde{x}_0,b)$ denote the moduli space of $\E$-marked log maps (modelled on $\Gammas$) up to isomorphism. \end{defn} When $\Gammas$ has a single vertex we use the notation,  $\mathcal{M}_{\E}(\v,\tilde{x}_0,b):=\mathcal{M}(\v, \Gammas, \tilde{x}_0,b).$ Finally, we set $$\mathcal{M}_{\E}(\v,\tilde{x}_0):=\bigsqcup_b \mathcal{M}_{\E}(\v,\tilde{x}_0,b)$$  \vskip 5 pt

The important point about $\E$-marked moduli spaces is that all of the spherical components must be stable,  meaning they have at least three special points.  Hence, we can take the additional perturbations $\upsilon$ to be non-trivial over these components providing enough additional flexibility to regularize them(as noted above,  the Floer data over the component $u_{root}$ the $\PSS$ component itself) --- in the present situation meaning that these moduli spaces are in fact generically empty when there are sphere bubbles. The precise statement is the following: 

\begin{lem} \label{eq:ruantransv} For generic $(J_S, \upsilon)$, the moduli spaces $\mathcal{M}(\v,\Gammas, \tilde{x}_0, b)$ are empty when $\degr(x_0) \leq 1$ and the graph underlying $\Gammas$ has more than a single vertex. \end{lem}
\begin{proof} This is proved as in \cite[Theorem 1.5]{Tehrani2} whose proof is in turn broadly similar to \cite[Proposition 3.16]{RuanTian} and \cite[Theorem 6.2.6]{McDuff:2004aa}. Indeed, the arguments from these latter two references adapt easily to prove that a certain universal moduli space of (parameterized) pre-log maps $$\tilde{\mathcal{M}}_{\operatorname{univ}}^{\operatorname{plog}}(\v,\Gammas, \tilde{x}_0, b)$$ is a Banach manifold. The new feature in the present setting is that one must prove that the obstruction map $$ ob_{\Gammas}: \tilde{\mathcal{M}}_{\operatorname{univ}}^{\operatorname{plog}}(\v,\Gammas, \tilde{x}_0, b) \to \mathcal{G}(\Gammas)$$ can be made transverse as well, meaning that for any $u=(C_\nu,u_\nu, [\zeta]_\nu)_{\nu \in \VGPSS}$, \begin{align} \label{eq:Du} \operatorname{D}ob_{\Gammas}: T_u\tilde{\mathcal{M}}_{\operatorname{univ}}^{\operatorname{plog}}(\v,\Gammas, \tilde{x}_0, b) \to T_{1}\mathcal{G}(\Gammas) \end{align} is surjective. This is a consequence of the following argument from \cite[Claim 2,  Proof of Theorem 1.5]{Tehrani2} which we summarize.  Choose representatives $\zeta_\nu$ for the meromorphic sections and coordinate charts near all of the marked points. Next, fix an edge $e$ and let $\tilde{e}:=e_{\nu,\nu'}$ be the orientation chosen on this edge. (In the case that one of the vertices is $\nu_{root}$, we will assume that $\nu_{root}=\nu'$ to make the argument written below apply uniformly.) For every $i \in I^e$, we will show that $[1_{e,i}] \in T_{1}\mathcal{G}(\Gammas)$ is the image of  $\eqref{eq:Du}$. There are two cases to consider: \vskip 5 pt

$i \in I^\nu$: (c.f.  \cite[Claim 2, Case (i) in proof of Theorem 1.5]{Tehrani2})Let $\mathbb{C}P^1_\nu$ denote the domain of $u_\nu$ and let $\beta: \mathbb{C}P^1_\nu \to \mathbb{C} $ be a function supported in a neighborhood of $z_{e_{\nu,\nu'}}$ and which is constant in a smaller neighborhood of this marked point. We choose $\beta$ so that $$ \frac{d}{dt}(e^{t\beta(z_{\tilde{e}})}\eta_{e,i})_{|t=0}=1$$ where $\eta_{e,i}$ is defined in \eqref{eq:ratio}. Let $f_t$ be the path in  $\tilde{\mathcal{M}}_{\operatorname{univ}}^{\operatorname{plog}}(\v,\Gammas, \tilde{x}_0, b)$ given by deforming $\zeta_{i,\nu}$ to $e^{t\beta}\zeta_{i,\nu}$ (and suitably deforming the $\bar{\partial}$ operator on $ND_i$). Then $$ \frac{d}{dt}[ob_{\Gammas}\circ f_t]_{|t=0}=[1_{e,i}] $$ 
\vskip 5 pt

$i \notin I^\nu$: (c.f.  \cite[Claim 2, Case (ii) in proof of Theorem 1.5]{Tehrani2}) In this case,  $\eta_{i,z_{\tilde{e}}}$ is defined by the expansion \eqref{eq:expansion}.  Take a cut-off function $\beta$ as above, and by a suitable cut-off constuction,  one can deform $u_\nu$ (with $\upsilon$ non-trivial) so that the expansion becomes: 
\begin{align} \label{eq:expansionproof} \pi_\mathbb{C} \circ (u_\nu)|_{\Delta} = e^{t\beta}\eta_{i,z_{\tilde{e}}} z^{\ord_i} + O(|z|^{\ord_i+1}) \end{align}

from which it follows that $[1_{e,i}]$ is in the image of the differential in this case as well.   \vskip 5 pt
     \end{proof}

\begin{lem} \label{eq:ruancompact} Let $\tilde{x}_0$ be a capped orbit such that $[u_{x_{0}}] \cdot \D=\v$.  Assume $\lambda_m> \v$, $(J_S, \upsilon)$ is sufficiently generic so that Lemma \ref{eq:ruantransv} holds and $||\upsilon||$ is sufficiently small:   

\begin{itemize} \item If $\degr(x_0) = 0,$ the moduli spaces $\mathcal{M}_{\E}(\v, \tilde{x}_0)$ are compact. 
\item If $\degr(x_0) = 1$, then $\mathcal{M}_{\E}(\v, \tilde{x}_0)$ admits a compactification whose boundary strata $\partial\mathcal{M}_{\E}(\v, \tilde{x}_0)$ are described as follows:  \begin{align} \label{eq:boundarystrata} \partial \overline{\mathcal{M}}_{\E}(\v, \tilde{x}_0):= \bigsqcup_{(y,[u_{x_{0},y}])}\bigsqcup_{(\mathfrak{p},b_{0},b_{1})} \mathcal{M}_{\E}(\v,\tilde{y},b_0) \times \mathcal{M}_{\E}(x_0,y,[u_{x_{0},y}],b_1) \end{align} 
where: \begin{itemize} \item $y$ is an orbit with $\degr(y)=0$, $[u_{x_{0},y}]$ is a relative homology class and $\tilde{y}:=(y,[u_y])$ is a capping disc such   that $(-[u_y])\#[u_{x_{0},y}]=-[u_{x_{0}}].$ \item $\mathfrak{p}$ is a partition of $\lbrace 2, \cdots, r(\tilde{x}_0)+1 \rbrace$ into two sets $S_0,S_1$ which are identified by the induced ordering with $S_0 \cong \lbrace 2,\cdots, r(\tilde{y}_0)+1 \rbrace$ and $S_1 \cong \lbrace 1, \cdots, r([u_{x_{0},y}]) \rbrace.$ \item $b_0:S_0 \to \lbrace 1,\cdots, k' \rbrace$, $b_1: S_1 \to \lbrace 1,\cdots, k' \rbrace$ are labellings of the parititoned sets.  \end{itemize}
 \end{itemize} \end{lem} 
\begin{proof} The proof follows the lines of Theorem \ref{thm: sequentialcomp} and we only reprise the main steps. First, if  $||\upsilon||$ sufficiently small, a similar energy argument prevents breaking along Hamiltonian orbits in $\D$. Then, as before, log Gromov-compactness (the adaptation to $\upsilon$-perturbed moduli spaces is carried out in \cite[Section 3.3]{Tehrani2}) implies that sphere bubbles which attach at points in $S$ other than $z=\lbrace \infty \rbrace$ must be constant. Unlike in the unstablized case, these constant configurations can arise(when marked points collide). However, where they do collide, the PSS solution must intersect (components of) $\E$ with multiplicity $>1$ which again happens in codimension two (and hence in our context does not occur generically).  Finally,  in view of Definition \ref{defn:stabdata},  the remaining sphere bubble configurations (with bubbles attaching at $z=\lbrace \infty \rbrace$) will all be stable and hence potential limits belong to some $\E$-marked moduli space $\mathcal{M}(\v,\Gammas, \tilde{x}_0, b)$.  These limits can be excluded using Lemma \ref{eq:ruantransv}. \end{proof} 

\begin{defn} \label{defn:PSSstabcomp} Suppose $\v$,  $\tilde{x}_0$,  $(J_S,\upsilon)$  are as in Lemma \ref{eq:ruancompact}.  Let $c$ be a component of $D_{|\v|}$.  Then as before,  we define $\mathcal{M}_{\E}(\v, \tilde{x}_0,c)$ to be the $\E$-marked log PSS maps such that $$u(z_0) \in D_{|\v|, c}.$$   \end{defn} 

\subsection{The splitting map} \label{subsection: stabilizedPSS} 

Suppose temporarily that $\operatorname{char}(\K)=0.$ For any $\v$, choose $m$ with $\lambda_m>\v$ as well as a generic $(J_S,\upsilon)$ (so that Lemma \ref{eq:ruancompact} holds) and define  $\PSS_{log}^{m}(\theta_{(\v,c)}) \in CF_{\E}^0(X \subset M, H^{m})$
by the rule \begin{align} \label{eq:PSSdef} \PSS_{log}^{m}(\theta_{(\v,c)})= \sum_{\tilde{x}_0} \sum_{u \in \mathcal{M}_{\E}(\v, \tilde{x}_0,c)}\frac{1}{r(\tilde{x}_0)!} \mu_u  \end{align} 
where $\mu_u \in |\mathfrak{o}_{x_{0}}|$ is defined by \eqref{eq:PSSlowe} and $\mathcal{M}_{\E}(\v, \tilde{x}_0,c)$ is defined in Definition \ref{defn:PSSstabcomp}.

\begin{lem}  If $\lambda_m>\v$ and $(J_S,\upsilon)$ chosen as in Lemma \ref{eq:ruancompact}, $ \PSS_{log}^{m}(\theta_{(\v,c)})$ defines a cohomology class.
\end{lem} 
\begin{proof}  This is essentially a consequence of Lemmas \ref{eq:ruantransv}, \ref{eq:ruancompact} together with the observation that in the strata arising in \eqref{eq:boundarystrata}, there are $\frac{r(\tilde{x}_0)!} {(r(\tilde{y})!)(r(u_{x_{0},y})!)}$ partitions $\mathfrak{p}$ which distribute the marked points between the two components. Thus the boundary of the sum of rational chains of the form $\frac{1}{r(\tilde{x}_0)!}\overline{\mathcal{M}}_{\E}(\v, \tilde{x}_0)$ is exactly the composition $\partial \circ \PSS_{log}^{m}.$ \end{proof}  

The cohomology class $\PSS_{log}^{m}(\theta_{(\v,c)}) \in HF_{\E}^0(X \subset M, H^{m})$ does not vary in one parameter families of (generic) $(J_S, \upsilon).$ Moreover as before, we have that for any $m'>m$,
\begin{align} \mathfrak{c}_{m,m'} \circ \PSS_{log}^{m}(\theta_{(\v,c)})= \PSS_{log}^{m'}(\theta_{(\v,c)}) \end{align}

We therefore obtain a well-defined map: $$  \PSS_{log}:\bigoplus_\v \K \cdot \theta_{(\v,c)} \to  SH^0_{\E}(X) $$ 
which we can compose with  \eqref{eq:obviso} to obtain a map 
\begin{align} \label{eq: PSSlim} \PSS_{log}:\bigoplus_{\v \in B(M,\D)} \K \cdot \theta_{(\v,c)} \to  SH^0(X,\K)  \end{align}

\begin{thm} \label{thm:additiveisominor} For any log nef pair $(M,\D)$, the map \eqref{eq: PSSlim} is an isomorphism. \end{thm} 
\begin{proof} 

For any $I \subset \lbrace 1,\cdots, k \rbrace$,  let $\v_I=(v_{I,i})$ denote the vector defined by the rule: 
\begin{align} \label{eq:vectorI}  v_{I,i}=
\begin{cases} 1 \text{ if } i \in I \\
                        0  \text{ if } i \notin I
 \end{cases} 
 \end{align} 
For any $I$ such that $\v_I \in B(M,\D)$ and any component $D_{I,c}$ of $D_I$,  consider the corresponding element $\PSS_{log}(\theta_{(\v_{I,i},c)})$ which we can view lying in the $q^0$ part of $SC^*(X, \K)$ (recall \eqref{eq: chainlevel}). In the $E_1$ page $H^*(\oplus_p F^p SC^*(X,\K)/F^{p+1}SC^*(X,\K))$, this becomes the element  $\PSS_{log}^{low}(\theta_{(\v_{I,i},c)})$ and it follows that these elements survive to the $E_{\infty}$ page. As the elements $\PSS_{log}^{low}(\theta_{(\v_{I,i},c)})$  generate the degree zero part of the spectral sequence multiplicatively,  it follows that the spectral sequence degenerates in degree 0 i.e. that we have $E_1^{p,-p} \cong E_\infty^{p,-p}.$ Chasing through the definitions, it is not difficult to see that we have a commutative square:  

\begin{equation} \label{eq:commydiagram1}
\begin{tikzcd}
\mathcal{A}_\K  \arrow{r}{\operatorname{gr}_F \PSS_{log}} \arrow{d}{\PSS_{log}^{low}} & \bigoplus_p F^{p}SH^0(X,\K)/F^{p+1}SH^0(X,\K) \arrow{d}{\cong} \\
\bigoplus_p E_1^{p,-p} \arrow{r}{\cong} & \bigoplus_p E_\infty^{p,-p}
\end{tikzcd}
\end{equation}

where in the upper horizontal arrow we have identified $\operatorname{gr}_F(\bigoplus_\v \K \cdot \theta_{(\v,c)})$ with $\bigoplus_\v \K \cdot \theta_{(\v,c)}$ because the filtration is canonically split.  The map $\operatorname{gr}_F \PSS_{log}$ is an isomorphism because all of the other maps in the diagram are isomorphisms.  It follows that \eqref{eq: PSSlim} is an isomorphism as well.    \end{proof}

\begin{cor} \label{cor:arbitraryfield} 
For an arbitrary field $\K$ (of any characteristic) there is a canonical isomorphism of rings \begin{align} \label{eq:ringisomy} gr_{F_w}SH^0(X,\K) \cong \mathcal{SR}_{\K}(\Delta(\D)_{(0)}) \end{align}  
\end{cor} 
\begin{proof} Let $\v_I$ be the vector from \eqref{eq:vectorI}.  Choose some slope $\lambda_m$ so that $ \lambda_m > w(\v_I)$ for all $\v_I \in B(M,\D)$.  By Theorem \ref{thm:additiveisominor},  the Floer differential vanishes on $CF^0(X;H^m)$ in characterstic zero. By comparing ranks,  it follows that the Floer differential vanishes over the field $\K.$ Thus, again over $\K$, the classes $\PSS_{log}^{low}(\theta_{(\v_I,c)})$ on the $E_1$ survive to the $E_{\infty}$ page.  Again because the elements $\PSS_{log}^{low}(\theta_{(\v_I,c)})$  generate the degree zero part of the spectral sequence multiplicatively,  we have that $E_1^{p,-p} \cong E_\infty^{p,-p}.$ We conclude that there is an isomorphism of the form \eqref{eq:ringisomy}.  \end{proof} 

\begin{rem} \label{rem:arbifield} \begin{itemize} \item We do not address the question of whether our constructions are independent of $\E.$ \item In view of Corollary \ref{cor:arbitraryfield}, one suspects that the map \eqref{eq: PSSlim} can be constructed over fields of arbitrary characteristic(or in fact over $\mathbb{Z}$). This seems to involve being able to regularize the moduli spaces $\overline{\mathcal{M}}(\v,\tilde{x}_0)$ directly by perturbing almost complex structures because typically virtual methods requiring working with fields of characteristic zero (even our fairly elementary approach involves dividing by $r(\tilde{x}_0)!$). Closely related questions are discussed in \cite[Section 4.5.]{Tehrani2}. \end{itemize} \end{rem}

Theorem \label{thm:additiveisominor} implies the following finiteness result for $SH^*(X,\K)$, which generalizes \cite{GP2}: 

\begin{lem} \label{lem:finiteness} For any log Calabi-Yau pair $(M,\D)$:
\begin{enumerate} \item \label{item:finitelem1}  $SH^0(X, \K)$ is a finitely generated $\K$-algebra. 
\item \label{item:finitelem2} $SH^*(X,\K)$ is a finitely generated module over $SH^0(X, \K)$ (and hence is finitely generated as a graded $\K$-algebra). 
\end{enumerate} 
\end{lem} 
\begin{proof}  \emph{Claim (1):} Let $F_w SH^0(X,\K)$ (depending as before on the line bundle $\mathcal{L}$ over $M$) be the filtration on $SH^0(X,\K)$.  We have by Theorem \ref{thm:additiveiso},  that the associated graded algebra $\operatorname{gr}_F SH^0(X,\K)$ is generated by the elements $\PSS_{log}^{low}(\theta_{(\v_I,c)})$ where $\v_I$ are the vectors  from  \eqref{eq:vectorI} and $D_{I,c}$ are the irreducible components of $D_I$. Because this filtration induces the discrete topology on  $SH^0(X,\K)$, we have by \cite[Chapter III,\S 2]{BourbakiChapterIII} that lifts of these elements generate $SH^0(X,\K)$ as well.  \vskip 5 pt \emph{Claim (2):} (Compare with the proof of  \cite[Theorem 5.30]{GP2}.)  Theorem 1.1 of \cite{GP2} identifies the full $E_1$ page of the spectral sequence induced from the $F_w$ filtration with a certain ring $H^*_{log}(M,\D)$.  $H^*_{log}(M,\D)$ is module finite over  $\operatorname{gr}_F SH^0(X,\K)$ and so  finite generation of $SH^*(X, \K)$ again follows from \cite[Chapter III, \S2]{BourbakiChapterIII}.
  \end{proof} 

\subsubsection{\texorpdfstring{$\Lambda$}{Lambda}-twisted coefficients}  \label{subsection:lambda}
This section is quite similar to \cite[Section 8]{Pascaleff}.  Fix a ground field $\K$,  let $\hatH$ be $H_2(M,\mathbb{Z})$ divided by torsion \begin{align} \label{eq:hatH} \hatH := H_2(M,\mathbb{Z})/ H_2(M,\mathbb{Z})_{\operatorname{tor}}, \end{align} and let $\Lambda=\K[\hatH]$ be the group ring on $\hatH.$ Our goal is to prove a version of Theorem\ref{thm:maintheorem} for an enhanced version of symplectic cohomology, $SH^*(X, \underline{\Lambda})$, which is linear over $\Lambda$ and where one keeps track of the homology class of Floer curves after attaching suitable capping discs (similar constructions are sometimes referred to as ``$\Lambda$-twisted coefficients" in the Floer theory literature). $SH^*(X, \underline{\Lambda})$ is thus an invariant of the pair $(M,\D)$ rather than $X$ itself,  but one can recover the usual symplectic cohomology ring $SH^*(X,\K)$ discussed above by base changing to the augmentation ideal.  As in \cite{GP2}, we equip (the co-chain complex computing) $SH^*(X, \underline{\Lambda})$ with an integer valued version of the action filtration, $F_wSH^*(X, \underline{\Lambda}).$

It should be noted that this twisted story is not needed for proving Theorem \ref{thm:conj1}  and is thus not relevant to the main ``thread" of this paper (in particular,  note that we don't know of a categorical version of $\Lambda$-twisting).  Nevertheless,  it is potentially interesting for two reasons.  First,  taking $\operatorname{Spec}(SH^0(X,\underline{\Lambda}))$ over $\operatorname{Spec}(\Lambda)$ gives a family of varieties with potential modular interpretations \cite{Keel-Hacking-Yu} (one could also possibly contemplate things like connections and periods for this family).  Second,  the twisted coefficients are the natural context for making comparisons with algebro-geometric constructions (\cite{sequel}).  As it involves very little extra work,  we lay some groundwork for these potential future applications here.

 The $\Lambda$-twisted Floer complex (first of some fixed admissible $H^{\lambda}$) is defined as follows: \begin{align} CF_{\Lambda}^*(X \subset M;H^{\lambda}):= \bigoplus_{\tilde{x}} \K \cdot \langle \tilde{x}\rangle \end{align} 
For any orbit, the set of capping discs is a torsor under $\hatH$, giving $CF_{\Lambda}^*(X \subset M;\Hla)$ the structure of an $\Lambda$-module. In fact, observe that for any $\v \in B(M,\D),$ the orbits $x \subset U_\v$ lie entirely in the fiber of the projection $U_{|\v|} \to D_{|\v|}$ and hence have a canonical fiberwise capping disc $F_x \subset M$. Using this disc, we may view the Novikov enhanced Floer cochains as being free $\Lambda$-modules $$CF_{\Lambda}^*(X \subset M;H^{\lambda}):= \bigoplus_{x} \Lambda \cdot \langle [x,F_x] \rangle $$

We have \begin{align} \label{eq:basechange} CF_{\Lambda}^*(X \subset M;H^{\lambda})\otimes_\Lambda \K \cong CF^*(X \subset M;H^{\lambda}) \end{align}  where the $\Lambda$ module structure on $\K$ is induced from the augmentation homomorphism $\epsilon_{\operatorname{aug}}: \Lambda \to \K.$

The differential again counts index one Floer trajectories in $\mathcal{M}(\tilde{x}_0, \tilde{x}_1)$(recall \eqref{eq:FloerRinv2}): 
\begin{equation} 
    \partial_{CF}(\langle \tilde{x}_1 \rangle)  =\sum_{\tilde{x}_{0}} \sum_{u \in \mathcal{M}(\tilde{x}_0,\tilde{x}_1)} \mu_u \langle \tilde{x}_0 \rangle
\end{equation} 
where $\mu_u=\pm 1$ is again the sign associated to $u \in \mathcal{M}(\tilde{x}_0,\tilde{x}_1)$. There are also relative versions of these constructions in degree zero, giving rise to cochain groups $CF_{\Lambda,\E}^0(X \subset M;H^{\lambda}).$ Using continuation maps, we can define $SH^*(X,\underline{\Lambda})$ and in degree zero the relative version $SH_{\E}^0(X,\underline{\Lambda}) \cong SH^0(X,\underline{\Lambda}).$ Just as before, we can define a class (in some fixed $CF_{\Lambda,\E}^0(X \subset M;H^{\lambda})$) by the formula  \begin{align} \label{eq:PSSdefrel} \PSS_{log}^{m}(\theta_{(\v,c)})= \sum_{\tilde{x}_0} \sum_{u \in \mathcal{M}_{\E}(\v, \tilde{x}_0)}\frac{1}{r(\tilde{x}_0)!} \mu_u <\tilde{x}_0> \end{align}
 which gives rise to a map \begin{align} \label{eq: PSSlim2} \PSS_{log}:\bigoplus_\v \Lambda \cdot \theta_{(\v,c)} \to SH_{\E}^0(X,\underline{\Lambda}) \cong SH^0(X,\underline{\Lambda})  \end{align} 

We are finally in a position to prove the twisted version of Theorem \ref{thm:maintheorem} mentioned in the introduction: 

\begin{thm} \label{thm:additiveiso} For any log nef pair $(M,\D)$, the map \eqref{eq: PSSlim2} is an isomorphism. \end{thm} 
\begin{proof} The proof is the same as Theorem \ref{thm:additiveisominor}, with appropriate modifications that we spell out. The filtration $F_w$ obviously lifts to the $\Lambda$-twisted Floer complexes giving rise to a spectral sequence. Theorem \ref{lem:GP2lemma} adapts without change to give rise to 
a ring isomorphism: $$ \PSS_{log}^{low}: (\bigoplus_\v \Lambda \cdot \theta_{(\v,c)}, \ast_{\operatorname{SR}}) \cong E_1^{p,-p} $$ which as before implies the collapse of the spectral sequence.  The analogue of  \eqref{eq:commydiagram1} is then given by \begin{equation} \label{eq:commydiagram2}
\begin{tikzcd}
\bigoplus_\v \Lambda \cdot \theta_{(\v,c)}   \arrow{r}{\operatorname{gr}_F \PSS_{log}} \arrow{d}{\PSS_{log}^{low}} & \bigoplus_p F^{p}SH^0(X,\underline{\Lambda})/F^{p+1}SH^0(X,\underline{\Lambda}) \arrow{d}{\cong} \\
\bigoplus_p E_1^{p,-p} \arrow{r}{\cong} & \bigoplus_p E_\infty^{p,-p}
\end{tikzcd}
\end{equation} 
as before, the map $\operatorname{gr}_F \PSS_{log}$ is an isomorphism because all of the other maps in the diagram are isomorphisms thereby implying that \eqref{eq: PSSlim2} is an isomorphism as well.
\end{proof} 

By  \eqref{eq:basechange} and flatness, we have \begin{align} \label{eq:cohbasechange} SH^0(X, \underline{\Lambda})\otimes_ \Lambda \K  \cong SH^0(X,\K) \end{align}   

In view of \eqref{eq:cohbasechange}, Lemma \ref{lem:finiteness} implies that for any log Calabi-Yau pair $(M,\D)$, $SH^0(X,\K)$ is also finitely generated.  Parallel to Corollary \ref{cor:arbitraryfield} and Lemma \ref{lem:finiteness},  we have:  

\begin{cor} \label{cor:arbitraryfield2} 
For an arbitrary field $\K$ (of any characteristic) there is a canonical isomorphism of rings \begin{align} \label{eq:ringiso2} gr_{F_w}SH^0(X,\underline{\Lambda}) \cong \mathcal{SR}_{\Lambda}(\Delta(\D)_{(0)}) \end{align}  
\end{cor} 
\begin{proof} The proof is identical to Corollary \ref{cor:arbitraryfield} and is omitted.  \end{proof} 

\begin{lem} \label{lem:finiteness2} For any log Calabi-Yau pair $(M,\D)$:
\begin{enumerate} \item $SH^0(X, \underline{\Lambda})$ is a finitely generated $\Lambda$-algebra. 
\item $SH^*(X,\underline{\Lambda})$ is a finitely generated module over $SH^0(X, \underline{\Lambda})$ (and hence is finitely generated as a graded $\Lambda$-algebra). 
\end{enumerate} 
\end{lem} 
\begin{proof} The proof is identical to Lemma \ref{lem:finiteness}     
 and is omitted. \end{proof}

\subsection{$\mathcal{W}(X)$ is semi-affine} \label{subsection:proper} \hspace*{\fill} \\

\textbf{Note:} For the remainder of this article, all pairs $(M,\D)$ will be assumed to be Calabi-Yau pairs.  The notation in this section will be a bit heavier than in previous sections,  so we begin by listing the main notations which will be used throughout the section--- the reader shouldd refer back to this list as needed. \vskip 5 pt

Fix some sufficiently small $\epsilon_1$ and recall the notations $\hatX,\hatLio$ from \eqref{eq:manicorners}.  Fix a specific rounding $\Sigma_{\vec{\epsilon}}$ (:=$\partial \bar{X}_{\vec{\epsilon}}$) of $\hatX$. \vskip 5 pt 
 
\textbf{Further notations:}
\begin{itemize} 
\item For any $\epsilon$ such that $\epsilon^2 \leq \operatorname{inf}_i\frac{2\pi \epsilon_0^2} {\kappa_i},$ we set $U_{i,\epsilon}$ to be the
region where $\frac{\rho_i}{\kappa_i/2\pi} \leq \epsilon^2.$ For $I \subset \lbrace 1,\cdots,k\rbrace, $ define $U_{I,\epsilon}$ to be \begin{align} \label{eq:UIe} U_{I,\epsilon}:=\cap_{i\in I} U_{i,\epsilon}.\end{align}

\item Let $(\epsilon^{\operatorname{pert}}_i)_{i \in \lbrace 1,\cdots, k \rbrace}$ be a $k$-tuple of real numbers with norm $||\epsilon^{\operatorname{pert}}_i||<<\epsilon_1$.  Define the ``$(\epsilon^{\operatorname{pert}}_i)$-skew manifold with corners" \begin{align} \label{eq:skewliouville} \widehat{X}_{(\epsilon^{\operatorname{pert}}_{i})}:=M\setminus (\cup_i U_{i, \epsilon_1+\epsilon^{\operatorname{pert}}_{i}}) \end{align}

\item Let $\delta^{\operatorname{max}}$ be a small positive real number and let $\vec{\delta}_I=(\vec{\delta}_{I,i})_{i \in I} $ be an $I$-tuple (for some $I \subset \lbrace 1,\cdots, k \rbrace$) of real numbers such that  $|\vec{\delta}_{I,i}| \leq \delta^{\operatorname{max}}.$ Let $S_{I,\epsilon_{1},\vec{\delta}_{I}}$ denote the open submanifolds \begin{align} \label{eq:flset1} S_{I,\epsilon_{1},\vec{\delta}_{I}}:= (\cap_{i \in I} \lbrace x \in X: 2\pi \frac{\rho_i}{\kappa_i}=(\epsilon_1+\vec{\delta}_{I})^2 \rbrace) \setminus (\cup_{j \notin I} \lbrace  2\pi \frac{\rho_i}{\kappa_i} \geq \epsilon_1^2 \rbrace)  \end{align} When $\vec{\delta}_{I,i}=\vec{0},$ we will use the simpler notation $S_{I,\epsilon_{1}}$ for  $S_{I,\epsilon_{1},\vec{0}}.$ 

\item For any non-empty subset $I$,  set $US_{I,\epsilon_{1}}$ be the union of all $S_{I,\epsilon_{1},\vec{\delta}_{I}}$ over all $\vec{\delta}_I,$ \begin{align}  \label{eq:flset2}  US_{I,\epsilon_{1}}:= \cup_{\vec{\delta}_{I}} S_{I,\epsilon_{1},\vec{\delta}_{I}}  \end{align} 

\item Finally, we have a tubular neighborhood of $\hatX$ given by  \begin{align}  \label{eq:flset3}  U\hatX := \cup_I US_{I,\epsilon_{1}} \end{align}  
\end{itemize}
 
(Here $I$ ranges over all non-empty subsets.)  In our main argument,  we will need to consider Hamiltonian functions which are not (small perturbations of) functions of a single Liouville coordinate.  We will therefore work with a slight variation of the functions from Definition \ref{defn:linearf}. 

\begin{defn} A function $g_L$ will be called homogeneous if:
       \begin{enumerate} \item \label{item:hammy1} on each subset $U_{I,\epsilon_{1}+\delta^{\operatorname{max}}} \setminus (\cup_{j\notin I} U_{j,\epsilon_{1}+\delta^{\operatorname{max}}}),$ $g_L=g_L(\rho_i)$ is a function of the variables $\rho_i$ with $i \in I$.   \item there exists $K_{\vec{\epsilon}}$ with $\operatorname{min}_{\D} R^{\vec{\epsilon}}>K_{\vec{\epsilon}}>1$ so that when $R^{\vec{\epsilon}} \geq K^{\vec{\epsilon}}$,  $g_L \geq 0$ and \begin{align} \label{eq:hcond101} \theta(X_{g_{L}})=g_L + \lambda \end{align} for some $\lambda>0.$ \end{enumerate} \end{defn}
       
 
The main example of homogeneous functions (other than Liouville admissible functions) that we will need is the following.   Let $(\epsilon^{\operatorname{pert}}_i)$ be a small vector and let $\tilde{X}_{0},  \tilde{X}_{1}$ be two different $C^0$-close roundings of $\widehat{X}_{(\epsilon^{\operatorname{pert}}_{i})}$.  Let $g_{0,L}$ be Liouville admissible with respect to $ \tilde{X}_{0}$ and $g_{1,L}$ Liouville admissible with respect to $\tilde{X}_1$.  Then  \begin{align}  g_{0,L}+ g_{1,L} \end{align} is also homogeneous (even though it is not a function of any Liouville coordinate).  The ability to take sums of functions with respect to two different Liouville coordinates will be very useful when considering the module action of $ SH^*(X)$ on $WF^*(L_0,L_1)$.  

 As usual, given two Lagrangian branes $L_0,L_1$, we will consider homogeneous functions such that each Hamiltonian $g_L$ have no chords at infinity.  Let $G_L$ be a small perturbation of $g_L$ which is supported near the chords and let $J_t$ be an almost complex structure which is of contact type near the level set where $R^{\vec{\epsilon}}=K^{\vec{\epsilon}}$.  Then for $R^{\vec{\epsilon}}>K^{\vec{\epsilon}},$ we have that: \begin{align} \label{eq:hcond200}  dR^{\vec{\epsilon}}(X_{G_{L}}) = 0 \end{align} To check \eqref{eq:hcond200}, note that $G_{L}=g_{L}$ is a function of the radial coordinates $\rho_i$ so $X_{G_{L}}$ is a linear combination of the angular vector fields.  Then using \eqref{eq:hcond101}, \eqref{eq:hcond200}, the integrated maximum principle argument carries over to show that Floer strips for the pair $(G_L, J_{t})$ have to lie below in the region $ X \setminus \lbrace R^{\vec{\epsilon}} \geq K^{\vec{\epsilon}} \rbrace$ (for similar arguments see \cite[\S 2.3]{GP2}). Thus, there are well-defined Floer cochain complexes $CF^*(L_0, L_1; G_L)$ and cohomology groups $HF^*(L_0, L_1; G_L)$.  

Suppose $L_0,  L_1$ are cylindrical and separated outside of $\hatLio$.   We next recall a key proposition from \cite{McLeanwrapped} which will enable us to Hamiltonian isotope any pair of $L_0,L_1$ into a normal form where Hamiltonian chords (for homogeneous functions satisfying a few extra conditions) between them can be easily computed.  To state this proposition requires the following preliminary definition. 

\begin{defn} Let $y$ be a point of $U_I$ and let $F_y$ denote the fiber of the projection $\pi_I: U_I \to D_I$ which passes through $y.$ A Lagrangian submanifold $L \subset X$ is fiber radial (with respect to $\pi_I$) near $y$ if there is a neighborhood $\mathcal{N}_y$ of $y$ such that in any unitary trivialization of the fiber $F_y \cong U \subset \mathbb{C}^{|I|}$, \begin{align} L \cap \mathcal{N}_y\cap F_y= \cap_{i \in I} \lbrace \phi_i = \alpha_i \rbrace \cap \mathcal{N}_y. \end{align} for some angles $\alpha_i.$   \end{defn}

\begin{prop} \cite[Lemma 5.4]{McLeanwrapped} \label{lem:McLeanham} There is Hamiltonian diffeomorphism $\phi_{\ham}: X \to X$ which is supported in $U\hatX$ such that for each $I \subset \lbrace 1,\cdots,k \rbrace$:\begin{itemize}\item $\phi_{\ham}(L_0), \phi_{\ham}(L_1)$ are transverse to $S_{I,\epsilon_{1}}$ and the Lagrangian immersions $i_{0,I}:= (\pi_I)_{|\phi_{\ham}(L_{0}) \cap S_{I,\epsilon_{1}}}$ and  $i_{1,I}:= (\pi_I)_{|\phi_{\ham}(L_{1}) \cap S_{I,\epsilon_{1}}}$ are transverse to each other and also the intersection points between $i_{0,I} \cap i_{1,I}$ are isolated if $|I|<n$.  \item If $y \in \operatorname{Im}(i_{0,I}) \cap \operatorname{Im}(i_{1,I})$ then $\phi_{\ham}(L_0)$(respectively $\phi_{\ham}(L_1)$) is fiber radial (with respect to $\pi_I$) near each point of $\phi_{\ham}(L_0) \cap \pi_I^{-1}(y) \cap S_{I,\epsilon_1}$ (respectively $\phi_{\ham}(L_1) \cap \pi_I^{-1}(y) \cap S_{I,\epsilon_{1}}$).  \end{itemize}  \end{prop}

 We will use the simpler notation $\tilde{L}_0:=\phi_{\ham}(L_0)$ and $\tilde{L}_1:=\phi_{\ham}(L_1).$ Note that the Lagrangians $\tilde{L}_0, \tilde{L}_1$ are still exact Lagrangians with primitives $f_{\tilde{L}_{i}}$ such that $df_{\tilde{L}_{i}}=\theta_{|\tilde{L}_i}.$ After making a further small perturbation,  we can and will also assume that $\tilde{L}_0, \tilde{L}_1$ meet transversely in $\hatLio$ as well.   For any $I$,  we also set $$ \mathcal{X}_I (\tilde{L}_0, \tilde{L}_1) :=\operatorname{Im}(i_{0,I}) \cap \operatorname{Im}(i_{1,I}) $$
\begin{lem} \label{lem:isolation} Set $i_{0,I,\vec{\delta}_{I}}:= (\pi_I)_{|\tilde{L}_{0} \cap S_{I,\epsilon_{1},\vec{\delta}_{I}}}$ and  $i_{1,I,\vec{\delta}_{I}}:= (\pi_I)_{|\tilde{L}_{1}) \cap S_{I,\epsilon_{1},\vec{\delta}_{I}}}.$ For $\delta^{\operatorname{max}}$ sufficiently small, we have that: \begin{align} \mathcal{X}_I (\tilde{L}_0, \tilde{L}_1)= \operatorname{Im}(i_{0,I,\vec{\delta}_{I}}) \cap \operatorname{Im}(i_{1,I,\vec{\delta}_{I}}) \end{align}  \end{lem}
\begin{proof}  If $|I|=n$, the assertion is trivial. Otherwise, because the Lagrangians are fiber radial, we have an inclusion $\mathcal{X}_I (\tilde{L}_0, \tilde{L}_1)) \subset \operatorname{Im}(i_{0,I,\vec{\delta}_{I}}) \cap \operatorname{Im}(i_{1,I,\vec{\delta}_{I}})$ provided $\delta^{\operatorname{max}}$ is sufficiently small. The opposite inclusion follows from the fact that for $\vec{\delta}_{I}$ near zero, the intersections are transverse and isolated and as $\vec{\delta}_{I} \to 0$, converge (after possibly passing to a subsequence) to points in  $\mathcal{X}_I (\tilde{L}_0, \tilde{L}_1)$.  \end{proof} 

For the remainder of this section, we choose $\delta^{\operatorname{max}}$ used in the definition of \eqref{eq:flset1}-\eqref{eq:flset3} small enough so that Lemma \ref{lem:isolation} holds.  We further assume
\begin{enumerate} \item There are no chords of $g_L$ at infinity.  \item all of the non-constant chords of $g_L$ lie in $U\hatX$ and that all of the chords of the Hamiltonians $g_L$ are all non-degenerate. 
 \item On $US_{I,\epsilon_{1}} \setminus (\cup_{I',I\subset I'}US_{I',\epsilon_{1}}),$ $g_L= g_L(\rho_i)$ is a convex function with $\frac{\partial g_L(\rho_i)}{\partial \rho_i} \leq 0$.  
 \end{enumerate}
 
 Having done this,  we can give concrete descriptions of all of the chords in each $US_{I,\epsilon_{1}}$.  In view of Lemma
 \ref{lem:isolation},  any chord in $US_{I,\epsilon_{1}}$ lies in a fiber $F_y:= \pi_I^{-1}(y)$ for $y \in \mathcal{X}_I (\tilde{L}_0, \tilde{L}_1).$ For any such $y$, set $US_{I,\epsilon_{1},y}=: US_{I,\epsilon_{1}} \cap F_y$ and $S_{I,\epsilon_{1},y}=: S_{I,\epsilon_{1}} \cap F_y.$ We have  $$ US_{I,\epsilon_{1},y} \cong [\epsilon_1-\delta^{\operatorname{max}},\epsilon_1+\delta^{\operatorname{max}}]^{|I|} \times S_{I,\epsilon_{1},y} $$ Similarly,  let $\tilde{L}_{i,\epsilon_1,y}:= \tilde{L}_i \cap S_{I,\epsilon_{1},y}$ for $i \in \lbrace 0,1 \rbrace$ (this is a finite collection of points). \begin{align} \tilde{L}_0 \cap US_{I,\epsilon_{1}} \cap F_y= [\epsilon_1-\delta^{\operatorname{max}},\epsilon_1+\delta^{\operatorname{max}}]^{|I|} \times \tilde{L}_{0,\epsilon_1,y} \\ \tilde{L}_1 \cap US_{I,\epsilon_{1}} \cap F_y=  [\epsilon_1-\delta^{\operatorname{max}},\epsilon_1+\delta^{\operatorname{max}}]^{|I|} \times \tilde{L}_{1,\epsilon_1,y} \nonumber \end{align} 

Hamiltonian chords in $US_{I,\epsilon_{1}}$ are determined by:
 \begin{itemize} \item a point $y \in \mathcal{X}_I (\tilde{L}_0, \tilde{L}_1).$
                  \item points $\alpha_0 \in \tilde{L}_{0,\epsilon_1,y}$, $\alpha_1 \in \tilde{L}_{1,\epsilon_1,y}.$
                  \item a winding vector $\v \in \mathbb{N}^k$ with support $|\v| \subseteq I.$   
\end{itemize}

 Concretely, after choosing a $U(1)^{|I|}$ equivariant identification \begin{align} \label{eq:toridfc} S_{I,\epsilon_{1},y} \cong T^{|I|}, \end{align} we can identify $\tilde{L}_{0,\epsilon_1,y}=\cup_\ell (\alpha_i^{\ell})$ where $(\alpha_i^{\ell})$ is an $I$-tuple of angles lying in $[0,2\pi)$ and $\ell$ ranges over $\lbrace 1,\cdots, |\tilde{L}_{1,\epsilon_1,y}| \rbrace$ (and similarly $\tilde{L}_{1,\epsilon_1,y}=\cup_{\ell'} (\alpha_i^{\ell'})$). Fix two points $\alpha_0 = (\alpha_{0,i})$ and $\alpha_1=(\alpha_{1,i})$ and set \begin{equation} 
v_{s,i}=
 \begin{cases}
 0, \quad \text{if } \alpha_{1,i} < \alpha_{0,i} \\
1, \quad \text{if } \alpha_{1,i} > \alpha_{0,i}.
\end{cases}
\end{equation} After passing to the universal cover of $US_{I,\epsilon_{1},y}$, we have a ``short chord" which is given by  $(\epsilon_1+\vec{\delta}_I) \times \vec{P} \subset [\epsilon_1-\delta^{\operatorname{max}},\epsilon_1+\delta^{\operatorname{max}}]^{|I|} \times \mathbb{R}^{|I|}$ where $\vec{\delta}_I$ is an appropriate vector in $[0,\delta^{\operatorname{max}}]^I$ and $\vec{P}$ is a straight line connecting $(\alpha_{0,i})$ with $(\alpha_{1,i}-2\pi v_{s,i}).$  More general chords are similarly determined by the straight-line connecting $(\alpha_{0,i})$ with $(\alpha_{0,i}-2\pi (v_{s,i }+ v_{i}))$ for some winding vector $\v=(v_i)$ with support $|\v| \subseteq I$.  Convexity and nondegeneracy imply that for any $\v$,  there is at most one chord determined by $\alpha_0,\alpha_1 \in  \tilde{L}_{1,\epsilon_1,y}$ of winding vector $\v$  which we denote by $x_{y, \v} (\alpha_0,\alpha_1)$.  We also have finitely many chords $x_y$ corresponding to intersection points $y$ in $\hatLio. $   \vskip 5 pt

For later use,  we record an observation about the actions of these chords.   Because the Lagrangians are fiber radial,  the primitive $f_{\tilde{L}_{0}}$ is constant along each component of $\tilde{L}_0 \cap US_{I,\epsilon_{1}} \cap F_y$ (and similarly for $f_{\tilde{L}_{1}}$). After identifying these components with $\alpha_0 \in \tilde{L}_{0,\epsilon_1,y}$ (respectively $\alpha_1 \in \tilde{L}_{1,\epsilon_1,y}$), we denote the resulting value by $f_{\tilde{L}_{0}}( \alpha_0)$(respectively $f_{\tilde{L}_{1}}( \alpha_1)$).  Exactly as in Lemma \ref{lem: sharpactions},  we have that: 

\begin{lem} \label{lem:actionwrappedlem} Suppose that $x_{y, \v}(\alpha_0,\alpha_1) \in US_{I,\epsilon_{1}}$ and $(\epsilon_i^{\operatorname{pert}})$ is a small vector.   By taking: \begin{itemize} \item The $\rho_i$ coordinate of $x_{y, \v}(\alpha_0,\alpha_1)$ sufficiently close to $\frac{\kappa_i}{2\pi} (\epsilon_1+\epsilon_i^{\operatorname{pert}})^2$ for every $i \in I$,  \item the Hamiltonian term $g_L$ to be sufficiently small along $x_{y, \v}(\alpha_0,\alpha_1)$, \end{itemize}  the action of the chord  $x_{y, \v}(\alpha_0,\alpha_1)$ can be taken arbitrarily close to \begin{align} \label{eq:sharpwrappedaction} A_{g_{L}^m}(x_{y, \v}(\alpha_0,\alpha_1)) \approx f_{\tilde{L}_{1}}( \alpha_1)- f_{\tilde{L}_{0}}( \alpha_0) + \sum_i \kappa_i (1-\frac{(\epsilon_1+\epsilon_i^{\operatorname{pert}})^2}{2})(\frac{\alpha_{1,i}}{2\pi}- (v_i+v_{s,i})-\frac{\alpha_{0,i}}{2\pi}) \end{align} 
where we have chosen identifications $\alpha_0=(\alpha_{0,i})$, $\alpha_1=(\alpha_{1,i})$ as in \eqref{eq:toridfc}.  (Note that the right-hand side of \eqref{eq:sharpwrappedaction} is independent of this identification.)
 \end{lem}  
 \begin{proof} Omitted. \end{proof}



\subsubsection{Local product computation} 

 Let  $g_{0,L}$,  $h_{F}$ be Liouville admissible Hamiltonians (of some slopes $\lambda_L$ and $\lambda_F>0$) with respect to two possibly different roundings $X^{L}_m,  X^{F}_m$ of some $\widehat{X}_{(\epsilon^{\operatorname{pert}}_{i})}$ (recall \eqref{eq:skewliouville}).  Then,  as noted at the beginning of this section,  $$g_{L}:=g_{0,L} + h_{F}$$ is also homogeneous.  We assume that $g_{0,L}$ and $g_{L}$ satisfy the conditions right after Proposition \ref{lem:McLeanham} so that all of the non-constant chords are of the form $x_{y, \v} (\alpha_0,\alpha_1)$.  
 
 \begin{lem} \label{lem:distinctactions} After further deforming $\tilde{L}_0$,  choosing $(\epsilon_i^{\operatorname{pert}})$ generically, and taking our Hamiltonians as in Lemma \ref{lem:actionwrappedlem},  we can ensure that all of the non-constant chords of $g_{0,L}$ (respectively $g_L$) have distinct actions and periods.  \end{lem} 
\begin{proof} For any $y \in \mathcal{X}_I (\tilde{L}_0, \tilde{L}_1)$, $\alpha_0 \in \tilde{L}_{0,\epsilon_1,y}$,  we can deform $\tilde{L}_0$ using a Hamiltonian perturbation to slightly modify $\alpha_0$ (over the same $y$) and then apply \cite[Lemma 5.11]{McLeanwrapped} (as in Step 2 of the proof of \cite[Lemma 5.4]{McLeanwrapped}) to make $\tilde{L}_0$ fiber radial near this new point.  Doing this as needed for each point,  we can modify all of the positions of $\tilde{L}_0$ over each $y$.  The result then follows by examining the formula for the approximate action in \eqref{eq:sharpwrappedaction}.   \end{proof} 

In what follows,  we will always assume that actions/periods are separated as in Lemma \ref{lem:distinctactions}. Before proceeding further,  we need to recall that solutions $u$ of \eqref{eq:generalizedFloery} have an associated geometric energy 
\begin{equation}
    \Egeo(u) :=  \frac{1}{2}\int_{\Sigma}|| du - X_{K}||^2
 \end{equation} 
as well as a topological energy
\begin{equation}\label{eq:Etop}
    \Etop(u) = \int_{\Sigma} u^* \omega - d(u^*K).
\end{equation}
We  have a relationship between these  two energies: 
\begin{equation}\label{generalenergyrelationship}
    \Egeo(u)=\Etop(u)- \int_{\Sigma} \mathfrak{R}_K, 
\end{equation} 
where $\mathfrak{R}_K$ is the curvature of the perturbation 1-form $K$.  In local coordinates $(s,t)$ on $\Sigma$,  we have 
\begin{align} \label{eq:localcurv} \mathfrak{R}_K= (\partial_sK(\partial_t)-\partial_tK(\partial_s)+ \lbrace K(\partial_s),K(\partial_t) \rbrace)ds \wedge dt \end{align} 

For our specific application,  we let $\mathring{Z}$ be a strip $Z:= \mathbb{R} \times [0,1]$ with one interior puncture.  Equip this with a positive cylindrical end at the interior puncture and two strip-like ends (one positive and one negative) at the boundary punctures.  Consider two closed one-forms $\beta_1$, $\beta_2$ which agree with $dt$ along the negative strip-like end. We assume \begin{itemize}  \item $\beta_1$ agrees with $dt$ along the cylindrical end and vanishes along the positive strip-like end. 
\item $\beta_2$ vanishes along the cylindrical end and agrees with $dt$ along the positive strip-like end.\end{itemize}  

Set $$K_{\operatorname{inv}}:= h_F\cdot \beta_1+ g_{0,L}\cdot \beta_2.$$

\begin{lem} \label{lem:flatcurve} The curvature of $K_{\operatorname{inv}}$, $\RK,$ vanishes.  \end{lem}
\begin{proof} The third term in \eqref{eq:localcurv} vanishes because the functions $h_F$ and $g_{0,L}$ Poisson commute.  The combination of the first two terms vanish because $\beta_1, \beta_2$ are both closed on $\mathring{Z}.$    \end{proof} 

 Next,  let $H_F$ be a small time-dependent perturbation of $h_F$ of the form discussed in \S \ref{sect: actionspec}.  For simplicity,  in what follows we take our perturbed Hamiltonian $H_F$ to have exactly one degree zero orbit $x_{\v,c}$ in each connected component $U_{\v,c}$ of each isolating set $U_\v$.  Take a small perturbation one-form $K_{\operatorname{pert}}$ supported on the cylindrical end so that for $s>>0$, $K_{\operatorname{inv}}+K_{\operatorname{pert}}=H_{F}dt.$ 

\begin{defn} \label{defn:tripleoutput} Set $K=K_{\operatorname{inv}}+K_{\operatorname{pert}}$.  Fix a surface dependent almost complex structure $J_{\mathring{Z}}$(as usual translation invariant along the ends),  and $x_1 \in \mathcal{X}(X; H_{F})$ , $x_2 \in \mathcal{X}(\tilde{L}_0,\tilde{L}_1; g_{0,L})$ , $x_0 \in \mathcal{X}(\tilde{L}_0,\tilde{L}_1; g_{L})$.  Let $$\mathcal{M}_{\mathring{Z}}(x_0,x_1,x_2)$$ denote the moduli space of solutions to  \eqref{eq:generalizedFloery} such that along the strip-like ends \begin{align}\operatorname{lim}_{s \to \infty}u(\varepsilon_2(s, t)) = x_2 \\\operatorname{lim}_{s \to -\infty}u(\varepsilon_0(s, t)) = x_0   \end{align} 
 and along the cylindrical end 
 \begin{align} \operatorname{lim}_{s \to \infty}u(\varepsilon_1(s, t))= x_1  \end{align} 
\end{defn}  
 
 Lemma \ref{lem:flatcurve} implies that for maps $u \in \mathcal{M}_{\mathring{Z}}(x_0,x_1,x_2),$  $$E_{top}(u) +\hbar \geq E_{geo}(u), $$ where $\hbar > 0$ is a constant that can be taken arbitrarily small.  Concretely this means that for any such solution with inputs $x_1 \in \mathcal{X}(X; H_F)$ , $x_2 \in \mathcal{X}(\tilde{L}_0,\tilde{L}_1; g_{0,L})$ , $x_0 \in \mathcal{X}(\tilde{L}_0,\tilde{L}_1; g_{L}) $, \begin{align} \label{eq:actionm11} A_{g_{0,L}}(x_2)+ A_{H_{F}}(x_1) \leq A_{g_{L}}(x_0) +\hbar \end{align} 

 By counting solutions to the moduli spaces from Definition \ref{defn:tripleoutput} in the usual way (now using a generic choice of surface dependent almost complex structure $J_{\mathring{Z}}$),  we obtain a map: 
\begin{align} \label{eq:m11} \mathfrak{m}_{1,1}: CF^*(X;H_F) \otimes CF^*(\tilde{L}_0,\tilde{L}_1; g_{0,L}) \to CF^*(\tilde{L}_0,\tilde{L}_1; g_{L})  \end{align}
which is compatible with filtrations in view of \eqref{eq:actionm11}.  Recall that for any $\v$,  $\mathcal{F}_\v$ denotes the manifold of orbits of $h_{F}$ with winding number $\v$.  Fix some $x_{y, \v} (\alpha_0,\alpha_1)$ and consider the special situation where $\mathcal{F}_\v \setminus \partial \mathcal{F}_\v$ intersects $x_{y, \v} (\alpha_0,\alpha_1)$ at its starting point.  Denote $A_{g_{L}}(x_{y, \v} (\alpha_0,\alpha_1))$ by $a$.  Fix some small $\hbar$, let $CF_{\geq a-\hbar} ^*(\tilde{L}_0,\tilde{L}_1;g_L)$ denote the chords with action at least $a-\hbar$,  let $CF_{\geq a + \hbar} ^*(\tilde{L}_0,\tilde{L}_1;g_L)$ denote the chords with action at least $a+\hbar$ and consider the quotient complex $$ CF_{(a-\hbar, a+\hbar)} ^*(\tilde{L}_0,\tilde{L}_1;g_L):= CF_{\geq a-\hbar} ^*(\tilde{L}_0,\tilde{L}_1;g_L)/CF_{\geq a+\hbar} ^*(\tilde{L}_0,\tilde{L}_1;g_L).$$ 
For any element $\frak{c} \in CF_{\geq a-\hbar} ^*(\tilde{L}_0,\tilde{L}_1;g_L)$,  we denote its image in $CF_{(a-\hbar, a+\hbar)} ^*(\tilde{L}_0,\tilde{L}_1;g_L)$ by $[\frak{c}]_{(a-\hbar, a+\hbar)}.$ The key local calculation is the following:  

\begin{lem} \label{lem:localcalc} Fix some $x_{y, \v} (\alpha_0,\alpha_1) \subset U_{\v,c}$ and assume that $\mathcal{F}_\v \setminus  \partial \mathcal{F}_\v$ intersects $x_{y, \v} (\alpha_0,\alpha_1)$ at its starting point.  Let $\hbar>0$ be a small number and suppose that the one-form $K_{\operatorname{pert}}$ is sufficiently small.  Then the lowest energy term,  $[\mathfrak{m}_{1,1}]_{(a-\hbar, a+\hbar)}$ takes $|\mathfrak{o}_{x_{\v,c}}| \otimes |\mathfrak{o}_{x_{y, \vec{0}} (\alpha_0,\alpha_1)}|$ isomorphically onto $|\mathfrak{o}_{x_{y, \v} (\alpha_0,\alpha_1)}|$:
\begin{align} \label{eq:lowenergymodule}  [\mathfrak{m}_{1,1}]_{(a-\hbar, a+\hbar)}:  |\mathfrak{o}_{x_{\v,c}}| \otimes |\mathfrak{o}_{x_{y, \vec{0}} (\alpha_0,\alpha_1)}| \cong |\mathfrak{o}_{x_{y, \v} (\alpha_0,\alpha_1)}| \subset CF_{(a-\hbar, a+\hbar)} ^*(\tilde{L}_0,\tilde{L}_1;g_L).  \end{align} \end{lem} 

The proof follows after the preparatory Lemma \ref{lem:econstruction}.  Recall that the Hamiltonian flow on $\mathcal{F}_\v$ generates an $S^1$-action which extends to the entire regularizing neighborhood $U_{|\v|}$ (and in particular each component $U_{|\v|,c}$).  The basic idea of our proof is to ``unwind" (in the sense of \eqref{eq:hamunwind}) the flows of $H_F$ and $g_L$ using the Hamiltonian action on $U_{|\v|}$ so that the contribution \eqref{eq:lowenergymodule} becomes multiplication by a certain element \eqref{eq: identityspin}. To this end,  set \begin{align} \label{eq:hatmsharp} \hat{h}_F=h_F-K_{S^{1}} \\  \hat{H}_F = H_F-K_{S^{1}} \end{align} where $K_{S^{1}}$ is the momentum map for the $S^1$-action in $U_{|\v|}$.  The spinning construction from \eqref{eq:Seidel} identifies $x_{\v,c}$ (and associated orientation lines) with a constant orbit $\hat{x}_{\v,c}$ of $\hat{H}_{F}$ in $U_{\v,c}$.  Because $\hat{x}_{\v,c}$ is a constant orbit,  we can multiply $\hat{H}_{F}$ by a small positive real number $\delta^p$ while keeping $\hat{x}_{\v,c}$ fixed.  Set $$X_{|\v|,c}=U_{|\v|,c} \cap X.$$ Then for $\delta^p$ sufficiently small,  $\hat{x}_{\v,c}$ is the only degree zero orbit of $\delta^p\hat{H}_{F}$ in $X_{|\v|,c}$.  It will be useful for our purposes to modify this Hamiltonian outside of $U_\v$.   First,  we consider a variant of $H_{F}$,  $H_{\v}$, defined by (compare \eqref{eq: Hpert2}) \begin{align} \label{eq:pertuv2} H_{\v}:= \delta_\v \rho_\v h_\v + h_{F} \end{align} This Hamiltonian is only perturbed in $U_\v$ and hence there is some $\epsilon_1'$ close to $\epsilon_1$ such that for $\rho_i>\frac{\kappa_i}{2\pi}(\epsilon_1')^2$,  the Hamiltonian $H_{\v}$ is independent of $\rho_i$.   Now let $\hat{H}_{\v}$ be the spinning of $H_{\v}.$ By construction, we have that for $\rho_i>\frac{\kappa_i}{2\pi}(\epsilon_1')^2$,$$ \delta^p\hat{H}_{\v}= \delta^p \pi v_i \rho_i+ G(\rho_j) $$ where $G(\rho_j)$ is independent of $\rho_i.$ We choose some $\epsilon_1^*$ which is slightly larger than $\epsilon_1'$ and consider functions $b_i(\rho_i)$ such that  \begin{equation} \label{eq:cuttoff3}
   b_i(\rho_i) = \begin{cases} 0 & \rho_i \leq \frac{\kappa_i}{2\pi}(\epsilon_1')^2 \\ c^*- \pi \rho_i & \rho_i \geq \frac{\kappa_i}{2\pi}(\epsilon_1^*)^2 \end{cases}
\end{equation} for some small positive constant $c^*.$ Finally,  set \begin{align} H_{\delta^{p}}= \sum_{i\in I} \delta^p v_i b_i(\rho_i)+ \delta^p\hat{H}_{\v} \end{align}

This modification creates some additional (constant) orbits in $X_{|\v|} \setminus U_\v$,  but we can ensure that for any such orbit $x$,  \begin{align} \label{eq:otherorbits} H_{\delta^{p}}(x)>  H_{\delta^{p}}(\hat{x}_{\v,c}) \end{align} 

Now let $S$ be the thimble domain from \S \ref{subsection:PSSopen} and let $\beta$ be the subclosed one-form over $S$ from that section.  Choose a surface dependent almost complex structure $J_S$ with each $J_{S,z} \in \mathcal{J}(M,\D)$ (as always,  surface independent near $z=\infty$) and consider $\PSS$ thimbles for the Hamiltonian $H_{\delta^{p}}$--- that is to say maps $ u: S \to X_{|\v|,c}$ satisfying: 
\begin{equation} \label{eq:PSSeq4}
    (du - X_{H_{\delta^{p}}} \otimes \beta)^{0,1} = 0
\end{equation}
(where $({0,1})$ is taken with respect to $J_S$) such that 
\begin{align}\label{eq:limitingcondition4}
    \lim_{s \rightarrow -\infty} u(\varepsilon(s,t)) &=  \hat{x}_{\v,c}
    \end{align}

Denote this moduli space by $\mathcal{M}(H_{\delta^{p}}, \hat{x}_{\v,c})$. 

\begin{lem} \label{lem:econstruction} Assume that for every $z \in S$,  $J_{S,z}$ is sufficiently close to some standard $J_0$ (recall Definition \ref{defn: split}) and that $J_S$ is generic.  Then for $\delta^p$ sufficiently small,  the moduli space $\mathcal{M}(H_{\delta^{p}}, \hat{x}_{\v,c})$ is compact.   \end{lem} 
\begin{proof} We need to demonstrate that,  under these assumptions, thimbles do not escape $X_{|\v|,c}$.  The key point is that for $\delta^p$ sufficiently small,  such $\PSS$ solutions have very small energy.  The fact that these solutions have low energy together with the fact that $J_S$ is close to a standard almost- complex structure allows us to employ the monotonicity argument of \cite[Lemma 4.10]{GP2} to conclude that the PSS solutions cannot escape $U_{|\v|, \epsilon'}$(recall \eqref{eq:UIe}) for some small $\epsilon'$.   Note that this lemma is proven by projecting (pieces of) $\PSS$ solutions to suitable holomorphic curves in divisor strata.  In our case,  our deformed Hamiltonian $H_{\delta^{p}}$  is independent of $\rho_i$ when $\rho_i>\frac{\kappa_i}{2\pi}(\epsilon_1^{\ast})^2$,  allowing these arguments to go through.  

Breaking along $\D$ can also be excluded for energy reasons (see the proof of  \cite[Lemma 4.5]{GP2} for a similar argument).  Namely,  in view of \cite[Lemma 2.9]{GP2}, we have that for any $x \in \D$,   $$ H_{\delta^{p}}(x) \geq \delta^p h_F(x) \approx \delta^p \lambda_F( \frac{1}{1-1/2\epsilon_{1}^2}-1)$$ In particular, for any $x \in \D$,  $H_{\delta^{p}}(x) > H_{\delta^{p}}(\hat{x}_{\v,c}).$ Suppose we had a PSS solution $u_1$ asymptotic to such an $x$.  Then the topological energy of such a solution is \begin{align} E_{top}(u_1)= \int_{u_1} u_1^*(\omega)+ H_{\delta^{p}}(x)  \end{align}  If  $ \int_{u_1} u_1^*(\omega) < 0$,  this is negative when $\delta^p$ small (hence impossible) and when  $\int_{u_1} u_1^*(\omega) \geq 0$,  this is larger than the topological energy of solutions in $\mathcal{M}(H_{\delta^{p}}, \hat{x}_{\v,c}).$ Therefore such solutions cannot break along $\D$ and hence remain in $X_{|\v|}$.  Similarly,  there is no breaking along other orbits in $X_{|\v|}\setminus U_\v$ in view of \eqref{eq:otherorbits}.  Finally,  breaking inside of $U_\v$ does not occur for generic $J_S$ and so Gromov compactness implies that $\mathcal{M}(H_{\delta^{p}}, \hat{x}_{\v,c})$ is compact as claimed.   \end{proof}

It follows that,  after taking $J_S$ generic (but still close to a split almost-complex structure), the count of index zero thimbles gives rise to a well-defined map: \begin{align} \label{eq: identityspin} e: \K \to |\mathfrak{o}_{\hat{x}_{\v,c}}| \end{align}  

\begin{rem} In the case of a single smooth divisor $\D$ and $\v \neq 0$,  the Hamiltonian $H_{\delta^{p}}$ is essentially a $\bigvee$-shaped Hamiltonian of small slope and the construction of the element \eqref{eq: identityspin} matches the construction of the unit in Rabinowitz Floer (co)homology.    \end{rem}

\begin{proof} (\emph{Of Lemma \ref{lem:localcalc}}:) We now use the map  \eqref{eq: identityspin} to analyze the module structure \eqref{eq:m11}.  First,  we apply Gromov compactness to argue that provided $\hbar$,  $K_{\operatorname{pert}}$ are sufficiently small,  solutions contributing to \eqref{eq:lowenergymodule} do not escape the isolating neighborhood $U_{\v,c}$.   To see this,  consider a sequence of $H_{F,n}$ tending to $h_F$ and a corresponding $K_{\operatorname{pert},n} \to 0$.  Let $u_n$ be Floer solutions with respect to the forms $K_{\operatorname{inv}}+ K_{\operatorname{pert},n}$ with inputs in $|\mathfrak{o}_{x_{\v,c}}|$ and $|\mathfrak{o}_{x_{y, \vec{0}} (\alpha_0,\alpha_1)}|$ and which have energy less than the action gap of $g_L$.  Then (after passing to a subsequence) $\lbrace u_n \rbrace$ $C^\infty_{loc}$-converges to solutions of zero topological energy,  which are necessarily constant by Lemma \ref{lem:flatcurve}.  Suppose that on each $u_n$ there is some point $z_n$ which intersects the boundary of $U_{\v,c}$.  The points $z_n$ must escape to infinity along the cylindrical or strip-like ends.  Thus,  after passing to a subsequence,  we can write $z_n=(s_n, t_n)$ where $(s,t)$ are coordinates along the end.  Then by rescaling by $s_n$ and passing to a subsequence,  we can produce a Floer cylinder or strip which crosses the boundary of $U_{\v,c}$.  As this solution must also have zero energy,  this is a contradiction.  Next,  let $\hat{H}_F$ be as in \eqref{eq:hatmsharp} and set  \begin{align*} \hat{g}_L= g_L-K_{S^{1}}. \end{align*} Using another ``spinning argument"(see e.g.  \cite[Proposition 3.8]{MR2563724}),  we obtain that \eqref{eq:lowenergymodule} is equivalent to an operation \begin{align} \label{eq:lowenergymodule2} [\hat{m}_{1,1}]_{a'-h,a'+h}:  |\mathfrak{o}_{\hat{x}_{\v,c}}| \otimes |\mathfrak{o}_{x_{y, \vec{0}} (\alpha_0,\alpha_1)}| \to CF^*(\tilde{L}_0,\tilde{L}_1;\hat{g}_L)_{(a'-h,a'+h)}  \end{align} Where $\hat{x}_{\v,c}$ is the orbit in $U_{\v,c}$ corresponding to $x_{\v,c}$ (now a constant orbit) and $a'=A_{\hat{g}_{L}}(\hat{x}_{y, \v}(\alpha_0,\alpha_1))$.  
 Let $\hat{x}_{\v,c}$ be the orbit of $\hat{H}_F$ in $U_{\v,c}$ corresponding to $x_{\v,c}$ (now a constant orbit).   As above,  if $\delta^p$ is sufficiently small this is the only index zero orbit of $\delta^p\hat{H}_{F}$ in $X_{|\v|,c}.$ We modify the Hamiltonian $\delta^p\hat{H}_F$ to $H_{\delta^{p}}$ and set $g_{\delta^p,L}= g_{0,L}+ H_{\delta^{p}}$.  Then \eqref{eq:lowenergymodule2} becomes identified with a map \begin{align} \label{eq:lowenergymodule3} [\hat{m}_{1,1}]_{a''-h,a''+h}:  |\mathfrak{o}_{\hat{x}_{\v,c}}| \otimes |\mathfrak{o}_{x_{y, \vec{0}} (\alpha_0,\alpha_1)}| \to CF^*(\tilde{L}_0,\tilde{L}_1; g_{\delta^{p},L})_{(a''-h,a''+h)}  \end{align} Assume our complex structures are sufficiently close to some standard $J_0$ and let $ e: \K \to |\mathfrak{o}_{\hat{x}_{\v,c}}| $ denote the map from \eqref{eq: identityspin}.  The composition $[\hat{m}_{1,1}]_{(a''-h,a''+h)} (e\otimes \operatorname{id})$ produces a map $$ m_e: |\mathfrak{o}_{x_{y, \vec{0}} (\alpha_0,\alpha_1)}|   \to CF^*(\tilde{L}_0,\tilde{L}_1; g_{\delta^{p},L})_{(a''-h,a''+h)}.$$ The composition is defined by counting broken configurations of thimbles $S$ and maps from $\mathring{Z}$ which match up at the punctures and can thus be identified with a continuation map by performing the standard gluing construction along the cylindrical ends.  Finally,  we deform the resulting family of Hamiltonians over the strip to a constant family of Hamiltonians (and our complex structures as needed),  concluding the proof.  \end{proof}  

\begin{rem} We sketch an alternative approach to Lemma \ref{lem:localcalc} using Morse-Bott models for $HF^*(X;H_F)$ and the map $\mathfrak{m}_{1,1}$.  Consider (an outward pointing) Morse function $\hD$ on $D_\v^{c_{0}}$ with a unique index zero critical point $\bar{x}_{crit}$ at the point in the base where $x_{y, \v} (\alpha_0,\alpha_1)$ projects.   $\pi_I^*(\hD)_{|\mathcal{F}_{\v}}$ then has a torus $T^{|\v|}_{\bar{x}_{crit}}$ worth of critical point over $\bar{x}_{crit}$.  We can perturb $h_F$ using this function to a time-independent function $H_{MB}$ which has a Morse-Bott set of orbits over $\bar{x}_{crit}$ which we also denote by $T^{|\v|}_{\bar{x}_{crit}}.$ 

Now take $g_{0,L}= g_{L}-H_{MB}.$ Then there is a map $u: \mathring{Z} \to X$ of zero geometric/topological energy with inputs the fundamental cycle on $[T^{|\v|}_{\bar{x}_{crit}}]$, $x_{y, \vec{0}} (\alpha_0,\alpha_1)$,  and output $x_{y, \v} (\alpha_0,\alpha_1)$ that  should contribute to a Morse-Bott model for $\mathfrak{m}_{1,1}.$ By perturbing this solution to the non-degenerate setting,  one could then potentially obtain an alternative justification for \eqref{eq:lowenergymodule}.  However,  making this approach precise would require developing more Morse-Bott theory than we wish to undertake here.   \end{rem} 

\subsubsection{Proof of Theorem}  

\begin{defn} A homogeneous collection is a sequence of functions $g^{m}_L$ (as before indexed by $m \in \mathbb{Z}^{>0}$) so that there exists $K_{\vec{\epsilon}}$ with $\operatorname{min}_{\D} R^{\vec{\epsilon}}>K_{\vec{\epsilon}}>1$ such that: 
       \begin{enumerate} \item \label{item:hammy1} On each subset $U_{I,\epsilon_{1}+\delta^{\operatorname{max}}} \setminus (\cup_{j\notin I} U_{j,\epsilon_{1}+\delta^{\operatorname{max}}}),$ $g^m_L=g^m_L(\rho_i)$ is a function of the variables $\rho_i$ with $i \in I$.   \item When $R^{\vec{\epsilon}} \geq K^{\vec{\epsilon}}$,  $g^{m'}_L \geq g^m_L$ whenever $m' \geq m$ and \begin{align} \label{eq:hcond100} \theta(X_{g^m_{L}})=g^m_L + \lambda_m^{L} \end{align} for some increasing sequence of $\lambda_m^{L}$ tending to infinity.   \item \label{item:hammy4} For any $\lambda>0$,  there exists a $g^m_L$ such that $g^m_L > \lambda(R^{\vec{\epsilon}}-1)$ when $R^{\vec{\epsilon}} \geq K^{\vec{\epsilon}}.$ 
       \item Each $g^m_L$ is a non-positive function of small norm on $\hatLio$.   \end{enumerate}    \end{defn} 
       
     For any $m' \geq m,$ there are continuation maps: 
$$ \mathfrak{c}_{m,m'}:  HF^*(L_0, L_1; G^m_L) \to HF^*(L_0,L_1; G^{m'}_L) $$
 This is again a consequence of the integrated maximum principle. Let $\rho(s)$ be a non-negative, monotone non-increasing cutoff function such that 
\begin{equation} \label{eq:cuttoff}
    \rho(s) = \begin{cases} 0 & s \gg 0 \\ 1 & s \ll 0 \end{cases}
\end{equation} 
Set
\begin{align} 
    \label{eq: concretehomotopy} G_{s,t}= (1-\rho(s))G^m_L +\rho(s)G^{m'}_L 
\end{align} 
For $R^{\vec{\epsilon}}$ large we have that this is a monotone homotopy and we now have the following
\begin{align}\label{eq:hcond1} 
    \theta(X_{G_{s,t}})&=G_{s,t}+\lambda_s, \\
    \label{eq:hcond2}  dR^{\vec{\epsilon}}(X_{G_{s,t}}) &= 0 
\end{align} 

Choose $J_{s,t}$ to be of contact type, the integrated maximum principle again applies to show that these continuation maps are well-defined.  It is also not difficult to see that the wrapped Floer cohomology groups \eqref{eq:wrappedlim} can be computed as : \begin{align} WF^*(L_0,L_1)  \cong \varinjlim_m HF^*(L_0, L_1; G_L^m)  \end{align}
To show this,  choose a sequence of $h^{\lambda}$ which are admissible in the sense of Definition \ref{defn:linearf} and which are linear for $R\geq K^{\vec{\epsilon}}$.   We can use the same argument as above to produce continuation maps between the directed systems $\lbrace HF^*(L_0, L_1; H_L^\lambda) \rbrace$ ($H_L^\lambda$ are small perturbations of the $h^{\lambda}$) and the directed system $\lbrace HF^*(L_0,L_1; G_L^m) \rbrace.$ Standard gluing and homotopy arguments show that these maps induce isomorphisms on the direct limits.  \vskip 5 pt   

Now assume that we have deformed $L_0, L_1$ to be fiber radial and that, for each $m,$ our Hamiltonians $g^m_L$ are taken to satisfy the conditions listed after  Proposition \ref{lem:McLeanham} so that the non-constant chords are of the form $x_{y,\v}(\alpha_0, \alpha_1)$.  Observe that the chords between $\tilde{L}_0$ and $\tilde{L}_1$ then exhibit a certain ``periodic" structure --- a single short chord $x_{y, \vec{0}}(\alpha_0,\alpha_1)$ gives rise to an entire collection of chords,  $x_{y,\v}(\alpha_0, \alpha_1)$.  The basic idea of our proof is that this periodic structure of these chords is reflected in the $SH^0(X)$ module structure.  If $x_{y, \vec{0}}(\alpha_0,\alpha_1)$ lies in $U_{|\v|,c}$,  then to lowest order,  the action of $\theta_{(\v,c)}$ in symplectic cohomology should take (a generator corresponding to) $x_{y, \vec{0}}(\alpha_0,\alpha_1)$ to (a generator corresponding to) $x_{y,\v}(\alpha_0, \alpha_1)$.  As in the previous section,  this can be formalized using filtrations and spectral sequences.  For any $\alpha_0,\alpha_1$ choose an identification as in \eqref{eq:toridfc} so that $\alpha_0=(\alpha_{0,i})$, $\alpha_1=(\alpha_{1,i})$ and define: \begin{align} \label{eq:wL} w_L(x_{y, \v}(\alpha_0,\alpha_1)) =  \frac{-f_{\tilde{L}_{1}}( \alpha_1)+ f_{\tilde{L}_{0}}( \alpha_0)}{1- \epsilon_1^2/2} + \sum_i \kappa_i (\frac{\alpha_{0,i}}{2\pi}+ (v_i+v_{s,i})-\frac{\alpha_{1,i}}{2\pi}) \end{align} 

Given a non-positive integer $p$, we let $w_L^p$ denote the $-p$-th largest value among the numbers $w_L(x_{y, \v}(\alpha_0,\alpha_1)),$ listed in order.  Sending  $\epsilon_i^{\operatorname{pert}} \to 0$,  the approximation from Lemma \ref{lem:actionwrappedlem} becomes:\begin{align} \label{eq:wrappedaction} A_{g_{L}}(x_{y, \v} (\alpha_0,\alpha_1)) \approx  -(1- \epsilon_1^2/2)w_L(x_{y, \v}(\alpha_0,\alpha_1))
    \end{align}

 We can therefore choose our Hamiltonian $g_L^m$ so that $w_L(x_{y, \v}(\alpha_0,\alpha_1))$ defines a descending filtration,  $F^pCF^*(\tilde{L}_0,\tilde{L}_1; g_L^m)$, ($p$ ranging over non-positive integers) by setting \begin{align} F^pCF^*(\tilde{L}_0,\tilde{L}_1; g_L^m) := \bigoplus_{x_{y, \v} (\alpha_0,\alpha_1),w_{L} \leq w_{L}^p} |\mathfrak{o}_{x_{y, \v} (\alpha_0,\alpha_1)}| \end{align}  Next,  set  \begin{align} CF^*(\tilde{L}_0,\tilde{L}_1; g_L^m)^p := \frac{F^pCF^*(\tilde{L}_0,\tilde{L}_1; g_L^m)}{F^{p+1}CF^*(\tilde{L}_0,\tilde{L}_1; g_L^m)} \end{align} and take $HF^*(\tilde{L}_0,\tilde{L}_1; g_L^m)^p := H^*(CF^*(\tilde{L}_0,\tilde{L}_1; g_L^m)^p).$  As in  \eqref {eq: chainlevel}, defined the (filtered) wrapped co-chains by the formula:  \begin{align} \label{eq: chainlevelwrapped} WC^*(\tilde{L}_0,\tilde{L}_1;\lbrace g_L^m \rbrace):=
     \bigoplus_{\ell}CF^*(\tilde{L}_0,\tilde{L}_1; g_L^m)[q] 
 \end{align} 
 with differential given by the same formula as \eqref{eq: directlimitchain}. We have $$H^*(WC^*(\tilde{L}_0,\tilde{L}_1; \lbrace g_L^m \rbrace)) \cong WF^*(\tilde{L}_0,\tilde{L}_1).$$   Suppose that whenever $m'>m$,  $g_L^{m'} \geq g_L^{m}$ ($\forall x \in X$ not just at $\infty$),  ensuring that \eqref{eq: concretehomotopy} is a monotone homotopy.  Then the induced continuation maps preserve the filtration by $w_L(x_{y, \v}(\alpha_0,\alpha_1))$ and there is an extension of this filtration, $F^pWC^*(\tilde{L}_0,\tilde{L}_1; \lbrace g_L^m \rbrace),$ to wrapped co-chains.  We therefore have a spectral sequence: 
 \begin{align} \label{eq:wrappedss}  E^{pq}_r(\tilde{L}_0, \tilde{L}_1) => WF^*(\tilde{L}_0,\tilde{L}_1)  \end{align} 

Finally,  suppose that  \begin{equation} \label{eq:limitylimity}
   \operatorname{lim}_{m \to \infty} g_L^m(x)= \begin{cases} 0 & x \in \hatLio \\ \infty & x \notin \hatLio \end{cases}
\end{equation}   
  
 Then,  if one has two sequences of functions $g_{L}^m$ and $\tilde{g}_{L}^m$,  \eqref{eq:limitylimity} implies that for any $g_{L}^m$,   one can always find a $\tilde{g}_{L}^{m'}$ which dominates it (and vice-versa).  This allows one to construct a filtered comparison map on the wrapped co-chains: 
$$ WC^*(\tilde{L}_0,\tilde{L}_1; \lbrace g_{L}^m \rbrace) \cong WC^*(\tilde{L}_0,\tilde{L}_1; \lbrace \tilde{g}_{L}^m \rbrace) $$ which means that the spectral sequence is independent of this choice.   By construction,  the first page of the spectral sequence is given by 
\begin{align} \label{eq:wrappedE1}   \bigoplus_q E^{pq}_1(\tilde{L}_0,\tilde{L}_1) \cong \varinjlim_m HF^*(\tilde{L}_0,\tilde{L}_1;g_L^{m})^p  \end{align}

 The argument of Lemma \ref{lem:distinctactions} shows that by further perturbing $\tilde{L}_0$ by a small isotopy and possibly shifting $\epsilon_1$,  we can assume that $w_L(x_{y, \v}(\alpha_0,\alpha_1))$ and $w_L(x_{y', \v'}(\alpha'_0,\alpha'_1))$ are distinct except when  $$y=y', \alpha_0=\alpha'_0, \alpha_1=\alpha'_1, w(\v)=w(\v') $$ 

In particular,  we can assume that the orientation lines $|\mathfrak{o}_{x_{y, \vec{0}} (\alpha_0,\alpha_1)}|$ define subspaces of the $E_1$ page.  We now turn to constructing a filtered action $SH^0(X,\K)$ on $WF^*(\tilde{L}_0,\tilde{L}_1)$ which will induce an action of $\mathcal{SR}_\K(M,\D)$ on the pages of \eqref{eq:wrappedss}.   To do this, we will want to work want to define $SH^*(X)$ using homogeneous collections--- in particular our Hamiltonians orbits get closer to $\hatX$ rather than getting progressively closer to $\D$ as was the case in in \S \ref{sect: actionspec}.  The obvious analogue of \eqref{eq:wrappedss} is a spectral sequence \begin{align} \label{eq:homspec} E_{r,\operatorname{hom}}^{pq} => SH^*(X) \end{align} which is superficially different from the spectral sequence in Lemma \ref{lem:spectralsequ}.  However,  there is a simple ``rescaling trick" which enables us to easily relate this setup to the one from \S \ref{sect: actionspec}.  In particular,  we have: 
\begin{lem} $$E^{p,-p}_{r,\operatorname{hom}} \cong \mathcal{SR}_\K(M,\D)$$  \end{lem} 
\begin{proof} We let $\psi_t(x)$ denote the time $\operatorname{ln}(t)$ flow of the Liouville field for any pair $(x,t)$, $x \in X,t \in \mathbb{R}^{\geq 0}$,  for which the flow is well-defined.  Let $\bar{X}_m$ and $\hlm$ denote the sequence of Liouville domains and Hamiltonians from \S \ref{sect: actionspec}.  We also perturb $\hlm$ to a Hamiltonian $H^m$,  which is perturbed exactly as before in each of the isolating sets $U_\v$ but no longer perturbed near $\D$ (as we do not need to consider curves which pass through $\D$).  In a slight abuse of notation, we use the same notation for these Hamiltonians.  Next set $t_m=\frac{1-\frac{1}{2}\epsilon_1^2}{1-\frac{1}{2}\epsilon_m^2}$ and let $X^{\#}_m$ be the rescaling of $\bar{X}_m$, $$X^{\#}_m:=\psi_{t_{m}}(\bar{X}_m)$$  so that its boundary is $C^0$ close to $\hatX$ (the rounding gets sharper as $m$ gets larger).  We let \begin{align} \label{eq:rescaledhm} h_{\#}^m:= e^{t_{m}} h^m \circ \psi_{-t_{m}} \\ H^m_{\#}=e^{t_{m}} H^m \circ \psi_{-t_{m}} \end{align} extended to all of $X$ by linearity along the collar region.  Pulling back Hamiltonian orbits and Floer solutions gives rise to a canonical bijection of Floer complexes: 
\begin{align} \label{eq:Floerrescale} CF^*(X\subset M;H^m) \cong CF^*(X \subset M; H_{\#}^m) \end{align} 
where the first complex defined with respect to some almost complex structure $J_t$ and the second is defined with respect to any almost complex structure $\tilde{J}_t$ on $M$ which agrees with $\psi_{t_{m},*}(J_t)$ on $\psi_{t_{m}}(X \setminus V_{0,m}).$ Note that the complex on the right-hand side does not depend on the choice of extension $\tilde{J}_t$ by the integrated maximum principle.  The isomorphism from \eqref{eq:Floerrescale} allows us to invoke the consequences of \S \ref{subsection: stabilizedPSS} to show that the spectral sequence degenerates in degree zero.  \end{proof}   



The basic theory of spectral sequences shows that $\mathfrak{m}_{1,1}$ from \eqref{eq:m11} gives rise to a module structure: 
$$\mathfrak{m}_{1,1}: E^{pq}_{r,\operatorname{hom}}\otimes \bigoplus_{p,q} E^{pq}_r(\tilde{L}_0,\tilde{L}_1) \to \bigoplus_{p,q}E^{pq}_r(\tilde{L}_0,\tilde{L}_1) $$
where $E^{pq}_{r,\operatorname{hom}}$ denotes the pages of the symplectic cohomology spectral sequence. from \eqref{eq:homspec}.  Restricting this to the action of the degree zero part, $E^{p,-p}_{r,\operatorname{hom}}$,  and using the identification $E^{p,-p}_{r,\operatorname{hom}} \cong \mathcal{SR}_\K(M,\D)$,  yields an action: 
\begin{align} \mathfrak{m}_{1,1}: \mathcal{SR}_\K(M,\D) \otimes \bigoplus_{p,q} E^{pq}_r(\tilde{L}_0,\tilde{L}_1) \to \bigoplus_{p,q} E^{pq}_r(\tilde{L}_0,\tilde{L}_1) \end{align}

We have finally collected all of the ingredients needed to prove the finiteness of wrapped Floer cohomology groups as modules over $SH^0(X)$: 

\begin{thm} \label{thm:properness} For any affine log Calabi-Yau variety $X$ and any pair of Lagrangian branes $L_0,L_1$, the wrapped Floer groups $WF^*(L_0,L_1)$ are finitely generated modules over $SH^0(X)$.  \end{thm}
\begin{proof}  Deform $L_0,L_1$ them to be fiber radial Lagrangians $\tilde{L}_0,\tilde{L}_1$ and consider the resulting spectral sequence $$ E_r^{pq}(\tilde{L}_0,\tilde{L}_1) => WF^*(\tilde{L}_0,\tilde{L}_1). $$ The key claim is that the $E_1$ page is finitely generated as $\mathcal{SR}_\K(M,\D)$ modules --- more precisely,  we have that $\bigoplus_{p,q} E_1^{pq}$ is generated as a $\mathcal{SR}_\K(M,\D)$ module by the orientation lines $|\mathfrak{o}_{x_{y, \vec{0}} (\alpha_0,\alpha_1)}|$ associated to short chords (including orientation lines associated to interior intersection points).  To see this,  note that by the description of the first page,  every element which is homogeneous with respect to the $p$-grading lies in some $HF^*(\tilde{L}_0,\tilde{L}_1;g_L^{m})^p$.  The chain complex $CF^*(\tilde{L}_0,\tilde{L}_1;g_L^{m})^p$ is freely generated as a $\K$-vector space by the elements $|\mathfrak{o}_{x_{y, \v} (\alpha_0,\alpha_1)}|$ where $x_{y, \v} (\alpha_0,\alpha_1)$ is a chord with $w_L(x_{y, \v} (\alpha_0,\alpha_1))=w_L^p$.  We can assume that $h^{m}_{F},$ $g_{0,L}^{m}$ are both Liouville admissible with respect to possibly different roundings of some  $\widehat{X}_{(\epsilon_i^{\operatorname{pert}})_{m}}$.   Lemma \ref{lem:distinctactions} shows that by choosing $\epsilon_i^{\operatorname{pert}}$ generically,  we can assure that the chords with $w_L(x_{y, \v} (\alpha_0,\alpha_1))=w_L^p$ all arise in distinct action windows.    

Now fix a $x_{y, \v} (\alpha_0,\alpha_1)$.  We can deform $h^{m}_{F}$ to some $\tilde{h}^{m}_{F}$ (and $g_L^m$ to some $\tilde{g}^{m}_{L}$)  so that $\mathcal{F}_\v$ intersects $x_{y, \v} (\alpha_0,\alpha_1)$ at its starting point,  while ensuring that the continuation maps $$CF^*(\tilde{L}_0,\tilde{L}_1;g_L^{m})^p \to CF^*(\tilde{L}_0,\tilde{L}_1;\tilde{g}_L^{m})^p $$ preserve the distinct action windows.  Note that $\theta_{(\v,c)} \in E_{r,\operatorname{hom}}^{p,-p}$ is represented by a generator of the orientation line $|\mathfrak{o}_{x_{\v,c}}|.$ (As before $c$ denotes the component $U_{\v,c}$ of $U_\v$ where $x_{y, \vec{0}} (\alpha_0,\alpha_1)$ lies.) In view of Lemma \ref{lem:localcalc},  it follows that on the level of the spectral sequence we have $$ \mathfrak{m}_{1,1}(\theta_{(\v,c)} \otimes z_{x_{y, \vec{0}} (\alpha_0,\alpha_1)}) = z_{x_{y, \v} (\alpha_0,\alpha_1)} + \ldots \in HF^*(\tilde{L}_0,\tilde{L}_1;g_L^{m})^p, $$ where $z_{x_{y, \vec{0}} (\alpha_0,\alpha_1)}$ is a generator of $|\mathfrak{o}_{x_{y, \vec{0}} (\alpha_0,\alpha_1)}|$,  $z_{x_{y, \v} (\alpha_0,\alpha_1)}$ is a generator of $|\mathfrak{o}_{x_{y, \v} (\alpha_0,\alpha_1)}|$ (determined by the choice of $z_{x_{y, \vec{0}} (\alpha_0,\alpha_1)}$),  and  $\ldots$ denotes terms with higher action $A_{g_L^m}$.   Using  induction on the action level,  we see that the elements in $|\mathfrak{o}_{x_{y, \v} (\alpha_0,\alpha_1)}|$ are themselves elements of the $E_1$ page and in the module generated by the short chords.  As these generate all classes additively,  we obtain the desired result.   

To finish the argument,  note that since $\mathcal{SR}_\K(M,\D)$ is finitely generated (and in particular Noetherian),  it follows that the remaining pages are also finitely generated as $\mathcal{SR}_\K(M,\D)$ modules, concluding the proof.   \end{proof} \vskip 5 pt

\emph{(Proof of Theorem \ref{thm:conj1}):}  Part \ref{item:finitelygen} is Part \ref{item:finitelem1} of Lemma \ref{lem:finiteness}.  Part \ref{item:finitelygen2} is Part  \ref{item:finitelem2} of the same Lemma.  Finally, Part   \ref{item:finitelygen3} is Theorem \ref{thm:properness}.

\section{Applications} \label{section:applications} 

\subsection{Homological properties of $\mathcal{W}(X)$}

This section is purely expository and can be skipped by experts.  Its purpose is to introduce some fundamental concepts from homological algebra and then state some recent results showing how these concepts apply in the context of Fukaya categories. As before $\K$ will denote a coefficient field and all our categories will be small and linear over $\K$.

We will carry out our homological algebra in the framework of dg-categories and assume the reader is familiar with the definitions of dg categories and their functors as well as derived categories of modules and bimodules over dg categories from \cite{MR2275593, MR2028075}. Of course, it is well-known that the Fukaya category is naturally an $A_\infty$ category rather a dg-category. However, for any (small) $A_\infty$ category $\mathcal{D}$, the Yoneda embedding provides a quasi-isomorphism onto a dg-subcategory of $\operatorname{Mod}(\mathcal{D})$ \cite[Chapter~1]{Seidel_PL} and so the two settings are equivalent for our purposes. Given an $A_\infty$ category $\mathcal{D}$, we let $\operatorname{Perf}(\mathcal{D}) \subset \operatorname{Mod}(\mathcal{D})$ denote the category of perfect modules, which is a pre-triangulated idempotent closed dg-category.  

Recall the Hochschild (co)homology groups of dg-categories: \begin{align} \HH^*(\mathcal{C}) := H^*(R\operatorname{Hom}_{\mathcal{C} \otimes \mathcal{C}^{op}}(\mathcal{C},\mathcal{C}))  \\  \HH_*(\mathcal{C}):= H^*(\mathcal{C} \otimes_{\mathcal{C}\otimes \mathcal{C}^{op}}^L \mathcal{C}) \end{align} 

  If $\mathcal{C}:=\Perf(\mathcal{D})$ for some $A_\infty$ category $\mathcal{D}$, we will use sometimes use the notations $\HH^*(\mathcal{D})$ and $HH_*(\mathcal{D})$. It is immediate from the definition that $\HH^*(\mathcal{C})$ is a (unital) ring. Slightly less obvious is the fact that it is graded-commutative. For any object $L \in \operatorname{Ob}(\mathcal{D})$ , there is a ring map \begin{align} \label{eq:bulkboundary} \mathcal{B}_L: \HH^*(\mathcal{D}) \to H^*(\operatorname{Hom}_{\mathcal{D}}^\bullet(L,L)).\end{align} Thus, we have that the Hom groups of the cohomological category (taking all cohomology groups not just degree zero) all have the structure of a (graded) module over $\HH^*(\mathcal{D}).$  

For many purposes, it is convenient to use the more explicit chain level models for Hochschild invariants. These are the so-called ``bar-complexes" $(\operatorname{CC}_*(\mathcal{C}),\partial)$ (or $(\operatorname{CC}^*(\mathcal{C}),\partial)$) that compute $\HH_*(\mathcal{C})$ (or $\HH^*(\mathcal{C})$). Ignoring differentials, the complex $\operatorname{CC}_*(\mathcal{C})$ is a sum of components of the form 
\begin{align}  \bigoplus_{X_{0},\cdots X_{i} \in \operatorname{Ob}(\mathcal{C})} \mathcal{C}(X_i,X_0)\otimes_k \cdots \otimes_k \mathcal{C}(X_{i-1},X_i) \end{align} where $i$ is a non-negative integer.  Again ignoring differentials, the Hochschild cochains $\operatorname{CC}^*(\mathcal{C})$ are a product over components of the form: 
\begin{align}  \prod_{X_{0},\cdots X_{i} \in \operatorname{Ob}(\mathcal{C})}\operatorname{Hom}_k (\mathcal{C}(X_0,X_1)\otimes_k \cdots \otimes_k \mathcal{C}(X_{i-1},X_i), \mathcal{C}(X_0,X_i)) \end{align} Using these models, one constructs a ``cap product" action (see \cite{MR3445567} for a convenient reference) $$\operatorname{CC}^*(\mathcal{C})\otimes \operatorname{CC}_*(\mathcal{C}) \to \operatorname{CC}_*(\mathcal{C})$$ given by taking a tensor of the form $\phi \otimes a_0\otimes a_1\otimes \cdots \otimes a_i$, where $\phi$ has arity $k \leq i$, to the following element:  \begin{align} \phi \otimes a_0\otimes a_1\otimes \cdots \otimes a_i \to (-1)^{\deg(\phi)(\sum_{j=1}^k (\deg(a_j)+1))} a_0\phi(a_1\otimes \cdots \otimes a_k)\otimes a_{k+1}\otimes \cdots \otimes a_i  \end{align} 

It is well-known that the cap product induces a module structure \begin{align} \label{eq:hochmod}  \HH^*(\mathcal{C})\otimes \HH_*(\mathcal{C}) \to \HH_*(\mathcal{C}) \end{align} 

The bar complexes are also very convenient for discussing functoriality of Hochschild invariants. In particular, it is easy to see that Hochschild homology is covariant. Given a dg-functor $F: \mathcal{A} \to \mathcal{C},$ there is an induced map \begin{align} F_*: \HH_*(\mathcal{A}) \to \HH_*(\mathcal{C}) \end{align} induced by the naive inclusion $\operatorname{CC}_*(\mathcal{A}) \subseteq \operatorname{CC}_*(\mathcal{C})$.
Similarly, one sees easily that Hochschild cohomology is contravariant for (fully-faithful) inclusions of dg-subcategories, namely given such an inclusion $F: \mathcal{A} \to \mathcal{C},$  there is an induced map
\begin{align} F^*:\HH^*(\mathcal{C}) \to \HH^*(\mathcal{A}). \end{align} 
The map $F^*$ is given by taking a cochain $\phi \in \operatorname{CC}^*(\mathcal{C})$, and applying it to tuples where the objects $X_j$ all lie in the subcategory $\mathcal{A}.$ This pull-back turns out to be a ring map \cite{Kellerdih} meaning (again in the setting of fully faithful inclusions) that it turns $\HH_*(\mathcal{A})$ into a module over $\HH^*(\mathcal{C}).$ We have the following straightforward observation: 
\begin{lem} The pull-back makes $F_*$ into an $\HH^*(\mathcal{C})$ module homomorphism.  \end{lem} 
\begin{proof} This is an explicit computation using \eqref{eq:hochmod} together with the explicit bar models for $F_*$ and $F^*$. 
\end{proof}

  We next recall some basic concepts from non-commutative (algebraic) geometry. These are properties of dg-categories which are modelled on similar properties in classical algebraic geometry.  \begin{defn} A dg category $\mathcal{C}$ is smooth if $\mathcal{C}$ is perfect as a $(\mathcal{C},\mathcal{C})$ bimodule; i.e. it lies in the subcategory split generated by tensor products of left and right Yoneda (= representable) modules. \end{defn}
   
  The smoothness of a finitely generated $\K$-algebra $R$ is equivalent to the smoothness of its category of perfect complexes, $\Perf(R)$. 
  
  Given any perfect $(\mathcal{C},\mathcal{C})$ bimodule, $\mathcal{B}$, there is a dual bimodule $\mathcal{B}^{!}$ such that if $\mathcal{N}$ is any other bimodule \begin{align} \label{eq:perfduality} R\Hom_{\mathcal{C}-\mathcal{C}}(\mathcal{B}^{!},\mathcal{N}) \cong \mathcal{B}\otimes^{L}_{\mathcal{C}-\mathcal{C}} \mathcal{N} \end{align} 

The main case which will be relevant for us is when $\mathcal{C}$ is a category with a single object (e.g. a dg-algebra) when $$\mathcal{C}^{!} \cong R\Hom_{\mathcal{C}\otimes \mathcal{C}^{op}}(\mathcal{C},\mathcal{C}\otimes\mathcal{C}^{op}) $$

where the right hand side is viewed as a $(\mathcal{C},\mathcal{C})$-bimodule. The following structures were initially introduced by Ginzburg \cite{GinzburgCY} in the setting of algebras and then further generalized to dg (and $A_{\infty}$) categories \cite{MR2795754, Ganatra}.

 \begin{defn} \label{defn: weakCY} We say that a smooth dg category is weakly Calabi-Yau of dimension $n$ (hereafter referred to as ``n-CY") if it is equipped with the data of a bi-module quasi-isomorphism \begin{align} \label{eq:CYcondition} \eta: \mathcal{C}^{!} \cong \mathcal{C}[-n] \end{align}  \end{defn} 

\begin{rem} The adjective ``weakly" in Definition \ref{defn: weakCY} is standard in the literature, where the term ``Calabi-Yau category of dimension n" is reserved for an isomorphism as in \eqref{eq:CYcondition} which lifts in a suitable sense to cyclic homology. This extra layer of complexity does not have any obvious implications for the present circle of ideas and so we stick to the simpler, weaker notion.  \end{rem}

These formulations make it clear that $\HH_*(\mathcal{C})$ is a module over $\HH^*(\mathcal{C})$. If $\mathcal{C}$ is a smooth $n-CY$ category, combining  \eqref{eq:perfduality} and  \eqref{eq:CYcondition} gives rise to a  (``Van den Bergh duality") isomorphism: 
\begin{align} \HH^*(\mathcal{C}) \cong \HH_{n-*}(\mathcal{C})  \end{align} 
 which is a map of $\HH^*(\mathcal{C})$ modules. 

We now turn to discussing our main example, the derived wrapped Fukaya category. It is profitable to organize the collection of Lagrangian branes into an $A_\infty$ category, the wrapped Fukaya category, $\mathcal{W}(X)$. The objects of this category are Lagrangian branes $L$ and, as the notation indicates, its cohomological category is the Donaldson category $H^*(\mathcal{W}(X))$. The derived wrapped Fukaya category is then given by taking $\WXp$. The details of the construction of $\mathcal{W}(X)$ will not be important for us, as we will only make use of some of its formal properties. The first concerns the relationship between $SH^*(X)$ and Hochschild invariants of $\mathcal{W}(X)$.  Namely, recall that there is a natural comparison map \begin{align} \label{eq:fullCO} \mathcal{CO}: SH^*(X) \to \HH^*(\mathcal{W}(X)) \end{align}
which is a lift of \eqref{eq:COmain} in the sense that 
\begin{align} \mathcal{CO}_{(0)} = \mathcal{B}_L \circ \mathcal{CO} \end{align} 
where $\mathcal{B}_L$ is the morphism from \eqref{eq:bulkboundary}. It is a fundamental result of \cite{Ghiggini, GPS} that $\mathcal{W}(X)$ is generated by Lagrangian co-cores of any Weinstein handlebody presentation. Combining this with the results of \cite{Ganatra, Gao} (see also \cite[\S 11]{Ghiggini}) yields:

\begin{thm} $\WXp$ is a smooth n-CY category with a compact generator. Moreover, the map \eqref{eq:fullCO} is an isomorphism. \end{thm}

\begin{rem} On a technical note,  we should point out that what Ganatra proves is that $\mathcal{W}(X)$ is smooth n-CY as an $A_{\infty}$ category.  So in our statement of the result,  we are using the well-known fact that $A_\infty$ smooth n-CY structures are preserved by Morita equivalences---for a detailed proof in the setting of dg-categories see \cite[Proposition 3.10(e)]{MR2795754} (the $A_\infty$ case does not seem to appear explicitly in the literature but is very similar).  \end{rem}

For later use it will be useful to reformulate Theorem \ref{thm:properness} in more categorical terms.  The following definition is somewhat non-standard (though related ideas appeared in \cite{HL-Preygel}):

\begin{defn} \label{defn:ucp} Let $\mC$ be a dg-category such that $\HH^0(\mC)$ is a finitely generated $\K$-algebra. We say that a category is semi-affine over $\HH^0(\mathcal{C})$ if for any two objects $E_1, E_2$, $H^*(\Hom_{\mathcal{C}}^\bullet (E_1,E_2))$ is a finitely generated module over $\HH^0(\mathcal{C}).$ \end{defn} 

 Theorem \ref{thm:properness} can therefore be rephrased as:  \begin{cor} For any affine log Calabi-Yau variety $X$, the dg-category $\WXp$ is semi-affine. \end{cor}

\begin{rem} Assume that $\mathcal{C}$ is actually a dg-category linear over $R=\HH^0(\mathcal{C})$ (which we take to be finitely generated over $\K$). It is useful to compare this with the more common notion of properness for dg-categories. Recall that a category is called proper if for any two objects $E_1,E_2$, $\Hom_{\mathcal{C}}^\bullet (E_1,E_2)$ is perfect as an $R$-module. Thus, the two definitions of properness coincide if $R$ is a smooth $\K$-algebra. On the other hand, our definition of properness is very natural from the point of view of algebraic geometry: given a smooth quasi-projective variety $Y$ over $\K$, $Y$ is semi-affiine iff $\Perf(Y)$ is semi-affine (Lemma \ref{lem:algebraicgeomprop}).  \end{rem}



\subsection{Automatic generation} \label{subsection:ha} In this section,  we give the proof of Proposition \ref{cor:autogenmain}.  Before explaining our argument,  we must recall the definitions of admissible subcategories and semi-orthogonal decompositions which play a central role in our proof.  In the discussion which follows, we continue with the convention that all categories are linear over a field $\K.$ \vskip 5 pt

\begin{defn} Let $C$ be a triangulated category and let $i: A \to C$ be a full triangulated subcategory. We say that $A$ is right (left)-admissible if the inclusion $i$ has a right (left) adjoint.  \end{defn}

For any full triangulated subcategory $A$, we can define the left and right orthogonal subcategories of $A$, which are denoted by $A^{\perp}$ and $^{\perp}A$ and are also triangulated categories. The property of a category being right-admissible is then equivalent to requiring that for each object $E$ of $C$, there is an exact triangle $$ E_A \to E \to E_{A^{\perp}} $$  where $E_A$ is in $A$ and  $E_{A^{\perp}}$ is in $A^{\perp}$ (the obvious analogue holds for left admissible categories as well). From this alternative characterization, it follows that if the category $A$ is right-admissible, than $A^{\perp}$ is left admissible. If there is a semi-orthogonal decomposition, we denote this by  $C= \langle A^{\perp},A \rangle.$  Finally, we have that for any right (respectively left) admissible subcategory $A$, the Verdier quotient $C/A$ is equivalent to $A^{\perp}$ (respectively $^\perp A$). 

Turning to the dg-versions, we say that a full pretriangulated dg-subcategory $\mathcal{A}$ of a pre-triangulated dg category $\mathcal{C}$ is right (left) admissible if 
$H^0(\mathcal{A})$ is a right-admissible subcategory of $H^0(\mathcal{C}).$ Given such a subcategory, we can also define $\mathcal{A}^{\perp}$ to be the full subcategory of objects whose image in the homotopy category lies in $H^0(\mathcal{A})^{\perp} \subset H^0(\mathcal{C}).$ We have a quasi-equivalence \begin{align} \mathcal{A}^{\perp} \cong \mathcal{C}/\mathcal{A} \end{align} where $\mathcal{C}/\mathcal{A}$ denotes the Drinfeld quotient dg-category (there is an analogous equivalence  $\mathcal{A} \cong \mathcal{C}/\mathcal{A}^{\perp}$). Hochschild homology is an additive invariant \cite[Theorem 11.7]{MR2451292},  in the sense that \begin{align} \label{eq:additiveinv} \HH_*(\mathcal{C})= \HH_*(\mathcal{A}) \oplus \HH_*(\A^{\perp})  \end{align}
as $\HH^*(\mathcal{C})$ modules.  The final ingredient we will need is the following result:

\begin{lem} \label{lem:nCYquotient} \cite[Proposition 3.10(d)]{MR2795754} Let $\mathcal{C}$ be a (pretriangulated) smooth n-CY dg-category and suppose that $\mathcal{A}$ is a Drinfeld quotient of $\mathcal{C}$. Then $\mathcal{A}$ is a (pretriangulated) smooth n-CY dg-category. \end{lem} 

\begin{rem}Lemma \ref{lem:nCYquotient} can be viewed as a non-commutative analogue of the fact that an open subscheme of smooth Calabi-Yau variety is smooth and Calabi-Yau.  \end{rem} 

Having reviewed these definitions and results,  we are now in a position to prove the following lemma which rules out semi-orthogonal decompositions in smooth $n$-CY dg-categories:  

\begin{lem} \label{lem: noadmissibleCY} Let $\mathcal{C}$ be a pretriangulated smooth n-CY dg-category with a compact generator  Let $i:\A \to \mathcal{C}$ be a (non-trivial) right admissible subcategory of $\mathcal{C}$. Suppose further that either: \begin{itemize} \item (i) $\Spec(\HH^0(\mathcal{C}))$ is connected.  \item (ii) The map $\HH^0(\mathcal{C}) \to \HH^0(\mathcal{A})$ is an isomorphism. \end{itemize} Then $i:\A \to \mathcal{C}$ is a quasi-equivalence. \end{lem} 

\begin{proof} \emph{Case (i)}: We have a semi-orthogonal decomposition of $\mathcal{C}$, $\mathcal{C}= \langle \A^{\perp}, \A \rangle.$ As $\A$ and $\A^{\perp}$ are both Drinfeld quotients of $\mC$, they are both smooth n-CY dg categories by Lemma \ref{lem:nCYquotient}.  By \eqref{eq:additiveinv}, we have that  $$\HH_n(\mC) \cong \HH_n(\A) \oplus \HH_n(\A^{\perp})$$ as $\HH^0(\mathcal{C})$ modules. Because $\mC$ is $n$-CY, $\HH_n(\mC)$ is a rank one free module over $\HH^0(\mC)$. The assumption that $\Spec(\HH^0(\mC))$ is connected implies that there is no non-trivial decomposition of a rank one free module into module summands and so $\HH_n(\A^{\perp})=0.$ Because $\A^{\perp}$ is $n$-CY, this implies $\HH^0(\A^{\perp})=0$ and that the category is trivial. \vskip 5 pt

\emph{Case (ii)}: We proceed as in \emph{Case (i)} and note that we have a commutative diagram \[
\begin{tikzcd}
\HH_n(\mC)   \arrow{r}{\pi_1} \arrow{d}{CY} & \HH_n(\A) \arrow{d}{CY} \\
\HH^0(\mathcal{C}) \arrow{r}{\cong} & \HH^0(\mathcal{A})
\end{tikzcd}
\] 

where the map $\pi_1$ denotes projection onto the first factor of \eqref{eq:additiveinv}. It follows that $\pi_1$ is an isomorphism and we have  $\HH_n(\A^{\perp})=0.$ Again, because $\A^{\perp}$ is $n-CY$, this implies that the category is trivial. 
 \end{proof}  
 
 \begin{rem} The above argument is similar to \cite[Theorem 36]{Sheel} which proves a similar result in the case of proper Calabi-Yau categories.  \end{rem} 

In complete generality,  it can be hard to apply Lemma \ref{lem: noadmissibleCY} because the definition of admissibility is a little abstract.  The key observation in our proof is that,  in the setting of semi-affine categories,  there is a readily checkable criterion for a subcategory $\A$ to be a right admissible subcategory.  This relevant definition is the following: 

\begin{defn} \label{defn: friendly} Let $S^{\bullet}$ be a dg-algebra over $\K$ and let $R \subseteq \HH^0(S^{\bullet})$ be a finitely generated $\K$-algebra. We say that $S$ is \emph{friendly} (relative to $R$) if \begin{itemize} \item a dg-module $N^{\bullet}$ is perfect over $S^{\bullet}$ iff $H^*(N^{\bullet})$ is finitely generated over $R$.  \end{itemize}  \end{defn} 

There are two main sources of friendly algebras: 
\begin{itemize} 
\item Classical algebras $S$ which \begin{itemize} \item have finite homological dimension \item are module finite over a finitely generated center $R$ (\cite[Proposition 7.25]{MR2434186}). \end{itemize} 
\item Derived categories of coherent sheaves provide a natural source of friendly algebras. More precisely, if $R$ is a finitely generated $\K$-algebra, let $Y \to \Spec(R)$ be a projective $R$-scheme which is smooth over $\K$, and let $E$ be a compact generator of $\Perf(Y)$. Then $S^{\bullet}:= \Hom^{\bullet}(E,E)$ is friendly relative to $R$ (see Proposition \ref{prop:agprop}).
\end{itemize} 

\begin{lem} \label{lem:admissibility} Let $\mathcal{C}$ be a pretriangulated, idempotent closed, semi-affine dg-category.  Let $E$ be an object of $\mC$ whose endomorphism dg-algebra $S^{\bullet}=\Hom^{\bullet}_{\mC}(E,E)$ which is friendly over $\HH^0(\mathcal{C})$. Then $\langle E \rangle$ is a right admissible subcategory of $\mC$.  \end{lem} 
\begin{proof} This is basically just pushing the standard argument for admissibility to its maximal extent and so we only give a sketch. The object $E$ allows us to define a functor \begin{align} \Hom_\mC^{\bullet}(E,-): \mC \to \operatorname{Mod}(S^{\bullet}) \end{align} which lands in the subcategory of modules which are cohomologically finite over $R$ by properness of $\mC$. Then by our assumption, image of this functor in fact lies in $\Perf(S^{\bullet})$. So our functor induces a functor \begin{align} i^*: \mC \to \Perf(S^{\bullet}). \end{align}  which is right-adjoint to the inclusion \begin{align} i: \Perf(S^{\bullet}) \to \mC. \end{align} Thus $<E>$ is a right-admissible subcategory.  \end{proof}

Combining Lemmas \ref{lem:admissibility}, \ref{lem: noadmissibleCY} together with Proposition \ref{prop:agprop} yields the following result: 

\begin{cor} \label{cor: automaticgen} Let $\mC$ be a smooth dg-category (stable, with compact generator) over $\K$ which is both n-CY and semi-affine. Suppose further that \begin{itemize} \item $E$ an object such that $$\Perf(\Hom^\bullet (E,E)) \cong \Perf(Y)$$ where $Y$ is a projective $\HH^0(\mC)$-scheme which is smooth over $\K$. \item Either \begin{enumerate} \item $\Spec(\HH^0(\mC))$ is connected or \item the induced map $\HH^0(\mC) \to \HH^0(\Perf(Y))=H^0(\mathcal{O}_Y)$ is an isomorphism. \end{enumerate} \end{itemize} Then $<E>=H^0(\mC).$ \end{cor}
\begin{proof} By Lemma \ref{lem:admissibility} and Proposition \ref{prop:agprop}, $<E>$ is an admissible subcategory of $H^0(\mC).$ The fact that $\mC$ is n-CY together with either of the hypotheses in the second bullet allow us to invoke Lemma \ref{lem: noadmissibleCY} to conclude that $<E>$ is the full category. \end{proof}

 \begin{prop} \label{cor:autogenmain} Let $(M,\D)$ be a Calabi-Yau pair and let $E$ be an object of $\WXp$ such that $$\Perf(\Hom^\bullet (E,E)) \cong \Perf(Y)$$ where $Y$ is a smooth quasi-projective scheme over $\K$. Then $<E>=H^0(\WXp).$  \end{prop} 
\begin{proof} Note that $Y$ inherits the structure of a $\HH^0(\WXp) \cong SH^0(X)$ scheme along the morphism $\HH^0(\WXp) \to \HH^0(\Perf(Y))= H^0(\mathcal{O}_Y).$ The argument of Lemma \ref{lem:algebraicgeomprop} allows one to deduce that $Y$ is projective over $SH^0(X)$. Moreover,  $SH^0(X)$ is connected because it degenerates to a Stanley Reisner ring, which is connected. We can therefore apply Corollary \ref{cor: automaticgen} with $\mathcal{C}=\WXp$. \end{proof} 

Proposition \ref{cor:autogenmain} yields the following Corollary:

\begin{cor} \label{cor:strongresmain} Suppose $SH^0(X)$ is smooth and $X$ admits a homological section, then $$\Perf(SH^0(X)) \cong \mathcal{W}(X).$$ \end{cor}

\subsubsection{Local models of Lagrangian fibrations} \label{subsection:Gross-Siebert}

We next give an extended example to illustrate Corollary \ref{cor:strongresmain}.  Consider specifically the affine variety defined by the equation \begin{align} 
    \label{eq: conicbundle3} X=\lbrace (z,u,v) \in  (\mathbb{C}^*)^{n-1} \times \mathbb{C}^2 |\; uv= c+ z_1+ \cdots + z_{n-1} \rbrace 
\end{align}  for some $c>0.$

We will use Corollary \ref{cor:strongresmain} to prove homological mirror symmetry for this affine log Calabi-Yau. To prepare for this,  we give a general geometric criterion (Proposition \ref{prop:weaksection}),  for a contractible conical Lagrangian $L_0$ to define a homological section.  For notational simplicity(i.e.  to avoid cluttering our notation with connected components of strata), we state our criterion only for those Calabi-Yau pairs $(M,\D)$ where all of the strata $D_I$ are connected.  The basic idea of this criterion is that:\begin{itemize} \item All Hamiltonian chords start and end at the same point.  (We will say that ``all of the Hamiltonian chords are Hamiltonian orbits.")  \item the Lagrangian $L_0$ should meet each of the (psuedo-) Morse Bott manifolds of Hamiltonian orbits cleanly in contractible submanifolds.  \end{itemize} 

 To spell this out more precisely,  fix an $m$ (meaning a $\Sigma_m$, $h^m$ etc.) and let $h^{m}_{\#}$ denote the rescaled Hamiltonians from \eqref{eq:rescaledhm}.  We will keep the same notation for manifolds of orbits $\mathcal{F}_\v$, with winding vector $\v$,  as well as the analogous submanifolds $D_\v^{c_{0}}$,  $S_\v^{c_{0}}$ and isolating sets $U_\v$ from the ``Isolating neighborhoods" mini-section in \S  \ref{subsection:relativeFloer}.  (We will freely use the constructions and notations from this and the ``Hamiltonian perturbations" mini-section in \S \ref{subsection:relativeFloer},  so the reader may wish to quickly review these.)
 
Suppose that all of the Hamiltonian chords $\mathcal{X}(L_0; h^m_{\#})$ are Hamiltonian orbits.  Then the set of chords of $L_0$ inside $U_\v$,  $\mathcal{F}_\v^{L_{0}} \subset \mathcal{F}_\v$, is identified with a subset of the Hamiltonian orbits inside $U_\v$ and all of the chords lie inside one of these $U_\v.$  In what follows, we let $L_{0,\v}:= L_0 \cap U_\v.$

Perturb $h^{m}_\#$ to functions $H^m_F$, $H^m_L$ so that all Hamiltonian orbits and chords of $L_0$ in $U_\v$ are non-degenerate. Consider the Floer complex \begin{align} \label{eq:localFloer} CF_{\operatorname{loc}}^*(U_\v \subset M; h^m_{\#}) := \bigoplus_{x_0 \in \mathcal{X}(U_\v;H^m_F)} |\mathfrak{o}_{x_{0}}| \end{align} Assume that $||H^m_F-h^m_{\#}||_{C^2}$ is sufficiently small. Then letting $x_0$ and $y_0$ be two Hamiltonian orbits in $U_\v$, a Gromov compactness argument shows that any Floer trajectory connecting $x_0, y_0$ (in $X$) lies entirely in $U_\v.$ Assuming the perturbation is small enough, it follows from the usual arguments that the usual Floer differential defines a differential on  \eqref{eq:localFloer}. We define $HF_{\operatorname{loc}}^*(U_\v \subset M; h^m_{\#})$ to be the cohomology of the resulting complex (this is called the ``local Floer cohomology".) 

One can analogously construct the local Lagrangian Floer cohomology groups $$HF^*_{\operatorname{loc}}(U_\v \subset M,L_{0,\v}; h^m_{\#})$$ by considering the complex generated by chords of $L_0$ in $U_\v$, $$ CF_{\operatorname{loc}}^*(U_\v \subset M, L_{0,\v}; h^m_{\#}) := \bigoplus_{x_0 \in \mathcal{X}(U_\v, L_0; H^m_L)} |\mathfrak{o}_{L_0,L_0, x_{0}}| $$  and (again under the assumption that $||H_L^m-h^m_{\#}||_{C^{2}}$ is sufficiently small) equipping it with the usual Lagrangian Floer theory differential.  After possibly making the perturbation even smaller, we have a map  \begin{align}  \mathcal{CO}_\v: HF^*_{\operatorname{loc}}(U_\v \subset M ; h^m_{\#}) \to HF^*_{\operatorname{loc}}(U_\v \subset M,L_{0,\v}; h^m_{\#}) \end{align} 

Before turning to our criteria,  we make a preliminary definition. 

\begin{defn}
Let $L$ be a manifold equipped with a Riemannian metric. We say that a function $h: T^*L \to \mathbb{R}$, is split if it is of the form 
$$h:=\pi^*(h_L)+ |v|^2 $$   
\end{defn}



\begin{defn} \label{defn:geomsect}
Fix an $m$ and  consider the following assumptions on a contractible Lagrangian $L_0$: \vskip 5 pt
\emph{Assumptions:}  \begin{enumerate}[label=(\alph*)] \item \label{item:geo1} All of the Hamiltonian chords $\mathcal{X}(L_0;h^m_{\#})$ are Hamiltonian orbits. 

 \item \label{item:geo2} For any $\v$, $L_{0,\v}$ is contractible and $\pi_I(L_{0,\v})=\bar{L}_{0,\v}$ is a contractible Lagrangian submanifold with corners which is embedded transversely in $D^{c_0}_{\v}.$  Moreover, for any point $x \in \bar{L}_{0,\v}$, $$L_{0,\v} \cap \pi_I^{-1}(x) $$ is a Lagrangian section of the fibration
 $$ (\rho_1,\cdots,\rho_n): \pi_I^{-1}(x) \cap U_\v \to \prod_{i \in I} [\rho_{i,\v}-c_0, \rho_{i,\v}+c_0] \subset \mathbb{R}^I  $$ 
\item \label{item:geo3} There is an outward pointing Morse function $\hD: D^{c_0}_{\v} \to \mathbb{R}$ which: \begin{itemize} \item restricts to a function which is split on some Weinstein tubular neighborhood $U_{\bar{L}_{0,\v}}$,  of $\bar{L}_{0,\v}.$  \item  has a unique critical point $c$ with $\deg(c)=0$. (It follows that $c \in \bar{L}_0.$) \end{itemize} 
\end{enumerate} 
\end{defn} 

In each neighborhood $U_\v$, the conditions \ref{item:geo1}-\ref{item:geo3} provide strong control over the local Lagrangian Floer cohomology as well as the map $\mathcal{CO}_\v$.   

\begin{lem} \label{lem:topsection} Suppose that $L_0$ is contractible and satisfies conditions \ref{item:geo1}-\ref{item:geo3}. Then for every $\v$, $HF_{\operatorname{loc}}^*(U_\v \subset M, L_0; h^m_{\#})=0$ if $\ast \neq 0$ and the local closed-open map \begin{align} \label{eq:localCO} \mathcal{CO}_\v: \mathbf{k} \cong HF^0_{\operatorname{loc}}(U_\v \subset M ; h^m_{\#}) \to HF^0_{\operatorname{loc}}(U_\v \subset M,L_0; h^m_{\#}) \end{align}
is an isomorphism.  \end{lem} 
\begin{proof} Condition \ref{item:geo1} ensures that there are no Hamiltonian chords outside of $\cup_\v U_\v$ and we can thererfore prove the lemma for each $\v$ separately. So, for the rest of the argument, we fix some $\v$. We will construct a perturbation $H^m_{\#}$ of $h^m_{\#}$  which behave nicely inside of $U_\v$.  Pull back $\hD: D^{c_0}_{I}$ to a Morse-Bott function $\pi_I^*(\hD)$ on the torus bundle over $D^{c_0}_{I}$ where $\rho_i=\rho_{i,\v}.$ The critical locus of $\pi_I^*(\hD)$ is a Morse-Bott function with tori as critical loci.

 Over the critical point $c$, condition \ref{item:geo3} ensures that there is a unique point $\hat{c}$ where $L_0 \cap \pi_I^{-1}(c)$ intersects the torus bundle. We can perturb $\pi_I^*(\hD)$ to a Morse function $\widehat{h}_\v^{base}$ in such a way that it has a unique degree zero critical point at $\hat{c}.$ We then spin $\widehat{h}_\v^{base}$ to a time-dependent function $h_\v$ and plug this into \eqref{eq:pertuv}. To extend this to all of $M$, choose the other $h_{\v'}$ for $\v' \neq \v$ and $h_\D$ arbitrarily and construct the perturbation $H^m_{\#}$ via \eqref{eq: Hpert2}. By construction, our Hamiltonians have a unique degree zero orbit $x_c$ (corresponding to the critical point $\hat{c}$) in $U_\v$ which is simultaneously a chord of $L_0$ (we denote this by $x_c^{L})$. The chord $x_c^{L}$ has degree zero again using Condition \ref{item:geo2}.

Take $H^m_F=H^m_L=H^m_{\#}$.  Next, we consider a specific perturbation one-form over the ``chimney domain" (the disc with one boundary puncture and one interior puncture), which will ensure that the topological and geometric energy of solutions which define \eqref{eq:COmainlambda} coincide. The chimney domain can be realized by taking \begin{align} \mathbb{R} \times [0,1]/ \sim \end{align} where $\sim$ denotes the equivalence relation $(s,0) \sim (s,1)$ for $s \geq 0$. The coordinate $s+it$ is well-defined and holomorphic everywhere except at the point $(0,0)$ where the conformal structure is locally modelled on the Riemann surface associated to $\sqrt{z}.$ In this model, we let $\beta$ the closed one-form $dt$ (extended to vanish at $(0,0)$). We take our perturbation one form to be $ K=H_{\#}^m(t,x)dt.$ For solutions to \eqref{eq:generalizedFloery} (with this choice of $K$), we have that topological and geometric energy agree (\cite[\S 3]{AbSch2}): \begin{align} E_{top}(u)= E_{geo}(u) \end{align}

 With these choices,  any solution $u$ with asymptotics $x_c$, $x_c^{L}$ have zero geometric energy: $$E_{geo}(u)=0.$$ They therefore arise by ``spinning" constant solutions (again with respect to some $\hat{J}_t$ given by conjugating $J_t$), namely $u= g_t \circ \tilde{u}$ where $\tilde{u}$ is constant. It follows that solutions with these asymptotics are unique. They are also regular because the constant solutions are regular and $\tilde{u}$ is regular iff its spinning $g_t \circ \tilde{u}$ is regular (see e.g. the proof of \cite[Theorem 3.15]{MR2563724}) concluding the proof.   \end{proof} 

We are finally in a position to give our geometric criterion:

 \begin{prop} \label{prop:weaksection} Let $L_0$ be a contractible Lagrangian brane and suppose there exists a sequence of contact boundaries $\Sigma_m$ (and functions $h^m_{\#}$) as in \S \ref{sect: actionspec} such that conditions \ref{item:geo1}-\ref{item:geo3} from Definition \ref{defn:geomsect} hold.  Then: \begin{itemize} \item $WF^*(L_0)=0$ if $\ast \neq 0$ and \item the map \eqref{eq:COmain} is an isomorphism in degree zero.  \end{itemize}  In other words, $L_0$ is a homological section.  \end{prop}
\begin{proof} For a given $m$, Lemma \ref{lem:topsection} implies immediately that $HF^0(L_0;h^m_{\#})=0$ if  $\ast \neq 0.$ By taking the limit, this implies that  $WF^*(L_0)=0$ if $\ast \neq 0.$ Next we explain that, \begin{align} \label{eq:COm} \mathcal{CO}_{(0)}: HF^0(X\subset M;H^m_F) \to HF^0(L_0;H^m_L) \end{align} is an isomorphism.
 Before perturbing all of the actions of the Hamiltonian orbits $x_0$ in a given $U_\v$ have the same action $A_{h^m_{\#}}(x_0)$, meaning that $A_{h^m_{\#}}$ can be thought of a function on the integral vectors $\v$ which are winding numbers of orbits of $h^m_{\#}.$ Assuming $||H^m_F-h^m_{\#}||_{C^{2}}$ and $||H^m_L-h^m_{\#}||_{C^{2}}$ are both sufficiently small, the Floer complexes $CF^*(X \subset M; H^m_F)$, $CF^*(L_0;H^m_L)$ are both filtered by $A_{h^m_{\#}}(\v)$, meaning that an orientation line associated to an orbit $x_0 \in U_\v$ has weight $A_{h^m_{\#}}(\v)$. We let $CF^*_{\geq a}(X \subset M; H^m_F)$, $CF^*_{\geq a}(L_0;H^m_L)$ denote the subcomplexes generated by (orientation lines associated to) orbits with $A_{h^m_{\#}}(\v)$ bigger than $a$.

 Another Gromov compactness argument shows that (again assuming the perturbations are sufficiently small), for any $a$, there exists $\delta$ sufficiently small such that 

\begin{align*} CF^*_{\geq a- \delta}(X \subset M; H^m_F)/CF^*_{\geq a+\delta}(X \subset M; H^m_F) \cong \bigoplus_\v CF_{\operatorname{loc}}^*(U_\v \subset M;h^m_{\#}) \\ CF^*_{\geq a-\delta}(L_0; H^m_L)/CF^*_{\geq a+\delta}(L_0; H^m_L) \cong \bigoplus_\v CF_{\operatorname{loc}}^*(U_\v \subset M,L_0;h^m_{\#}).\end{align*} where in both of these equations $\v$ ranges over all vectors with $A_{h^m_{\#}}(\v) \in (a-\delta,a+\delta].$ After passing to direct limits, it follows from the second equation that  $WF^*(L_0)=0$ if $\ast \neq 0.$ Moreover, the induced map \begin{align*} \mathcal{CO}_{(0)}^{a}: CF^*_{\geq a- \delta}(X \subset M; H^m_F)/CF^*_{\geq a+\delta}(X \subset M; H^m_F) \to CF^*_{\geq a- \delta}(L_0; H^m_L)/CF^*_{\geq a+\delta}(L_0; H^m_L)   \end{align*} agrees with 
$\bigoplus_\v \mathcal{CO}_\v$, which induces an isomorphism on degree zero cohomology in view of \eqref{eq:localCO}. It follows that  \eqref{eq:COm} is an isomorphism. Passing to the limit implies that \eqref{eq:COmain} is an isomorphism in degree zero as required. \end{proof} 

\begin{rem} Proposition \ref{prop:weaksection} should be viewed as a preliminary statement.  It is very likely that the hypotheses could be streamlined/simplified with a little bit more effort.  However, the present version suffices for our current application.  \end{rem} 

We now return to the concrete examples defined by \eqref{eq: conicbundle3}.  As mentioned,  these are affine log Calabi-Yau and we describe a simple compactification to a Calabi-Yau pair.  Let $\bar{D}$ denote the toric divisors in $\mathbb{C}P^{n-1}$ (so that $\mathbb{C}P^{n-1} \setminus \bar{D}=(\mathbb{C}^*)^{n-1}$). Set $Z \subset \mathbb{C}P^{n-1}$ to be a generic hyperplane. Let
$\mathbb{C}P^1$ be equipped with its standard toric divisors $\lbrace 0 \rbrace, \lbrace \infty \rbrace$.  Consider the blow up of $M^b:=\mathbb{C}P^{n-1} \times
\mathbb{C}P^1$ at $Z \times \lbrace 0 \rbrace$ which we denote by \begin{align}\label{eq:blows} \pi: M \to M^{b}.\end{align} On
$M^{b}$, consider divisors $D^b_i=\bar{D}_i \times
\mathbb{C}P^1$ for $i= \lbrace 1, \cdots, n \rbrace$ and $D^b_{n+1}:= \mathbb{C}P^{n-1} \times \lbrace 0 \rbrace$, $D^b_{n+2}:= \mathbb{C}P^{n-1} \times \lbrace \infty \rbrace$.  Let $D_i$ denote
the proper transform of all of the divisors $D^b_i$. Then the union of these divisors $\D=\cup D_i$
is a normal crossings anti-canonical divisor on $M$.  In the case where $Z$ is the compactification of $Z^o \hookrightarrow  (\mathbb{C}^*)^{n-1}$, the hypersurface cut out by $c+ z_1+ \cdots + z_{n-1}=0$, we have that the affine variety $X$ is the complement of $\D$ in $M$ i.e. $X=M\setminus \D.$ However, it will be convenient to vary $Z$ in the proof of Lemma \ref{lem: nonzeroproducty}; the result will give rise to deformation equivalent convex symplectic manifolds. 
Choose rational numbers $0< \kappa_{n+1}<n$,  $0<\kappa_{n+2}$ and equip the blowup with the $\mathbb{Q}$-divisor given by \begin{align} \label{eq:conicpolar} \kappa_{n+1} D_{n+1} + \kappa_{n+2} D_{n+2} + \sum_{i=1}^{n}D_i.  \end{align}
(We allow $\kappa_{n+1}, \kappa_{n+2}$ to be rational for the argument in Lemma \ref{lem:homologicalsec}.) It follows from considerations in toric geometry that this gives a K{\"a}hler class for any  $\kappa_{n+1},  \kappa_{n+2}$ as above. 

\begin{lem} \label{lem:homologicalsec} $\WXp$ admits a homological section $L_0$. \end{lem} 
\begin{proof} Let $L^b$ be the ``positive real locus" in  $(\mathbb{C}^*)^{n-1}$, $L^b := (\mathbb{R}^{\geq 0})^{n-1} \times \mathbb{R}^{\geq 0} \subset (\mathbb{C}^*)^{n-1}.$ It is straightforward to check that the closure of $L^b$ is disjoint from the blow-up locus $Z \times 0$ (because $c>0$). Hence, as the blow-up map  \eqref{eq:blows} restricts to a diffeomorphism away from the exceptional divisor $E$ \begin{align} \label{eq:opendiffeo} X \setminus (X \cap E) \to (\mathbb{C}^*)^{n-1}. \end{align} We can thus lift $L^b$ to a Lagrangian $L_0 \in X$, $$L_0: = \pi^{-1}(L^b).$$  We will construct a particular regularized (as in Theorem \ref{thm: MTZ}) symplectic form on $(M,\D)$ together with a nice primitive $\theta$ on $X$ (as in Theorem \ref{thm:niceprimitive}) so that $L_0$ is an exact, conical Lagrangian.  From the construction, it will be immediate that $L_0$ is a geometric section in the sense of Lemma \ref{lem:topsection} and thus gives rise to a homological section.

To do this, we will take  $\kappa_{n+1}$ small and $\kappa_{n+2}=1-\kappa_{n+1}.$ Next, note that $M$ is a hypersurface in the $\mathbb{C}P^1$ bundle, $\mathbb{P}(\mathcal{O}(Z) \oplus \mathcal{O})$ over $M^b:=\mathbb{C}P^{n-1} \times \mathbb{C}P^1$ (in an abuse of notation we let $\mathcal{O}(Z)$ be the line bundle associated to the hypersurface $Z \times \mathbb{C}P^1 \subset M^{b}$). Let $\omega_{\mathbb{C}P^{n-1}}= n\omega_{FS}$ be the standard Fubini-Study K{\"a}hler form on $\mathbb{C}P^{n-1}$ rescaled by $n$. Lift this to a K{\"a}hler form $$\omega_b:=\omega_{\mathbb{C}P^{n-1}} \times \omega_{\mathbb{C}P^1}$$ on $M^{b}$ by taking product with the standard symplectic structure on $\mathbb{C}P^1$. Then a choice of Hermitian metric on the line bundle $\mathcal{O}(Z)$ induces a K{\"a}hler form on $\mathbb{P}(\mathcal{O}(Z) \oplus \mathcal{O})$ so that the $\mathbb{C}P^1$ fibers have area $\kappa_1$. We can then restrict this form to the hypersurface $M$ to obtain a K{\"a}hler form which we denote by $\omega.$ Away from the exceptional divisor $E$ we have that \begin{align} \omega= \pi^* \omega_b +i \kappa_{n+1} \partial \bar{\partial} \psi \end{align} for some potential $\psi.$ 

Following Section 3.2. of \cite{AAK}, we will modify this K{\"a}hler form to a K{\"a}hler form $\omega$ so that  \begin{align} \label{eq:omegaopen } \omega=\pi^* \omega_b, \text{ on }X \setminus \pi^{-1}(U) \end{align} where $U$ is a tubular neighborhood of $Z \times 0 \subset M^b$ which does not intersect the closure of $L^b.$ Namely, we choose a cut-off function $\chi: M^b \to [0,1]$ which is supported in $U$ and which is $S^1$ invariant with respect to the rotation action on $\mathbb{C}P^1$. We require that $\chi=1$ in a smaller open set about $Z \times 0$.  On the complement of $E$, we set \begin{align} \omega = \pi^* \omega_b +i \kappa_{n+1} \partial \bar{\partial} (\chi \psi) \end{align}
This extends to a form on all of $M$ which again K{\"a}hler (for sufficiently small $\kappa_{n+1}$). As the toric divisors in $\mathbb{C}P^{n-1} \times \mathbb{C}P^1$ are already regularized, we can deform this to a symplectic structure which admits a regularization (as in Theorem \ref{thm: MTZ}) and so that \eqref{eq:omegaopen } still holds.  It follows that $L_0$ is Lagrangian for such an $\omega$ (and it is exact for any primitive because it is contractible). We can similar assume that our nice primitive $\theta$ is invariant under the infinitesimal torus action (coming from the identification \eqref{eq:opendiffeo}) outside of this preimage.  For such a primitive, $L_0$ is conical at infinity.  

\end{proof}

\begin{rem} The basic idea of Lemma \ref{lem:homologicalsec} is to lift the cotangent fiber from $(\mathbb{C}^*)^n$ to the total space of its birational modification. In \cite{sequel}, we will extend these ideas to show that $\mathcal{W}(X)$ admits a homological section whenever $(M,\D)$ admits a toric model. \end{rem} 

Take $\kappa_{n+2}>>n$ and $n> \kappa_{n+1} \geq n/2$. Set $w_{1}=\PSSlog(\theta_{D_{n+1}}) \in SH^0(X,\mathbb{Z}), w_{2}=\PSSlog(\theta_{D_{n+2}}) \in SH^0(X,\mathbb{Z})$ and $u_i=\PSSlog(\theta_{D_{i}}) \in SH^0(X,\mathbb{Z})$ for $i \in \lbrace 1,\cdots, n \rbrace.$ 

\begin{lem} \label{lem: nonzeroproducty} In $SH^0(X,\mathbb{Z})$, we have \begin{align} \label{eq: prod1}  w_1w_2= 1 \end{align} \begin{align} \label{eq:prod2} \prod_i u_i \neq 0 \end{align} \end{lem}  
\begin{proof} Equation \eqref{eq: prod1} follows from \cite[Lemma 6.11]{GP1} so we turn to proving \eqref{eq:prod2}.\footnote{Lemma 6.11 of \cite{GP1} works with a different K{\"a}hler class (corresponding to a small blowup). However, the argument goes through in this case as well (and in fact is substantially simpler due to the homology class decomposition \eqref{eq:shdecomp} which in general does not provide additional information).} To do this, first recall that there is a decomposition \begin{align} \label{eq:shdecomp} SH^0(X):= \bigoplus_{\mathbf{n} \in H_1(X,\mathbb{Z})} SH^0(X)_{\mathbf{n}} \end{align} according to the homology class of the orbits.  We will proceed by contradiction, so suppose accordingly that \begin{align} \label{eq:contra}  \prod_i u_i = 0 \end{align} 

\emph{Step 1:} In this step, we show that  \eqref{eq:contra} implies \eqref{eq:prodzero} below. Consider the particular vectors $\mathbf{n}_i$:  \begin{align*}  \mathbf{n}_i=\underbracket{(0,\cdots, 1, \cdots 0)}_{\text{1 in i-th position and 0 elsewhere}} \in \mathbb{Z}^{n-1}, \quad i \in \lbrace 1,\cdots,n \rbrace \\  \mathbf{n}_{i}=(-1,\cdots, -1, \cdots, -1) \in \mathbb{Z}^{n-1}, \quad i=n \end{align*} If \eqref{eq:contra} holds, then we will show that for any collection of $\alpha_i \in SH^0(X)_{\mathbf{n}_{i}}$, $i \in \lbrace 1,\cdots, n \rbrace$ \begin{align} \label{eq:prodzero} \prod_i \alpha_i=0. \end{align}

To see \eqref{eq:prodzero}, notice first that for all $\mathbf{n}_i$ as above that  $SH^0(X)_{\mathbf{n}_{i}}$ is a free rank-one module over $\K[w_1,w_2]/(w_1w_2=1)$: \begin{align} \label{eq:freerank1} SH^0(X)_{\mathbf{n}_{i}} \cong \bigoplus_{m \in \mathbb{Z}} \K \cdot u_i w_1^m \end{align} This follows follows from the fact $SH^0(X)$ is a deformation of the Stanley-Reisner ring and, in the Stanley Reisner ring, these subspaces are free rank one-modules over the ring generated by $w_1,w_2$. Equation \eqref{eq:prodzero} now follows since any such $\alpha_i$ be rewritten in the form $f_i(w_1)u_i$ and thus \begin{align*} \prod \alpha_i= \prod_i (f_i(w_1) u_i)   \\ (\prod_i f_i(w_1)) (\prod_iu_i)= 0. \end{align*}      \vskip 10 pt

\emph{Step 2:} We will argue that  in turn that \eqref{eq:prodzero} contradicts the main result of \cite[Theorem 1.1]{Tonkonog}(or more precisely it's refinement by homology classes as described in Lemma 6.5 of \emph{loc. cit.}).  As \eqref{eq:prodzero} is a statement about $SH^0(X)$ which does not  reference specific elements or filtrations, we are free to deform $\kappa_{n+1}$, $\kappa_{n+2}$ in the Kahler cone and take $\kappa_{n+1}=\kappa_{n+2}=1$ so that $M$ is monotone. We will also take our hyperplane $Z$ to a small rotation $\rho$ of one of the toric components $\bar{D}_1$ of $\bar{D} \subset \mathbb{C}P^{n-1}.$ The rotation $\rho$ gives rise to an $S^1$-equivariant symplectomorphism $\phi_{\rho}:M^{tor} \cong M$ where $M^{tor}$ is the toric Fano blow-up of $M^b$ along $\bar{D}_1$. Let $L^{tor}$ denote the standard monotone torus in $M^{tor}$ (corresponding to the barycenter of the moment polytope) and let $L:=\phi_{\rho}(L^{tor})$. Then $L$ is a monotone torus in $M$ with two additional important properties: \begin{enumerate} \item the map $H_2(L) \to H_2(M)$ is trivial. \item $L$ lies in $X$ and $j: L \to X$ exact for a suitable $S^1$-invariant primitive of $\omega_{|X}.$ \end{enumerate} Because $L$ is monotone, it has a superpotential $$W_L \in \mathbb{Z}[H_1(L)].$$ Let $\epsilon_{\operatorname{aug}}: \mathbb{Z}[H_1(L)] \to \mathbb{Z}$ denote the augmentation homomorphism. Then \cite[Theorem 1.1]{Tonkonog} states that \begin{align} \label{eq:tonk} n!<\psi^{n-2}pt>_{M,n}= \epsilon_{\operatorname{aug}}(W_L^{n}) \end{align} where $<\psi^{n-2}pt>_{M,n}$ denotes the genus zero gravitational descendant invariant over all curve classes with $\beta \cdot -K_M = \beta \cdot \D=n.$ Because $H_2(L) \to H_2(M)=0$ (property (1) above), \cite[Lemma 6.5]{Tonkonog} refines \label{eq:tonk} to determine gravitational descendants in individual curve classes $\beta$ with $\beta \cdot \D=n.$ The $\beta$ we are interested in are classes of curves $\pi^{-1}(pt \times \bar{\beta})$ where $pt \in \mathbb{C}P^1$ is any point disjoint from the toric divisors $\lbrace 0 \rbrace, \lbrace \infty \rbrace$ and $\bar{\beta}$ is a line in $\mathbb{C}P^{n-1}$. It is easy to see that this is the unique effective curve class such that $$\beta \cdot D_i=: \begin{cases} 0 &\mbox{if } i=n+1,n+2 \\
1 & \mbox{if } i \neq n+1,n+2 \end{cases} $$   Note that the super-potential is a sum of terms $$W_L=\sum_i W_i $$ where $W_i$ counts curves that intersect $D_i$ exactly once. Equation \eqref{eq:tonk} then refines to \begin{align} \label{eq:tonk2} n!<\psi^{n-2}pt>_{M,\beta}= \epsilon_{\operatorname{aug}}(\prod_{i=1}^{n} W_i). \end{align} It follows in particular that \begin{align} \label{eq:tonk2}  \prod_i W_i \neq 0 \end{align} The result then follows by a standard domain degeneration argument which shows that $W_i$ is in the image of the Viterbo restriction $j^{!}: SH^0(X)_{\mathbf{n}_{i}}  \to D^*(L).$ $$ W_i \in \operatorname{Image}(j^{!}(SH^0(X)_{\mathbf{n}_{i}})).$$ (See e.g. \cite[Theorem 1.1]{MR3974686} in the case of a smooth divisor; the normal crossings case is no different once the $\PSS_{log}$ classes have been constructed.) Because Viterbo restriction is a ring-homomorphism, it follows that   \eqref{eq:tonk2} contradicts \eqref{eq:prodzero}. \end{proof} 

We record a few preparatory lemmas which will be needed in Proposition \ref{prop:GrossSiebert}. The following lemma shows that when $X$ is affine log Calabi-Yau and $\operatorname{char}(\K)=0$, the cohomology of $X$ (two periodic with $\K$ coefficients) can be extracted from the wrapped Fukaya category: 

\begin{lem} \label{lem:Zhao} Assume that $\K$ is a field of characteristic zero and $X$ is any affine log Calabi-Yau variety: $$ \operatorname{HP}_*(\Perf(\mathcal{W}(X))) \cong H^{*+n}(X,\K)((\beta))$$   \end{lem} 
\begin{proof} 
 The paper \cite{Zhao} constructs a localized version of periodic $S^1$-equivariant symplectic cohomology  $\operatorname{HP}_{S^1,loc}^*(X)$ such that $$ \operatorname{HP}_{S^1,loc}^*(X) \cong H^*(X)((\beta))$$ where $\beta$ is a formal variable of cohomological degree $2$(the ``Bott element"). Roughly speaking, Zhao's theory is given by applying the Tate construction to Floer cohomologies of Hamiltonians at a finite slope and then taking the direct limit over higher slopes, whereas the usual periodic symplectic cohomology is given by first taking the direct limit over the slopes and then applying Tate construction. However in the affine Calabi-Yau case, because the symplectic cochain complexes live in bounded degree, these subtleties disappear and the localized, periodic $S^1$-equivariant symplectic cohomology is the same as the usual one: $$ \operatorname{HP}_{S^1,loc}^*(X) \cong \operatorname{HP}_{S^1}^*(X) $$  Finally, using \cite{Sheelcyclic}, we have that the cyclic extension of the open-closed maps give rise to an isomorphism $$ \mathcal{OC}_{S^{1}}: \operatorname{HP}_*(\Perf(\mathcal{W}(X))) \cong  \operatorname{HP}^{*+n}_{S^1}(X) $$ \end{proof} 

\begin{lem} \label{lem:eulerchar} If $X$ is the variety described by \eqref{eq: conicbundle3}, we have $$\chi(X)=: \begin{cases} 1 &\mbox{if } n=2 \\
-1 & \mbox{if } n > 2 \end{cases} $$  \end{lem} 
\begin{proof} In view of  \eqref{eq:opendiffeo}, we have that $\chi(X)=\chi(E \cap X)=\chi(Z^o)$, where in the last equality we have used the fact that the variety $E \cap X$ retracts onto $Z^o$. In dimension two, $Z^o$ is just a point yielding the first case. When $n>2$, $Z^o$ is a generalized pair of pants which is well-known to be homotopy equivalent to a torus with a point removed, yielding the second equality.   \end{proof} 

\begin{prop} \label{prop:GrossSiebert} Let $\K$ denote a field of characteristic zero and let $\mathcal{A}$ be the ring \begin{align} \label{eq:GSring}  \mathcal{A}:=(\K[u_1,\cdots, u_{n},w_1,w_2]/ (\prod_j u_j = 1+ w_1, w_1w_2=1) \end{align} We have an equivalence of categories  $$ \Perf(\mathcal{W}(X)) \cong \Perf(\Spec(\mathcal{A}))$$ (Here $\mathcal{W}(X)$ denotes the wrapped Fukaya category with $\K$ coefficients.) \end{prop}
\begin{proof}
  From filtration considerations and \eqref{eq: prod1}, we have that $$ SH^0(X) \cong \K[u_1,\cdots, u_{n},w_1,w_2]/ \mathcal{I} $$  where $\mathcal{I}$ is the ideal determined by the equations: 
\begin{align} \prod u_j= N^{\mathbf{a}}+ N^{\mathbf{b}}w_1  \\ w_1w_2=1 \end{align} 

For $N^{\mathbf{a}},N^{\mathbf{b}} \in \mathbb{Z}$ with at least one of these numbers being non-zero. We claim that in fact both $N^{\mathbf{a}}$ or $N^{\mathbf{b}}$ are non-zero. Without loss of generality, assume $\K=\mathbb{C}$. To see this, because at least one of $N^{\mathbf{a}}$ or $N^{\mathbf{b}}$ is non-zero, then $SH^0(X)$ is smooth. We therefore have, \begin{align} \label{eq:isomorphismcat}  \Perf(\mathcal{W}(X)) \cong  \Perf(\Spec(SH^0(X))) \end{align} by Corollary \ref{cor:strongresmain}. Suppose only one of either $N^{\mathbf{b}}$ or $N^{\mathbf{b}}$ is non-zero, in which case $\Spec(SH^0(X)) \cong (\mathbb{C}^*)^n.$ Consider what the equivalence \eqref{eq:isomorphismcat} would imply on the level of periodic cyclic homology. On the algebro-geometric side, we would have \cite{MR800920}:  $$ \operatorname{HP}_*(\Perf(\mathbb{C}^*)^n) \cong H^*((\mathbb{C}^*)^n)((\beta)). $$ Because $\chi(H^*((\mathbb{C}^*)^n))=0$,  Lemma \eqref{lem:Zhao} and Lemma  \ref{lem:eulerchar} gives a contradiction. \vskip 10 pt

 Since both  $N^{\mathbf{a}}$ or $N^{\mathbf{b}}$ are non-zero, we obtain $SH^0(X) \cong \mathcal{A}$ by rescaling the generators. The result now follows from the isomorphism \eqref{eq:isomorphismcat}.  \end{proof} 

We can bootstrap this example to prove HMS in a number of other cases using the standard correspondence between finite abelian coverings and quotients in mirror symmetry.   
To be more precise, let $G_0$ be a finite index subgroup of the fundamental group of $X$.
Denote the quotient group by $G=\pi_1(X)/G_0$ and let $X_{G_0}$ be the corresponding finite cover of $X$.
There is a canonical lift of $L_0$ to this finite cover (again  taking ``the positive real locus") and hence we can index all possible lifts $L_{g}$ by elements of $g \in G$.
We set $L = \bigoplus_g L_g \in \mathcal{W}(X)$. On the mirror side, there is a corresponding action of $G^{\vee}:= \Hom(G,\mathbb{C}^*)$ on $Y=\Spec(\mathcal{A})$ by diagonal linear symmetries on $Y$ preserving the holomorphic volume form (see e.g. \cite[Section 8]{CPU}).We take $Y/ G^{\vee}$ to denote the resulting Calabi-Yau quotient orbifold.


\begin{cor} \label{cor:Gcovers} Suppose $\operatorname{char}(\K)=0.$ We have an equivalence of categories  $$ \Perf(\mathcal{W}(X_{G_{0}})) \cong \Perf(Y/G^{\vee})$$  \end{cor}
\begin{proof}
\emph{Covering side:} It is well-known that the endomorphisms of $L$ can be computed via a certain cross-product algebra: $$ \operatorname{Hom}_{\mathcal{W}(X_{G_{0}})}(L,L) \cong WF^0(L_0,L_0) \rtimes G $$  The fact that $L$ split-generates the wrapped Fukaya category of $X_{G_{0}}$ follows from the fact that $WF^0(L_0,L_0) \rtimes G$ is friendly together with Lemma \ref{lem:admissibility} (alternatively this can be deduced easily from Abouzaid's generation criterion together with the fact $L_0$ split generates $\mathcal{W}(X)$). \vskip 10 pt
\emph{Orbifold side:} The endomorphism algebras of $\bigoplus_g \mathcal{O}_{Y} \otimes V_g$  are easily seen
to agree with the crossed product algebra $\mathcal{A} \rtimes G$. The objects  $\mathcal{O}_{Y} \otimes V_g$ generate the category $\Perf(Y/G^{\vee})$. \vskip 5 pt

It follows that sending $$ L \to \bigoplus_g \mathcal{O}_{Y} \otimes V_g$$ gives an equivalence of categories. 
\end{proof}

\subsection{Singularities of \texorpdfstring{$SH^0(X,\underline{\Lambda})$}{SH^0(X)}} \label{subsection:singy}
Throughout this section we will make the following standing assumptions: \vskip 5 pt 

\textbf{Standing assumptions:} \begin{itemize} \item $\K$ is a field of characteristic zero. 
 \item All strata of $\D$ are connected,  so the dual intersection complex is simplicial. 
 \end{itemize}  \vskip  5 pt 

 We have seen that Theorem \ref{thm:additiveiso} implies that $SH^0(X,\underline{\Lambda})$ is finitely generated as a ring (Lemma \ref{lem:finiteness2}). We now turn to somewhat deeper applications of this Theorem, which concern the singularities of the family of varieties: \begin{align} \label{eq:family} \operatorname{Spec}(SH^0(X,\underline{\Lambda})) \to \operatorname{Spec}(\Lambda) \end{align}

 The main result of this section is the following proposition:

\begin{prop} \label{prop: singmain} Let $(M,\D)$ be a maximally degenerate Calabi-Yau pair of dimension $n$ and let $\K$ be a field of characteristic zero. For any ($\K$)-point $s \in  \operatorname{Spec}(\Lambda)$, the fiber $\operatorname{Spec}(SH^0(X,\underline{\Lambda}))_s$ is a reduced n-dimensional scheme of finite type which has Gorenstein, Du Bois singularities. Furthermore it is Calabi-Yau.  \end{prop}

\begin{example} \label{ex: conicbundles} The following class of examples (generalizing the example studied in \S \ref{subsection:Gross-Siebert}) is useful for illustrating Proposition \ref{prop: singmain}. Let $f: (\mathbb{C}^*)^{n-1} \to \mathbb{C}$ be a Laurent polynomial whose zero set $Z^o$ is smooth. We define the \emph{affine conic bundle} 
to be the affine variety $X$ defined by the equation: \begin{align} 
    \label{eq: conicbundle} X=\lbrace (u,v, \bar{x}) \in \mathbb{C}^2 \times (\mathbb{C}^*)^{n-1} |\; uv=f(\bar{x}) \rbrace 
\end{align}  
Mirrors to these varieties were first constructed in the physics literature \cite{HIV} and then later from the point of view of SYZ fibrations in \cite{AAK}. 
  To describe the mirror, let $\mathcal{P}$ denote the Newton polytope of $f(\bar{x})$ in $M_\mathbb{R}$ and let $$\operatorname{cone}(\mathcal{P}) \subset M_\mathbb{R}
 \oplus \mathbb{R}$$ be the cone over $\mathcal{P}$ viewed as a fan in  $M_\mathbb{R} \oplus \mathbb{R}.$ Associated to this is a Gorenstein affine toric variety $\bar{Y}_{\operatorname{aff}}.$ Let
further $H$ be an anticanonical divisor in $\bar{Y}_{\operatorname{aff}}$ defined by the
function $p := \chi_{0,1} - 1$, where $\chi_{\mathbf{n},k} : \bar{Y}_{\operatorname{aff}} \to
\mathbb{C}$ is the function
associated with the character $(\mathbf{n},k) \in N \oplus \mathbb{Z}$ (here
$N$ denotes the dual lattice to $M_\mathbb{Z} \subset M_\mathbb{R}$ as is
standard in the literature on toric varieties) of the dense torus of
$\bar{Y}_{\operatorname{aff}}$. Let $Y_{\operatorname{aff}}$ be the affine variety given by taking the complement of $H$ in $\bar{Y}_{\operatorname{aff}}$: \begin{align} Y_{\operatorname{aff}}:=\bar{Y}_{\operatorname{aff}} \setminus H  \end{align}  Mirror symmetry predicts that \begin{align} SH^0(X,\K) \cong \Gamma(\mathcal{O}_{Y_{\operatorname{aff}}}) \end{align}   
and this prediction has been verified in the case when $\dim(X)=3$ in \cite{CPU}. \end{example} 

To prove Proposition \ref{prop: singmain}, we first recall some definitions and facts from combinatorial commutative algebra. 

\begin{defn} Let $\Delta$ be an (abstract) simplicial complex with vertices $\lbrace e_1,\cdots, e_k \rbrace$. The Stanley-Reisner rings $\mathcal{SR}_\K(\Delta)$  with ground ring $\K$ is the ring $\K[x_1, \cdots, x_k]/(I_\Delta)$ where $I_\Delta$ is the ideal generated by all monomials $x_{i_{1}} \cdots x_{i_{s}}$ such that $\lbrace e_{i_{1}}, \cdots ,e_{i_{s}} \rbrace \notin \Delta.$  \end{defn}   

Our interest in these rings is that, as noted in the introduction,  the ring $\bigoplus_\v \Lambda \cdot \theta_\v$ equipped with the product from \eqref{eq: Amodule} is isomorphic to the Stanley-Reisner ring $\mathcal{SR}_\Lambda(\Delta(\D))$ on the dual intersection complex: \begin{align} (\bigoplus_\v \Lambda \cdot \theta_\v,\ast_{\operatorname{GS}}) \cong \mathcal{SR}_\Lambda(\Delta(\D)) \end{align}
(Recall our standing assumption that all strata of $\D$ are connected which ensures that $\Delta(\D)$ is a simplicial complex.) As the Stanley-Reisner rings are combinatorially defined, it is relatively easy to understand their singularities. For example, there is a classical criterion (Lemma \ref{lem:BrunsHerz}) for when these rings are Gorenstein. Very recently, Kollar and Xu have shown that the dual intersection complexes of maximally degenerate Calabi-Yau pairs satisfy this criterion: 

\begin{thm} \label{thm:KX} \cite{MR3539921} Let $(M,\D)$ be a maximally degenerate Calabi-Yau pair and let $\K$ be a field of characteristic zero. Then the dual intersection complex $\Delta(\D)$ satisifes Condition (1) and (2) of  Lemma \ref{lem:BrunsHerz}. \end{thm} 
\begin{proof} This is implicit but not explicitly stated in \cite{MR3539921}. (In this proof, all references to page numbers or Propositions will be to that paper.) 
We let $d=\operatorname{dim}_\mathbb{C}M-1.$ \vskip 5 pt 
\emph{Claim 1:} We first explain why $|\Delta(\D)|$ is a rational homology sphere.  \emph{Proposition 31} shows that $\tilde{H}_i(\Delta(\D);\K)=0$ for $i<d$ and the fact that $H_i(\Delta(\D),\K)=\K$ is proven at the bottom of \emph{Paragraph 32} (see also the top of \emph{Paragraph 33}). \vskip 5 pt 

 \emph{Claim 2:} We show that $|\Delta(\D)|$ is a rational homology manifold, meaning that for any $p$, \begin{align*} H_i(|\Delta(\D)|,|\Delta(\D)|\setminus p; \K)=0, \quad \textit{for} \quad i<d \\ H_d(|\Delta(\D)|,|\Delta(\D)|\setminus p; \K)= \K \end{align*} For a given $p$ , let $F$ be the face such that $p$ lies in its interior. By \emph{Claim 32.1} the dimension at any point is $d$ and so if the link of $F$, $\operatorname{lnk}(F)$, is empty, then $p$ is locally a $d$-manifold about $p$ ($p$ then lies in a maximal dimensional face). Otherwise, we have that $$ H_i(|\Delta(\D)|,|\Delta(\D)|\setminus p; \K) \cong \tilde{H}_{i-j-1}(\operatorname{lnk}(F),\K)$$ where $j$ is the dimension of $F$. The face $F$ corresponds to a non-empty $d-j-1$ dimensional stratum $D_F$ of $\D$ and the pair $(D_F, \cup_{i \notin F}(D_i \cap D_F))$ determines a maximally degenerate Calabi-Yau pair. Moreover, the dual intersection complex of this pair is PL-homeomorphic to $\operatorname{lnk}(F)$. \emph{Claim 2} thus follows from \emph{Claim 1} applied to these lower dimensional Calabi-Yau pairs.  \end{proof} 

\begin{cor} \label{cor: GorStan} Let $(M,\D)$ be a maximally degenerate Calabi-Yau pair and $\K$ a field of characteristic zero, then $\mathcal{SR}_\K(\Delta(\D))$ is Gorenstein. \end{cor} 
\begin{proof} Combine Theorem \ref{thm:KX} and Lemma \ref{lem:BrunsHerz}. \end{proof} 

\begin{rem} \begin{itemize} \item If $\operatorname{dim}_\mathbb{C}M=n \leq 5$, then it is known that $|\Delta(\D)| \cong S^{n-1}$ (see page of \cite{MR3539921}), so in this case Corollary \ref{cor: GorStan} holds over a field of any characteristic. \item Without the maximally degenerate hypothesis,  Condition (1) of  Lemma \ref{lem:BrunsHerz} still holds (in characteristic zero), which implies $\mathcal{SR}_\K(\Delta(\D))$ is Cohen-Macaulay \cite[Cor. 5.3.9]{MR1251956}. \end{itemize} \end{rem} 

Let $Y$ be a variety (=reduced, separated scheme of finite type) over a field of characteristic zero $\K$. Du Bois, following ideas of Deligne, constructed a filtered complex $\Omega^{\bullet}_Y$, whose associated graded pieces $Gr_F^p\Omega^{\bullet}_Y$ are complexes of coherent sheaves. If $\K=\mathbb{C}, \Omega^{\bullet}_Y$  generalizes the De Rham complex of a smooth variety in the sense that its analytification gives a resolution of the constant sheaf $\underline{\mathbb{C}}$ on the analytification of $Y$, $Y^{\operatorname{an}}$. There is a natural map in $\operatorname{D^bCoh}(Y)$,  \begin{align} \label{eq: DuBois}  \mathcal{O}_Y \to  Gr_F^0 \Omega^{\bullet}_Y \end{align} 

A singularity is called Du Bois if  \eqref{eq: DuBois} is a quasi-isomorphism.  Du Bois observed that these singularities enjoy many of the nice Hodge theoretic properties of smooth varieties. Koll\'ar has shown that these singularties provide a natural context for proving Kodaira vanishing theorems, making them important in the minimal model program. We have the following ``folklore" fact, whose proof we defer to the appendix (Proposition \ref{prop:DuBoisgen}):  

\begin{prop} \label{prop:DuBoismain}  For any field $\K$ of characteristic zero and any simplicial complex $\Delta$, $\mathcal{SR}_\K(\Delta)$ is Du Bois. \end{prop} \vskip 10 pt

\emph{Proof of Proposition \ref{prop: singmain}:} By Corollary \ref{cor: GorStan} and Proposition \ref{prop:DuBoismain}, we have that $\mathcal{SR}_\K(\Delta)$ is Du Bois. Moreover, by Theorem \ref{thm:additiveiso}, there is a filtered degeneration from $\operatorname{Spec}(SH^0(X,\underline{\Lambda}))_s$ to $\mathcal{SR}_\K(\Delta).$ We can therefore apply Claim (1) of Lemma \ref{lem:assocgraded} to conclude that $\operatorname{Spec}(SH^0(X,\underline{\Lambda}))_s$ is Gorenstein and Du Bois as well.  \vskip 10 pt

For any Calabi-Yau pair $(M,\D)$ (as usual equipped with some polarization $\mathcal{L}$), the filtration $F_w SH^0(X,\underline{\Lambda})$ is positive and we can form the Rees algebra $\mathcal{R}(X,\underline{\Lambda}):= \mathcal{R}(SH^0(X,\underline{\Lambda})).$ Taking $\operatorname{Proj}$ of this algebra determines a projective compactification of the family \eqref{eq:family}: \begin{align} \label{eq:projfamily} \operatorname{Proj}(\mathcal{R}(X,\underline{\Lambda})) \to \operatorname{Spec}(\Lambda) \end{align} 

The fibers of this are given by $\operatorname{Proj}(\mathcal{R}(X,\underline{\Lambda})_s)$ where $\mathcal{R}(X,\underline{\Lambda})_s := \mathcal{R}(SH^0(X,\underline{\Lambda})_s).$ If our initial pair $(M,\D)$ is Fano and we take our line bundle $\mathcal{L}=\mathcal{O}(\sum_i D_i)$, then the following analogue of Proposition \ref{prop: singmain} holds for this compactified pair.

\begin{prop} Let $(M,\D)$ be a Fano manifold (with polarization $\mathcal{L}=\mathcal{O}(\sum_i D_i)$) then the fibers of \eqref{eq:projfamily} are Gorenstein and Du Bois.    \end{prop} 
\begin{proof} In the Fano case, the rings $SH^0(X,\underline{\Lambda})_s$ are generated in weight one because all of the $\theta_i$ variables arising in the proof of Lemma \ref{lem:finiteness} have weight one. We can therefore apply Claim (2) of Lemma \ref{lem:assocgraded}. \end{proof} 

\begin{rem} In Appendix \ref{section:appendixC}, we discuss an example of Gross-Siebert intrinsic mirror family $(\mathcal{A}_\Lambda, \ast_{GS})$ where all of the fibers are singular.  We expect (though at present cannot prove) that the same holds for the symplectically constructed family \eqref{eq:family} in these examples.  \end{rem}

\subsection{Categorical crepant resolutions} \label{sec:ccr}

\emph{Throughout this section,  we keep the standing assumptions from the beginning of \S \ref{subsection:singy} in place. } \vskip 5 pt For any finitely generated $\K$-algebra $R$, let $\operatorname{dgcat}_R^{\operatorname{idm}}$ denote the $\infty$-category of pre-triangulated, idempotent complete $R$-linear dg-categories. The following is a dg-version of Kuznetsov's notion of a categorical resolution with an action of $\Perf(R)$ (see Definition 3.2 and discussion following Definition 3.4 of \cite{MR2403307}):

\begin{defn} \label{defn:cr} Let $R$ be a finitely generated $\K$-algebra. A categorical resolution of $\Spec(R)$ is a pair $(\mathcal{C}, \pi^*)$ where \begin{itemize} \item $\mathcal{C} \in  \operatorname{dgcat}_R^{\operatorname{idm}}$ is smooth over $\K$ and \begin{itemize} \item $\HH_\K^0(\mathcal{C})=R$ \item $\mathcal{C}$ is semi-affine. \end{itemize} \item $\pi^*$ is an $R$-linear fully-faithful embedding \begin{align} \pi^*: \Perf(R) \hookrightarrow \mathcal{C} \end{align} \end{itemize}  \end{defn} 

 Categorical resolutions are the central objects in a circle of ideas known as ``categorical birational geometry" \cite{MR1957019,MR2403307} (see also Conjecture \ref{conj:BO} below). In particular, one expects $\Spec(R)$ to have many categorical resolutions, just like a variety has many resolutions. For example, if $\Spec(R)$ has rational singularities, then for any resolution of singularities $\pi: Y \to \Spec(R)$, the natural embedding $\pi^*: \Perf(R) \hookrightarrow \Perf(Y)$ gives a categorical resolution of singularities (see Remark \ref{rem:fun} for a symplectic perspective on this).  By analogy with traditional birational geometry, we would like to focus on those categorical resolutions which are ``minimal" in a suitable sense. This leads to the notion of a crepant categorical resolution. Kuznetsov's approach to defining crepancy in \cite[Definition~3.5]{MR2403307} is to require that the identity is a relative Serre functor for $\pi_*$ in a suitable sense (he refers to such resolutions as ``strongly crepant").  However, at least at first glance, there  does not seem to be a natural geometric interpretation of this condition in the Fukaya category context. Instead, we take the following approach:  

\begin{defn}\label{defn:ccr} Let $R$ be a finitely generated $\K$-algebra such that $\Spec(R)$ in an n-dimensional Gorenstein Calabi-Yau variety. We say that a categorical resolution $(\mC, \pi^*)$ is crepant if $\mC$ is Calabi-Yau of dimension $n$ (see \S \ref{subsection:ha}  for a review of this notion). \end{defn} 

\begin{rem} \label{rem:crepe} Definition \ref{defn:ccr} can be generalized to the case when $\Spec(R)$ is Gorenstein but not Calabi-Yau by replacing \eqref{eq:CYcondition} with an isomorphism: \begin{align} \mC^{!} \cong \mC \otimes_R \omega_R^{-1}[-n] \end{align} Taking this definition, it should not be difficult to show that a strongly crepant resolution in the sense of \cite[Definition~3.5]{MR2403307}  is crepant (justifying the terminology). We are less sure about the reverse implication, which is however known when $\mC$ is Morita equivalent to the category of perfect modules over an (ordinary) algebra \cite[Section 7.2]{GinzburgCY}. \end{rem}

The following central conjecture, which is a variant of fundamental conjectures of Bondal-Orlov \cite[Conjecture~5.1]{MR1957019} and Kuznetsov \cite[Conjecture~4.10]{MR2403307}, states that this crepancy condition should determine the categorical resolution uniquely (up to isomorphism): 

\begin{conj} \label{conj:BO} If $\Spec(R)$ is a normal, Gorenstein Calabi-Yau variety, then categorical crepant resolutions of $\Spec(R)$ are unique. Moreover, in dimension 3, categorical crepant resolutions exist iff and only $\Spec(R)$ has a commutative crepant resolution.  \end{conj}

 Conjecture \ref{conj:BO} is in general wide open, but there has been partial progress in certain special but important settings \cite{VandenBergh, MR3181550} (see also \S \ref{subsection:questions} for more details). In any case, the notion of categorical crepant resolution satisfies at least one nice property that justifies the terminology \emph{resolution} : it is trivial over the smooth locus of $\Spec(R)$. To state this formally, recall from \cite{Toen} or \cite[6.3.1.14, 6.3.1.17]{HA} that $\operatorname{dgcat}_R^{\operatorname{idm}}$ is equipped with a tensor product  $$(\mathcal{C},\mathcal{D}) \to \mathcal{C}\hat{\otimes}_R \mathcal{D}$$ We then have that: 

\begin{lem} \label{lem:cattriv} Let $\mC$ be a categorical crepant resolution of $\Spec(R).$ Then for any affine subset $\Spec(B)$ of the regular locus of $\Spec(R)$, the embedding \begin{align} \pi_B^*: \Perf(B) \cong \Perf(R) \hat{\otimes}_R \Perf(B) \to \mC \hat{\otimes}_R \Perf(B)  \end{align} is an equivalence of categories.  
 \end{lem}
\begin{proof} It follows from Lemma \ref{lem:nCYquotient} that $\mC \hat{\otimes}_R \Perf(B)$ smooth n-CY. Moreover,  $\mC \hat{\otimes}_R \Perf(B)$ is generated by objects of the form $L \hat{\otimes}_R B$, and for any two such objects $L_1, L_2$, $$ \operatorname{Hom}(L_1 \hat{\otimes}_R B,L_2\hat{\otimes}_R B) \cong \operatorname{Hom}(L_1,L_2)\otimes_R B.$$  Thus, $\mC \hat{\otimes}_R \Perf(B)$ is semi-affine as well and $\pi_B^*$ is fully-faithful. Thus $(\mC \hat{\otimes}_R \Perf(B), \pi_B^*)$ is a categorical crepant resolution of $\Perf(B)$. So it suffices to consider the case where $B=R$ i.e. $R$ is smooth. Then the image of $\pi^*$ is an admissible subcategory of $\mC$ which must be all of $\mC$ by Corollary \ref{cor: automaticgen}. \end{proof} 

Lastly, we will need the following ``rectification" statement:

\begin{lem} \label{lem:rectification} Let $\mathcal{C}$ be a pretriangulated dg-category with $\HH^*(\mathcal{C})=0$ for $*<0.$ Then $\mathcal{C}$ is quasi-equivalent to a dg-category which is linear over $\HH^0(\mathcal{C}).$ \end{lem}
\begin{proof} There seems to be no proof of this statement in the literature which uses the language of dg-categories, however it can be extracted from the literature on $(\infty,1)$ categories. More precisely, the result is a combination of three facts: \vskip 5 pt

\emph{Fact I:} For any commutative ring $R$, there is an equivalence of ($\infty$-) categories between the category of pre-triangulated dg-categories over $R$ and the category of stable $(\infty,1)$ categories equipped with a monoidal action of $\Perf(R)$(\cite{LeeCohn}). \vskip 5 pt

\emph{Fact II:} For any $\K$-algebra $R$ (or more generally $E_2$ algebra over $\K$), to specify a monoidal action of $\Perf(R)$ on an $(\infty,1)$ category $\mathcal{C}$ which is compatible with an existing $\K$-linear structure on $\mathcal{C}$ is equivalent to specifying an $E_2$ algebra map $$R \to \operatorname{Hom}_{End(C)}(id,id) $$   where  $\operatorname{Hom}_{End(C)}(id,id)$ denotes the endomorphisms of the identity functor inside the category of $\K$-linear continuous endo-functors of $\mathcal{C}$ to itself. We have $H^*(\operatorname{Hom}_{End(C)}(id,id)) \cong 
\HH^*(\mathcal{C})$ (see \cite[Section 5.3.1]{HA} for a full-treatment or \cite[Appendix]{MR3300415} for a friendly summary). \vskip 5 pt 

\emph{Fact III:}  Any $E_2$ algebra $S$ has a connective cover $\tilde{S} \to S$ whose cohomology groups vanish in positive degrees and such that in nonpositive degrees $H^*(\tilde{S}) \to H^*(S)$ is an isomorphism (\cite[Proposition 7.1.3.13]{HA}). \vskip 5 pt  

Returning to the statement of the lemma, the assumption that $\HH^*(\mathcal{C})=0$ implies that the connective cover $\widetilde{\operatorname{Hom}_{End(C)}(id,id)}$ of $\operatorname{Hom}_{End(C)}(id,id)$ has non-vanishing cohomology in a single degree, equal to $\HH^0(\mathcal{C}).$ It follows that it is equivalent to $\HH^0(\mathcal{C})$ as an $E_2$ algebra and that we have a map: $$ \HH^0(\mathcal{C}) \to \operatorname{Hom}_{End(C)}(id,id)$$ which in view of \emph{Fact I} and \emph{Fact II} gives $\mathcal{C}$ the desired $\HH^0(\mathcal{C})$-linear structure. 
 \end{proof}

\begin{cor} \label{cor:SHlin}  $\WXp$ is quasi-equivalent to a dg-category which is $SH^0(X)$ linear. \end{cor}
\begin{proof} For affine log Calabi-Yau varieties, $SH^*(X)$ is concentrated in non-negative degrees. This follows immediately from Theorem 1.1 of \cite{GP2}, although it is more elementary. The fact that \eqref{eq:fullCO} is an isomorphism implies the same for $\HH^*(\WXp).$ We can now invoke Lemma \ref{lem:rectification}. \end{proof}

In the following theorem, we let $\WXp^{\otimes_{SH}}$ denote a strictly $SH^0(X)$ linear model for $\WXp.$

\begin{thm} \label{thm:mainresult} If $(M,\D)$ is a maximally degenerate Calabi-Yau pair. Suppose that $X$ is equipped with a homological section $L$. Then the induced embedding $$ \pi^*: \Perf(SH^0(X)) \hookrightarrow \WXp^{\otimes_{SH}} $$ turns $\WXp^{\otimes_{SH}}$ into a categorical crepant resolution of $SH^0(X).$  \end{thm}
\begin{proof} The fact that $\WXp^{\otimes_{SH}}$ is a categorical crepant resolution is a matter of definitions.  \end{proof}

\begin{example} \label{ex: conicbundleswrap} Continuing with Example \ref{ex: conicbundles}, a unimodular triangulation $\Sigma$ of $\mathcal{P}$ determines a crepant resolution $\pi: Y_\Sigma \to Y_{\operatorname{aff}}.$ In such cases, one further expects an equivalence: \begin{align} \mathcal{W}(X) \cong \Perf(Y_\Sigma) \end{align} 
However, in dimension $\geq 3$, there are Newton polytopes $\mathcal{P}$ which do not have unimodular triangulations, which means that the affine varieties $Y_{\operatorname{aff}}$ (of dimension $\geq 4$) don't admit crepant resolutions. However, one can still construct stacky resolutions of these singularities using toric geometry.\footnote{For interesting non-commutative interpretations of these stacky resolutions, see \cite{Vandenberghspenko}.} \end{example} 

\begin{rem} \label{rem:fun} When the pair $(M,\D)$ is only log nef, we expect a similar result for a partially wrapped Fukaya category $\mathcal{W}(X,\frak{f})$ where $\frak{f}$ is a suitably defined stop \cite{GPS}.\footnote{Implicit in this should be an isomorphism: $SH^0(X,\frak{f}) \cong SH^0(X)$.} However, the result would be categorical resolutions which are not crepant (because partially wrapped categories are not Calabi-Yau) meaning that there is no hope that such a characterization could help to pin down the category. In fact, given a maximally degenerate log Calabi-Yau pair $(M,\D)$, we can take any collection of ample divisors $E$ which meet $\D$ transversely and consider the log nef pair $(M,\D \cup E)$ (take a volume form which is non-vanishing on $M \setminus \D$). One expects that choosing different $E$ should typically give rise to different categorical resolutions of the Stanley Reisner rings $\mathcal{SR}(\Delta(\D)).$ For beautiful illustrations of these expectations, see \cite{LekiliPolishchuk}.  \end{rem} 

\begin{cor} For any affine subset $\Spec(B)$ in the regular locus of $\Spec(SH^0(X))$, there is an equivalence of dg-categories
\begin{align} \label{eq:birationalmirror} \Perf(B) \cong \WXp^{\otimes_{SH}} \otimes_{SH^0(X)} B \end{align} 
   \end{cor}
\begin{proof} Equation \eqref{eq:birationalmirror} follows from Theorem \ref{thm:mainresult} and Lemma \ref{lem:cattriv}. \end{proof}

The equivalence \eqref{eq:birationalmirror} may be viewed as saying that HMS holds ``birationally" for the intrinsic mirror partner. It also implies a similar statement relating localized $SH^*(X,\K)$ and polyvector fields over $B$. One can sometimes say more provided one has a ``nice" generator for the Fukaya category. The relevant definition is the following: 

\begin{defn} \label{defn: tilting} Let $\mathcal{C}$ be a pre-triangulated dg-category. We say that $E$ is a tilting generator if $E$ split generates $H^0(\mathcal{C})$ and $H^*(\operatorname{Hom}_{\mathcal{C}}^{\bullet} (E,E))=0$ if $\ast \neq 0.$ \end{defn}


 To explain the significance of these tilting generators, note that the case where $\operatorname{Spec}(SH^0(X))$ is smooth, Corollary \ref{cor:strongresmain} gives an effective strategy for proving HMS. To handle more general cases, one needs to decide when the categorical crepant resolution has a commutative interpretation.  This is wide open in general, however, there is some important progress on this assuming the categorical crepant resolution is the category of modules over a \emph{noncommutative crepant resolution} (nccr) \cite{VandenBergh}. 

\begin{defn} \label{defn:Vnccr} Let $R$ be a Gorenstein, normal, Calabi-Yau domain and let  $\mathcal{A}$ be a module finite $R$-algebra with $Z(\mathcal{A})=R$. $\mathcal{A}$ is called a non-commutative resolution of $R$ if: \begin{enumerate} \item $\mathcal{A}$ has finite homological dimension.\item $\mathcal{A}$ is maximal Cohen Macaulay as an $R$-module. (This means that $\operatorname{Ext}^{i}_R(\mathcal{A},R)=0$ for $i>0.$) \item $\mathcal{A}=End_R(M)$ for some reflexive R-module $M$. \end{enumerate} \end{defn}

\begin{thm} \label{thm:vdb} (\cite{VandenBergh})  Suppose either $\operatorname{dim}(R)=2$ or $\operatorname{dim}(R)=3$ and $R$ has terminal singularities and let $\mathcal{A}$ be an $R$-nccr. Then $\operatorname{D^bCoh}(\mathcal{A}) \cong \operatorname{D^bCoh}(Y)$, where $Y$ is a(any) crepant resolution of $\operatorname{Spec}(R).$\footnote{Van den Bergh conjectures that the terminal hypothesis can be dropped; see Conjecture 4.6 of \cite{VandenBergh}.}\end{thm} 

\begin{lem} Let $R$ be a Gorenstein, normal, Calabi-Yau domain and let $(\mathcal{C},\pi^*)$ be a  categorical crepant resolution with a tilting generator $E$. Then $\mathcal{A}=\operatorname{Hom}(E,E)$ is an nccr. \end{lem} 
\begin{proof}  For algebras, $HH^0(\mathcal{A})=Z(\mathcal{A})$ and so $Z(\mathcal{A})=R$. The module finiteness follows from the fact that $\mathcal{C}$ is semi-affine. We next explain why the remaining three conditions of Definition \ref{defn:Vnccr} hold.  \vskip 5 pt

\emph{Condition (1):} Because $\mathcal{A}$ is homologically smooth, \cite[Prop. 18.6.2]{MR3971537}) shows that $\mathcal{A}$ has finite homological dimension.  \vskip 5 pt

\emph{Condition (2):} Note that because $\mathcal{A}$ is $n$-CY, the relative Serre functor over $R$ is trivial by \cite[Section 7.2]{GinzburgCY}. Concretely, this means that there is a quasi-equivalence of $\mathcal{A}-\mathcal{A}$ bimodules: $$\operatorname{RHom}^{\bullet}_R(\mathcal{A},R) \cong \mathcal{A} $$  In particular we have that $\mathcal{A}$ is maximal Cohen Macaulay. \vskip 5 pt 

\emph{Condition (3):} Because $R$ is non-singular in codimension one, Lemma \ref{lem:cattriv} implies that for any height one prime $p$, the base change $\mathcal{A}_p:=\mathcal{A}\otimes_R R_p$ is derived equivalent $R_p$. As derived equivalence and Morita equivalence are equivalent for local rings (such as $R_p$) by \cite[Corollary 2.13]{MR1978574}, we have that $\mathcal{A}_p$ is Morita equivalent to $R_p$ as well. We may therefore invoke \cite{MR3251829}[Proposition 2.11] to conclude that $\mathcal{A}=End_R(M)$ for a reflexive R-module $M$, concluding the proof. \end{proof}

Concretely, for wrapped Fukaya categories with a homological section and tilting generator, $\mathcal{W}(X)$ will be Morita equivalent to an nccr. This yields the following corollary:  

\begin{cor} \label{cor:lowdmirror} Let $(M,\D)$ be as in Proposition \ref{cor:ccr} with $\dim_\mathbb{C}M \leq 3$. Suppose that $\operatorname{Spec}(SH^0(X,\K))$ is integral with terminal singularities and that $\Perf(\mathcal{W}(X))$ admits a tilting generator (Definition \ref{defn: tilting}). Then $\operatorname{Spec}(SH^0(X,\K))$ admits a crepant resolution $Y$ and there is a derived equivalence: $$\Perf(\mathcal{W}(X)) \cong \Perf (Y) $$  \end{cor}

\subsubsection{Further questions} \label{subsection:questions}

We finish this article by discussing a few possible strengthenings of our results that seem potentially tractable and, at least in the author's mind, would represent significant advances. The first two questions ask more refined questions about the singularities of $\operatorname{Spec}(SH^0(X,\K))$. Perhaps the most important question is:
\begin{ques} Is $\operatorname{Spec}(SH^0(X))$ always a normal domain?  \end{ques}

 The most promising approach to answering this question seems to involve the conjectural description of $SH^0(X,\K))$ in terms of rational curves which was hinted at in \S \ref{subsection:ims}. For example, in the algebro-geometric context \cite{KY} have shown that the fibers of $\Spec(\mathcal{A}_\Lambda)$ are all normal varieties whenever $X$ has a dense algebraic torus. More generally, the dimension one strata of $\D$ are rational curves which contribute a certain term to the multiplication law on $(\mathcal{A}_\Lambda,*_{\operatorname{GS}}).$ In the absence of other contributions, this term has the effect of smoothing out the codimension one crossings singularities of the Stanley-Reisner rings. It may therefore be possible to use a filtration/deformation theory argument to show that these singularities are smoothed out, even in the presence of further contributions. Another interesting question is whether the singularities of $\operatorname{Spec}(SH^0(X))$ are rational. In fact, we can ask a broader question: 

\begin{ques} \label{ques:Stafford} Let $\Spec(R)$ be a normal Gorenstein Calabi-Yau variety and suppose that $(\mC,\pi^*)$ is a categorical crepant resolution of $\Spec(R)$. Does $\Spec(R)$ have only rational singularities?  \end{ques}

One piece of motivation for this is the well-known fact that crepant singularities are rational. Perhaps slightly more convincing is the fact that Van den Bergh and Stafford have given an affirmative answer to Question \ref{ques:Stafford} under restrictive hypotheses. 

A different line of inquiry concerns the existence of homological sections. The first question is whether we can remove the hypotheses from Theorem \ref{thm:mainresult}. 

\begin{ques} Do homological sections exist always? For which maximally degenerate pairs $(M,\D)$ does $\mathcal{W}(X)$ admit a tilting collection? \end{ques}   

As noted above, in \cite{sequel} we will use an extension of the arguments of Lemma \ref{lem:homologicalsec} to show that $\mathcal{W}(X)$ admits a homological section whenever $(M,\D)$ admits a toric model. It is less clear how to proceed without any assumption on the geometry. For example, in low dimensions, it should be easy to construct Legendrian submanifolds $\partial L$ of $\partial \bar{X}$ which behave like boundaries of geometric sections. However, it is not obvious how to give an explicit condition under which they are fillable (the most obvious conditions involve the behavior of $\partial L$ under the Liouville flow and seem hard to work with).

Finally, the assumption that $\Perf(\mathcal{W}(X))$ admits a tilting collection does not at first glance seem natural from a symplectic perspective. An important feature of a tilting collection is that it provides a natural t-structure. An alternative source of t-structures are Bridgeland stability conditions, which conjecturally do admit geometric constructions. This leads us to the following question: 

\begin{ques} Suppose $(M,\D)$ satisfy the hypotheses of \ref{thm:mainresult}, $\operatorname{dim}_\mathbb{C}(M)=3$, and $\Perf(\mathcal{W}(X))$ admits a Bridgeland stability condition. Then  does $\operatorname{Spec}(SH^0(X,\K)$ admit a crepant resolution $Y$ such that  $\Perf(Y)$ is derived equivalent to $\mathcal{W}(X)$? 
 \end{ques}

In fact, a candidate crepant resolution $Y$ is given by the moduli space of stable point like objects in  $\mathcal{W}(X)$. The key point is to show that this moduli space is well-behaved, namely a smooth semi-affine variety. Once this is established, such moduli spaces come equipped with universal Fourier-Mukai kernels that should represent the desired equivalence.

\appendix
\section{Commutative and non-commutative geometry} \label{section:appendixA}

In this section, we collect a couple of lemmas concerning non-commutative properties of derived categories of coherent sheaves. The first lemma is a generalization of a well-known result in the case of varieties which are proper over a field.  

\begin{lem} \label{lem:algebraicgeomprop} Let $\K$ be a field and let $Y$ be a smooth quasi-projective $\K$-scheme such that $\Perf(Y)$ is semi-affine (recall Definition \ref{defn:ucp}). Then $Y$ is projective over $H^0(\mathcal{O}_Y) \cong \HH^0(\Perf(Y)).$ \end{lem}
\begin{proof} Set $R=\HH^0(\Perf(Y))$. Then $Y$ is quasi-projective over $\K$ and hence quasi-projective over $R$ (in particular separated over $R$). We next note that a proper, quasi-projective morphism is projective and so it suffices to check that $\pi: Y \to \Spec(R)$ is universally closed. To check this, let $F$ be a coherent sheaf on $Y$. Setting $E_1=\mathcal{O}_Y$ and $E_2=F$, we learn that $\pi_*(F)$ is coherent. The result now follows from the general fact that a morphism of Noetherian schemes $\pi_*: X \to Y$ is universally closed iff $\pi_*$ preserves coherent sheaves \cite[Proposition 2.4.5 and Proposition 2.4.7]{HL-Preygel} or \cite{Rydh}. \end{proof} 

We now turn to the following proposition,  which shows that derived categories of coherent sheaves provide a natural source of friendly algebras in the sense of Definition \ref{defn: friendly}.

\begin{prop} \label{prop:agprop} Let $R$ be a finitely generated $\K$-algebra and let $Y \to \Spec(R)$ be a projective $R$-scheme which is smooth over $\K.$ Let $E$ be a compact generator of $\Perf(Y)$. Then $S^{\bullet}:= \Hom^{\bullet}(E,E)$ is friendly relative to $R.$ \end{prop}
\begin{proof} Since $E$ is a compact generator, we have an equivalence \begin{align*} \operatorname{QCoh}(Y) \to \operatorname{Mod}(S^\bullet) \\ F \to \Hom^{\bullet}(E,F). \end{align*} Therefore, it suffices to show that $F \in \Perf(Y)$ iff $\Hom^{\bullet}(E,F) \in \operatorname{D^bCoh}(R)$ is finitely generated over $R.$ This is equivalent to showing that $F \in \Perf(Y)$ iff for any perfect complex $E \in \Perf(Y)$, $\Hom^{\bullet}(E,F) \in \operatorname{D^bCoh}(R).$ This follows from the following more general claim: \vskip 5 pt

\emph{Claim:} For any projective $R$-scheme $Y$ (not necessarily smooth!), if $\Hom^{\bullet}(E,F) \in \operatorname{D^bCoh}(R)$ for all perfect $E$, then $F \in \operatorname{D^bCoh}(Y).$  \vskip 5 pt
\emph{Proof of Claim:}
If $Y$ is projective over $R$, choose an embedding $i : Y \hookrightarrow \mathbb{P}^n_R$ over $R$. Then $F$ is bounded coherent if and only if $i_\ast F$ is. Furthermore, by adjunction we have that $\operatorname{Hom}_{\operatorname{QCoh(Y)}}^*(i^*E,F)=\operatorname{Hom}_{\operatorname{QCoh}\mathbb{P}^n_{R}}^*(E,i_*F)$ meaning that $\operatorname{Hom}_{\operatorname{QCoh}(\mathbb{P}^n_{R})}^\bullet(E,i_*F) \in \operatorname{D^bCoh}(R)$ iff the same is true for $\operatorname{Hom}_{\operatorname{QCoh(Y)}}^*(i^*E,F)$. So it suffices to assume $Y=\mathbb{P}^n_R.$ 

For $Y=\mathbb{P}^n_R$, let $B_{R}= \operatorname{End}(\cO \oplus \cdots \oplus \cO(n))$ denote the Beilinson exceptional algebra over $R$. It's well-known $\cO \oplus \cdots \oplus \cO(n)$ is a split-generator for $\Perf(\mathbb{P}^n_R)$ and hence we again have an equivalence of categories \begin{align} \label{eq:beilinson} \operatorname{QCoh}(\mathbb{P}^n_R) \cong \operatorname{Mod}(B_R). \end{align} It suffices to show that objects on the right-hand side of \eqref{eq:beilinson} whose underlying complex of $R$-modules is bounded coherent correspond to objects in $\operatorname{D^bCoh}(\mathbb{P}^n_R) \subset  \operatorname{QCoh}(\mathbb{P}^n_R)$ on the left-hand side. We also recall that the exterior tensor product gives rise to an equivalence: \begin{align} \label{eq:Toenbox}  \operatorname{QCoh}(R) \otimes \operatorname{QCoh}(\mathbb{P}^n_k) \cong \operatorname{QCoh}(\mathbb{P}^n_R)  \end{align} 
which takes $\operatorname{D^bCoh}(R) \otimes \operatorname{Perf}(\mathbb{P}^n_k)$ to $\operatorname{D^bCoh}(\mathbb{P}^n_R).$
 Set $B^e= B_R \otimes_R B_R^{op}$. Then the diagonal bi-module $B_R$ has a finite resolution over $B^e$. Given such an $M \in \operatorname{Mod}(B_R)$ with bounded coherent cohomology, this means that $M$ can be built in finitely many steps from $M \otimes_{B_{R}} B^e$ (in this case the tensor product is underived because $B^e$ is flat over $B_R$).  $M \otimes_{B_{R}} B^e$ is then just the same as $M \otimes_k B_k$ where $B_k$ is the Beilinson algebra over $k$. This means $M \otimes_{B_{R}} B^e$ corresponds to a complex in $\operatorname{D^bCoh}(\mathbb{P}^n_R)$ under \eqref{eq:Toenbox} which means that $M$ does as well. 

\end{proof}

\begin{rem} Suppose $char(\K)=0$. Then using standard tricks (Chow's lemma and resolution of singularities) one can prove Proposition \ref{prop:agprop} under the weaker assumption that $Y \to \Spec(R)$ is proper as opposed to projective.  However, we will not make use of this generalization in this paper.  \end{rem}

\section{Some commutative algebra} \label{section:appendixB} 

We begin by collecting some facts concerning the singularities of these Stanley-Reisner rings. There is a well-known criterion for when the Stanley-Reisner ring is Gorenstein. For any complex $\Delta$, let $\operatorname{core}(\Delta)$ be the sub-simplicial complex generated by those vertices whose star is not all of $\Delta.$ For any simplicial complex $\Delta$, we let $|\Delta|$ denote its geometric realization. 

\begin{lem} \label{lem:BrunsHerz} Let $\Delta$ be a simplicial complex of dimension $d$ and let $\K$ be a field. Suppose that for any $p \in |\Delta|$ \begin{enumerate}\item $\tilde{H}_i(|\Delta|;\K)= H_i(|\Delta|, |\Delta| \setminus p; \K)=0$ for $i<d$ \item  $\tilde{H}_d(|\Delta|;\K)= H_d(|\Delta|, |\Delta| \setminus p; \K)=\K$  \end{enumerate}  then $\mathcal{SR}_\K(\Delta)$ is Gorenstein. \end{lem} 
\begin{proof} For any complex $\Delta$, let $\operatorname{core}(\Delta)$ be the sub-simplicial complex generated by those vertices whose star is not all of $\Delta.$ Condition (2) implies that $\operatorname{core}(\Delta)=\Delta$ (otherwise the complex can be contracted) and we therefore can apply \cite[Theorem 5.6.1]{MR1251956}. \end{proof} 

 The following proposition seems to be a ``folklore result"(see e.g.  the discussion after \cite[Theorem 1.2]{SRDuBois}.  However, as we could not find a proof in the literature, we sketch the argument:\footnote{We thank Hailong Dao for explaining the argument written below}

\begin{prop} \label{prop:DuBoisgen} For any field $\K$ of characteristic zero and any simplicial complex $\Delta$, $\mathcal{SR}_\K(\Delta)$ is Du Bois. \end{prop} 
\begin{proof} \emph{(Sketch)} The proof is based on a criterion due to Schwede \cite{Schwede}. To state it, we must recall that a ring local ring $(R, \mathfrak{m})$ of characteristic $p$ is $F$-injective if the Frobienius map induces an injection on local cohomology groups $H_\mathfrak{m}^i(R) \to H_\mathfrak{m}^i(^{1}R)$, where $^{1}R$ denotes $R$ viewed as an $R$-module by the action of Frobenius. An arbitrary ring $R$ is $F$-injective if for every prime $P$, the local ring $(R_P,P)$ is $F$-injective.

Turning to the proof of the proposition, we can assume $\K=\mathbb{Q}$. Schwede's criterion says that $\mathcal{SR}_\mathbb{Q}(\Delta)$ is Du Bois if 
$\mathcal{SR}_{\mathbb{F}_{p}}(\Delta)$ is F-injective for an open subset of $\operatorname{Spec}(\mathbb{Z}).$ Therefore it suffices to check $F$-injectivity of $\mathcal{SR}_{\mathbb{F}_{p}}(\Delta).$ Giving each of the variables $x_i$ weight one turns  $\mathcal{SR}_{\mathbb{F}_{p}}(\Delta)$ into a non-negatively graded ring (with degree zero piece $\mathbb{F}_{p}$). Letting $\mathfrak{m}$ denote the irrelevant ideal, it suffices to check that the local ring  $\mathcal{SR}_{\mathbb{F}_{p}}(\Delta)_m$ is $F$-injective \cite[Theorem 5.12]{Datta}. By Remark 4.2 of \cite{Schwede}, this follows from the F-purity of $\mathcal{SR}_{\mathbb{F}_{p}}(\Delta)_m$, which is well-known (see e.g.  \cite[\S 2]{MR2346198}).     \end{proof} 

Let  $(\mathcal{A}, F_\bullet),$ denote a (commutative, finitely-generated) $\K$-algebra $\mathcal{A}$ with an ascending filtration $F_w \mathcal{A}.$  We will assume that all of our filtrations are positive meaning \begin{itemize} \item  $F_0\mathcal{A}=\K \cdot 1$ \item $F_w\mathcal{A}=0$ for $w<0$. \end{itemize} We can form the Rees ring which as $\K$-module is given by: 
\begin{align*} \mathcal{R}(\mathcal{A}):= \bigoplus_{w \in \mathbb{N}} F_w\mathcal{A} \end{align*}

equipped with its natural multiplication. There is a natural homomorphism $\K[t] \to \mathcal{R}(\mathcal{A})$ which sends $t$ to $1 \in F_1\mathcal{A}$. We remind the reader of the two most essential properties of Rees-algebras: 

\begin{lem} For any $(\mathcal{A}, F_\bullet)$ as above: 
\begin{itemize} \item  There is an isomorphism:
\begin{align} \label{eq:awayzero} \mathcal{R}(\mathcal{A})_t \cong \mathcal{A}[t,t^{-1}] \end{align}   
where  $\mathcal{R}(\mathcal{A})_t$ denotes the localization.
\item The ``special fiber" $\frac{\mathcal{R}(\mathcal{A})}{t\mathcal{R}(\mathcal{A})}$ recovers the associated graded ring, \begin{align}  \frac{\mathcal{R}(\mathcal{A})}{t\mathcal{R}(\mathcal{A})} \cong \operatorname{gr}_F(\mathcal{A}) \end{align} 
\end{itemize} 
 \end{lem} 

In more geometric terms, we have a family \begin{align*} \pi: Y_\mathcal{R}:= \operatorname{Spec}(\mathcal{R}(\mathcal{A})) \to \mathbb{A}^1 \end{align*}

whose fiber over zero is  $\operatorname{Spec}(gr_F(\mathcal{A}))$ and whose other fibers are all isomorphic to $\operatorname{Spec}(\mathcal{A}).$ Moreover, both the total space and the base have $\mathbb{G}_m$-actions (coming from the grading on $\mathcal{R}(A)$ and giving $t$ degree 1 in the base) so that $\pi$ is $\mathbb{G}_m$-equivariant with respect to the action. It is also important to note that for any positively filtered pair $(\mathcal{A},F_\bullet)$, the variety $\operatorname{Proj}(\mathcal{R}(\mathcal{A}))$ is a compactification of $\operatorname{Spec}(A)$ (this is a direct consequence of \eqref{eq:awayzero}). Additionally, the complement $ D_\mathcal{R}:= \operatorname{Proj}(\mathcal{R}(\mathcal{A})) \setminus \operatorname{Spec}(A)$ is a divisor isomorphic to $\operatorname{Proj}(\operatorname{gr}_F(\mathcal{A})).$  This is useful because many situations (e.g. moduli theory) require working with families of pairs as opposed to affine varieties.

\begin{lem} \label{lem:assocgraded} Let $(\mathcal{A}, F_\bullet)$ be a commutative ring equipped with a positive filtration so that $\operatorname{gr}_F(\mathcal{A})$ is Gorenstein and Du Bois then:  \begin{enumerate} \item $\mathcal{A}$ is Gorenstein and Du Bois as well \item Suppose further that $\mathcal{A}$ is finitely generated in degree one. Then $\operatorname{Proj}(\mathcal{R}(\mathcal{A}))$ is Gorenstein and Du Bois. \end{enumerate}  \end{lem} 
\begin{proof} \emph{Claim (1):} Because $\pi: Y_\mathcal{R} \to \mathbb{A}^1$ is flat, \cite[Lemma 47.21.8]{stacksproject} implies that irrelevant ideal is Gorenstein. Because the $\mathbb{G}_m$ action is contracting, all of $Y_R$ is Gorenstein, which means that the other fibers of $\pi$ are Gorenstein as well (again making use of the same result). For the Du Bois property, \cite[Theorem 5.12]{Datta} and \cite[Theorem 4.1]{schwedekovacs} imply that $Y_R$ is Du Bois and we conclude using \cite[Theorem 2.3]{schwedekovacs}. 
\vskip 5 pt 
\emph{Claim (2):} For the second claim, first note that if $\mathcal{A}$ is generated in degree one, so is $\mathcal{R}(\mathcal{A})$. for any $f$ in degree one, the degree zero piece of $\mathcal{R}(\mathcal{A})_f$, $\mathcal{R}(\mathcal{A})_{(f)}$, is Gorenstein (respectively Du Bois). We have that $\mathcal{R}(\mathcal{A})_f \cong \mathcal{R}(\mathcal{A})_{(f)}[f,f^{-1}]$. The Gorenstein property of $\mathcal{R}(\mathcal{A})_{(f)}$ follows from  \cite[Lemma 47.21.8]{stacksproject} and the Du Bois follows from \cite{MR1761628} (note that $\mathcal{R}(\mathcal{A})_{(f)}$ is a retract of $\mathcal{R}(\mathcal{A})_f$).  Because $\mathcal{R}(\mathcal{A})$ is generated in degree one, it follows that $\operatorname{Proj}(\mathcal{R}(\mathcal{A}))$ can be covered by $\operatorname{Spec}(\mathcal{R}(\mathcal{A})_{(f)})$ with $f$ in degree one and hence is Gorenstein and Du Bois as well.  \end{proof} \vskip 5 pt

\section{An intrinsic family with no smooth fiber} \label{section:appendixC}

Here we give an example where the Gross-Siebert family has no smooth fiber. 
Let $M$ be a $(1,1)$ hypersurface in $\mathbb{P}^2 \times \mathbb{P}^2.$  The divisor $\D$ will have three components. They are $D_1$, a $(0,1)$ hypersurface intersected with $M$, $D_2$ a $(1,0)$ hypersurface intersected with $M$ and $D_3$ a $(1,1)$ hyperplane section.  Let $\AL$ denote the ring of theta functions over $\Lambda=\mathbb{C}[H_2(M)]$.  Due to positivity of the boundary,  this ring has a positive ascending filtration and we let $\mathcal{SR}_\Lambda(M,\D)$ denote the associated graded ring.  For a given $\theta_\v$ function, we let $\bar{\theta}_\v$ denote the corresponding function in the associated graded.

 Each pairwise intersection is connected and the triple intersection of all of the divisors is two points. It follows that $\mathcal{SR}_\Lambda(M,\D)$ is thus generated by $\bar{\theta}_{(1,0,0)}=x_1$, $\bar{\theta}_{(0,1,0)}=x_2$, $\bar{\theta}_{(0,0,1)}=x_3$, $\bar{\theta}_{(1,1,1),1}=u, \bar{\theta}_{(1,1,1),2}=v$. It follows from positivity of the filtration on $\AL$ that  $\theta_{(1,0,0)}, \theta_{(0,1,0)}, \theta_{(0,0,1)}, \theta_{(1,1,1),1}, \theta_{(1,1,1),2}$ generate $\AL$. Returning to the associated graded, we have that $\mathcal{SR}_\Lambda(M,\D)$ is the free ring over $\Lambda$ on these 5 variables modulo the relations:
\begin{align} x_1x_2x_3=u+v \\ uv=0 \end{align} 

Note that this can be viewed as a hypersurface singularity by setting e.g. $v=x_1x_2x_3-u$ to obtain 
\begin{align} \label{eq:hyp} u(x_1x_2x_3-u)=0 \end{align} The Jacobian ring of this singularity is of course infinite dimensional.  However, the next two general observations significantly restrict the form of the deformation to a finite dimensional moduli space.  \vskip 5 pt

\textbf{Observation 1}: There is a filtration on $\AL$ coming from the fact that any curve has to intersect the divisor $D_3$ positively.  It sets $\deg(\theta_{(1,0,0)})=0$, $\deg(\theta_{(0,1,0)})=0$, $\deg(\theta_{(0,0,1)})=1$, $\deg(\theta_{(1,1,1),i})=1$. The associated graded ring with respect to this filtration is the Stanley Reisner ring. There are also two other filtrations coming from the divisors $D_1$ and $D_2$, given by setting the degrees $\deg(\theta_{(1,0,0)})=1$, $\deg(\theta_{(0,1,0)})=0$, $\deg(\theta_{(0,0,1)})=0$, $\deg(\theta_{(1,1,1),i})=1$ (and similarly for $D_2$).  Because the divisors $D_1$ and $D_2$ are only NEF and not positive, the associated graded will not be the Stanley Reisner ring however this will still be useful. 

 \vskip 5 pt

\textbf{Observation 2}: The deformation must be $\mathbf{G}_m-equivariant$ with respect to the $\mathbf{G}_m$ that acts with weights $w_{\mathbb{C}^*}(\theta_{(1,0,0)})=1$, $w_{\mathbb{C}^*}(\theta_{(0,1,0)})=1$, $w_{\mathbb{C}^*}(\theta_{(0,0,1)})=-1$, $w_{\mathbb{C}^*}(\theta_{(1,1,1),i})=1$. \vskip 5 pt

 We have that by \eqref{eq:hyp}, \begin{align} \label{eq:gentheta} \theta_{(1,1,1),1}(\theta_{(1,0,0)}\theta_{(0,1,0)}\theta_{(0,0,1)}-\theta_{(1,1,1),1})= g  \end{align} 
for some $g \in F_1 \AL$. 
\begin{lem} \label{lem:whocares1} The most general form \label{eq:gentheta} consistent with the above observations has $g=g_{\vec{a}}$ \begin{align*} \label{eq: x1x2} g_{\vec{a}}=a_1\theta_{(1,0,0)}\theta_{(0,1,0)}+ a_2(\theta_{(1,0,0)}^2)+a_3(\theta_{(0,1,0)}^2) +\\ (a_4\theta_{(1,0,0)}^2 \theta_{(0,1,0)}+a_5 \theta_{(0,1,0)}^2\theta_{(1,0,0)}) \theta_{(0,0,1)} + (a_6 \theta_{(1,0,0)}+ a_7 \theta_{(0,1,0)}) \theta_{(1,1,1),1}  \end{align*} \end{lem}
\begin{proof} The function $g$ must have degree at most one (in the filtration coming from $D_3$) which means at most one power of $\theta_{(1,1,1),1}$ or $\theta_{(0,0,1)}$. Also with respect to $D_1$ and $D_2$ the degree is at most 2, meaning the combined exponents of $\theta_{(0,1,0)}$ or $\theta_{(1,0,0)}$ is at most two. The rest follows from the fact that $w_{\mathbb{C}^*}(g)=2$($g$ is homogeneous of weight $2$). \end{proof}

\begin{lem} For any $\vec{a} \in \Lambda^7,$ set $f_{\vec{a}}=\theta_{(1,1,1),1}(\theta_{(1,0,0)}\theta_{(0,1,0)}\theta_{(0,0,1)}-\theta_{(1,1,1),1})- g_{\vec{a}}$.  We have an isomorphism \begin{align} \Lambda[\theta_{(1,0,0)},\theta_{(0,1,0)},\theta_{(0,0,1)},\theta_{(1,1,1),1} ]/(f_{\vec{a}}) \cong \AL \end{align} for some choice of $\vec{a}.$
\end{lem} 
\begin{proof} By Lemma \ref{lem:whocares1},  we have a map of filtered rings $$\Lambda[\theta_{(1,0,0)},\theta_{(0,1,0)},\theta_{(0,0,1)},\theta_{(1,1,1),1}]/(f_{\vec{a}}) \to \AL$$ for some $\vec{a} \in \Lambda^{7}$.  Moreover this map is surjective and an isomorphism on associated graded.  If we let $I$ denote the kernel, then $\bar{I}$ vanishes in the associated graded.  Because the filtration is positive and ascending it induces the discrete topology on $I$. Hence $I=0$. \end{proof} 

\begin{cor} \label{lem:whocares2} The mirror family contains no member which is smooth. \end{cor}
\begin{proof} To prove this,  observe that by the Jacobian criterion each member of the mirror family has an $\mathbb{A}^1 $ worth of singularities where $\theta_{(1,1,1),1}=\theta_{(1,0,0)}=\theta_{(0,1,0)}=0, \theta_{(0,0,1)}=c$. \end{proof}

\begin{rem} \begin{itemize} \item It should not be difficult to calculate the exact coefficients of this mirror family. We expect that the singular locus of the fiber over the augmentation ideal in $\operatorname{Spec}(\Lambda)$ should be modelled on the cDV singularity $uv=xy^2$.  \item In the symplectic setting,  we do not as yet know how to construct the extra filtrations coming from NEF divisors $D_1$,  $D_2$ in Observation 1.  Thus we cannot presently prove the analogue of Corollary \ref{lem:whocares2} for $SH^0(X,
\underline{\Lambda}).$  \end{itemize} \end{rem}

\bibliography{shbib}
\bibliographystyle{apalike} 

\end{document}